\documentclass[10.9pt,a4paper]{amsart}
\usepackage{a4wide}
\usepackage[utf8]{inputenc}
\usepackage[english]{babel}

\usepackage{amsfonts,amssymb,amsmath}

\usepackage[nobysame,alphabetic, initials]{amsrefs}

\usepackage{comment}

\usepackage{graphicx}

\usepackage[dvipsnames]{xcolor}
\usepackage[colorlinks,
    linkcolor={blue!50!black},
    citecolor={blue!50!black},
    urlcolor={black!80!black}]{hyperref}
 \usepackage{enumerate}

\usepackage{tikz}
\usepackage[all]{xy}
\usetikzlibrary{calc, matrix, arrows, cd, decorations.markings, knots}
\usetikzlibrary{decorations.pathmorphing}
\usetikzlibrary{decorations.pathreplacing}

\newtheorem{innercustomthm}{Theorem}
\newenvironment{customthm}[1]
  {\renewcommand\theinnercustomthm{#1}\innercustomthm}
  {\endinnercustomthm}

\newtheorem{theorem}{Theorem}[section]

\newtheorem{corollary}[theorem]{Corollary}
\newtheorem{lemma}[theorem]{Lemma}
\newtheorem{proposition}[theorem]{Proposition}

\newtheorem{question}[theorem]{Question}

\newtheorem{theoremalpha}{Theorem}
\newtheorem{corollaryalpha}{Corollary}

\theoremstyle{definition}
\newtheorem{definition}[theorem]{Definition}

\newtheorem{notation}{Notation}
\newtheorem{example}[theorem]{Example}

\newtheorem{remark}[theorem]{Remark}
\newtheorem{construction}[theorem]{Construction}

\newtheorem{claim}{Claim}
\newtheorem*{claim*}{Claim}

\newcommand{\s}{\mathfrak{s}}
\newcommand{\N}{\mathbb{N}}
\newcommand{\Z}{\mathbb{Z}}
\newcommand{\Q}{\mathbb{Q}}
\newcommand{\R}{\mathbb{R}}
\newcommand{\C}{\mathbb{C}}

\newcommand{\E}{\mathcal{E}}
\newcommand{\F}{\mathcal{F}}
\newcommand{\G}{\mathcal{G}}
\newcommand{\zs}{\operatorname{zs}}
\newcommand{\pr}{\operatorname{pr}}
\newcommand{\id}{\operatorname{Id}}
\newcommand{\sign}{\sigma}
\newcommand{\Hom}{\operatorname{Hom}}
\newcommand{\bAut}{\operatorname{bAut}}
\newcommand{\Aut}{\operatorname{Aut}}

\newcommand{\coker}{\operatorname{coker}}
\newcommand{\nd}{\operatorname{nd}}
\newcommand{\ev}{\operatorname{ev}}
\newcommand{\PD}{\operatorname{PD}}

\newcommand{\im}{\operatorname{Im}}

\newcommand{\Id}{\operatorname{Id}}
\newcommand{\Pin}{\operatorname{Pin}}
\def\ol{\overline}
\def\wt{\widetilde}

\newcommand{\BTOPSpin}{\operatorname{BTOPSpin}}

\newcommand{\Spin}{\operatorname{Spin}}
\newcommand{\TOPSpin}{\operatorname{TOPSpin}}
\newcommand{\STOP}{\operatorname{STOP}}
\newcommand{\BSTOP}{\operatorname{BSTOP}}

\newcommand{\SO}{\operatorname{SO}}

\newcommand{\bsm}{\left(\begin{smallmatrix}}
\newcommand{\esm}{\end{smallmatrix}\right)}
\newcommand{\sm}{\setminus}

\newcommand{\smfrac}[2]{\mbox{\footnotesize$\displaystyle\frac{#1}{#2}$}} 
\newcommand{\tmfrac}[2]{\mbox{\large$\frac{#1}{#2}$}} 

\definecolor{bettergreen}{rgb}{0.0, 0.5, 0.0}

\begin{document}
\title{Unknotting nonorientable surfaces}
\author[A.~Conway]{Anthony Conway}
\address{The University of Texas at Austin, Austin TX 78712}
\email{anthony.conway@austin.utexas.edu}

\author[P.~Orson]{Patrick Orson}
\address{California Polytechnic State University, San Luis Obispo, CA, United States}
\email{patrickorson@gmail.com}

\author[M.~Powell]{Mark Powell}
\address{School of Mathematics and Statistics, University of Glasgow, United Kingdom}
\email{mark.powell@glasgow.ac.uk}

\def\subjclassname{\textup{2020} Mathematics Subject Classification}
\expandafter\let\csname subjclassname@1991\endcsname=\subjclassname
\subjclass{
57K40, 
57N35, 
57R67. 
}

\begin{abstract}
Given a nonorientable, locally flatly embedded surface in the $4$-sphere of nonorientable genus $h$, Massey showed that the normal Euler number lies in $\lbrace -2h,-2h+4,\ldots,2h-4,2h \rbrace$.
We prove that every such surface with knot group of order two is topologically unknotted, provided that the normal Euler number is not one of the extremal values in Massey's range. When $h$ is~$1$,~$2$, or~$3$, the same holds even with extremal normal Euler number; the~$h=1$ case is due to Lawson.

We also study nonorientable embedded surfaces in the 4-ball with boundary a knot $K$ in the 3-sphere, again where the surface complement has fundamental group of order two and nonorientable genus $h$.
We prove that any two such surfaces with the same normal Euler number become topologically isotopic, rel.\ boundary, after adding a single tube to each.
If the determinant of $K$ is trivial, we show that any two such surfaces are isotopic, rel.\ boundary, again provided that they have non-extremal normal Euler number, or that $h$ is~$1$,~$2$, or~$3$.
 \end{abstract}

\maketitle

\section{Introduction}
\label{sec:Introduction}

There is an involution of $\C P^2$ given by complex conjugation $[z_0:z_1:z_2] \mapsto [\ol{z}_0:\ol{z}_1:\ol{z}_2]$. Remarkably, the quotient of $\C P^2$ by this involution is the 4-sphere~\cites{MR341511, Kuiper, MR971226}.
The fixed point set of the involution is the standard $\R P^2$ in $\C P^2$, and the image of this fixed set under the quotient map is called the \emph{standard positive embedding} of $\R P^2$ in $S^4$. The image of the standard positive embedding under a reflection of $S^4$ is the \emph{standard negative embedding}.   The embeddings are distinguished by their normal Euler number  (Definition~\ref{def:Eulernumber}), which is $-2$ for the positive embedding and $2$ for the negative one.
Additional details on the standard embeddings~$\R P^2\subseteq S^4$  were given by Lawson~\cite{LawsonSplitting,Lawson}; see also Example~\ref{example:standard embedding} for an alternative construction.

The \emph{nonorientable genus} of a compact, nonorientable surface $F$ with at most one boundary component is the dimension of the $\Z_2$-vector space $H_1(F;\Z_2)$.
Taking connected sums of $h$ standard embeddings of~$\R P^2$ realises every normal Euler number in the range $\lbrace -2h,-2h+4,\ldots,2h-4,2h \rbrace$. In fact, Massey~\cite{Massey} showed that for every locally flat
 embedding $F \subseteq S^4$ of a closed, nonorientable surface of nonorientable genus~$h$, the normal Euler number $e(F)$ must lie in this range.
Such a surface is said to be topologically \emph{unknotted} if it is topologically isotopic to a connected sum of standardly embedded projective planes.

Every unknotted surface $F\subseteq S^4$ has knot group $\pi_1(S^4 \sm F) \cong \Z_2$.
By definition a \emph{$\Z_2$-surface} in $S^4$ will refer to a locally flatly embedded surface $F$, as above, with knot group~$\Z_2$.
We consider the following question, which already appeared, for example, in~\cite[p.~181]{KawauchiSurvey},~\cite[p.~55]{KamadaBook}, and~\cite[p.~7]{Matrix}.

\begin{question}
\label{quest:UnknottingConjecture}
Is every $\Z_2$-surface in $S^4$ topologically unknotted?
\end{question}

Lawson~\cite{Lawson} proved that~$\Z_2$-projective planes are indeed unknotted.
Finashin, Kreck, and Viro~\cite{FinashinKreckViro} proved that once the nonorientable genus $h$ and normal Euler number $e$ are fixed, there are only finitely many topological isotopy types of smoothly embedded $\Z_2$-surfaces.
Kreck~\cite{KreckOnTheHomeomorphism} showed that if a smoothly embedded $\Z_2$-surface has $(h,e) \in \{(10,\pm 16),(2,0)\}$, then it is topologically unknotted.
These are all the prior instances of progress on Question~\ref{quest:UnknottingConjecture}.

We say $e(F)$ is \emph{extremal} if $|e(F)|=2h$ and \emph{non-extremal} otherwise.
Here is our main result.

\begin{theoremalpha}
\label{thm:TopUnknottingIntro}
Let $F \subseteq S^4$ be a $\Z_2$-surface of nonorientable genus $h$ and normal Euler number~$e$.
If $h \leq 3$ or if $e$ is non-extremal, then $F$ is unknotted.
\end{theoremalpha}

We know of no example where the answer to Question~\ref{quest:UnknottingConjecture} is negative.
After this article appeared as a preprint,  Pencovitch generalised the calculations from our appendix, and in so doing proved that Theorem~\ref{thm:TopUnknottingIntro} also holds when $e$ is extremal and $h=4,5$~\cite{Pencovitch}.
 The extremal cases for~$h \geq 6$ remain as interesting open questions.
In fact, Pencovitch showed that the last step in our proof does not go through when $e$ is extremal and~$h=6,7$, so a new proof method would be needed for these cases.

\begin{remark}
\label{rem:smooth}
One can ask a version of Question~\ref{quest:UnknottingConjecture} where one considers smoothly embedded surfaces up to smooth isotopy. In terms of positive answers to such a question, Bleiler and Scharlemann~\cite{BleilerScharlemann} showed that if a smoothly embedded~$\R P^2 \subseteq S^4$ has three critical points, then it is smoothly unknotted.

However, in general, this fully smooth version of Question~\ref{quest:UnknottingConjecture} is known to have a negative answer.
Finashin, Kreck, and Viro produced an infinite family of smoothly distinct, smoothly embedded $\Z_2$-surfaces in $S^4$ with nonorientable genus $h=10$ and normal Euler number~$e=\pm 16$~\cite{FinashinKreckViro}.
They achieved this by combining annulus surgery with Donaldson's theorem to construct infinitely many $\Z_2$-surfaces $\lbrace F_i\rbrace_{i \geq 1}$ with non-diffeomorphic double branched covers.
Kreck later proved that all the $F_i$
are topologically unknotted~\cite{KreckOnTheHomeomorphism}.
Finashin used a similar strategy to produce an infinite family of smoothly embedded, pairwise non-diffeomorphic $\Z_2$-surfaces with $(h,e)=(k,\pm (2k-4))$, for every $k \geq 6$~\cite[Theorem~B]{Finashin}.
More recently, for each $k=7,8,9,10$, Havens produced analogous smooth infinite families of irreducible $\Z_2$-surfaces~\cite[Theorem 1.3.1]{HavensThesis}.
Theorem~\ref{thm:TopUnknottingIntro} implies that Finashin's and Havens' examples are all topologically unknotted.
After this paper appeared
as a preprint,
Miyazawa~\cite{Miyazawa} posted a preprint proving the existence of a smooth infinite family of pairwise not smoothly isotopic $\Z_2$-projective planes (relying on Lawson's result~\cite{Lawson} for topological isotopy),  and Mati\'c-\"Ozt\"urk-Reyes-Stipsicz-Urz\'ua used Theorem~\ref{thm:TopUnknottingIntro} to give an example of a pair of topologically isotopic,  but non-diffeomorphic smoothly embedded $\Z_2$-surfaces with $(h,e)=(5,6)$~\cite{MaticOzturkReyesStipsiczUrzua}.
\end{remark}

\subsection{Surfaces with boundary}
\label{sub:WithBoundaryIntro}

Theorem~\ref{thm:TopUnknottingIntro} is deduced as a consequence of a uniqueness result for properly, locally flatly embedded, compact, nonorientable surfaces $F \subseteq D^4$ with boundary a fixed knot in~$S^3$.
Such an $F$ is called a \emph{$\Z_2$-surface} in $D^4$ if it has knot group $\pi_1(D^4 \sm F) \cong \Z_2$.

In order to state our main result in this setting, we start by recalling a fact about the normal Euler number of nonorientable surfaces in $D^4$.
First, if $F \subseteq D^4$ is a nonorientable surface, then as shown in \cite[Theorem 2 and Corollary 5]{GordonLitherland} and \cite[Theorem 6]{GilmerLivingston}, we have the equality $\sigma(K)=\sign(\Sigma_2(F))+\smfrac{1}{2}e(F)$, where $\Sigma_2(F)$ denotes the double branched cover of $F$ and $\sigma(K)$ denotes the signature of~$K$.
If~$F$ is a~$\Z_2$-surface of nonorientable genus $h$, then Yasuhara~\cite[Corollary~1.1.1]{YasuharaConnecting} showed that
\[ e(F) \in \lbrace -2h+2\sigma(K),-2h+4+2\sigma(K),\ldots, 2h-4+2\sigma(K),2h+2\sigma(K) \rbrace.\]
We say that $e(F)$ is \emph{extremal} if $|e(F)-2\sigma(K)|=2h$.

\begin{theoremalpha}
\label{thm:SurfacesWithBoundaryIntro}
Let~$F_0,F_1\subseteq D^4$ be
$\Z_2$-surfaces of the same nonorientable genus $h$, the same normal Euler number $e$ and the same boundary $K$.  Assume that~$|\det(K)|=1$.
If either~$h \leq 3$ or~$e$ is non-extremal,  then~$F_0$ and~$F_1$ are ambiently isotopic rel.\ boundary.
\end{theoremalpha}

We deduce Theorem~\ref{thm:TopUnknottingIntro} from  Theorem~\ref{thm:SurfacesWithBoundaryIntro} by puncturing the closed surfaces in question i.e.\ a given closed $\Z_2$-surface and the standard unknotted surface with the same nonorientable genus and  normal Euler number,  and applying Theorem~\ref{thm:SurfacesWithBoundaryIntro} with $K$ the unknot.
Theorem~\ref{thm:SurfacesWithBoundaryIntro} is proved in Section~\ref{sec:introproofs}.

As we explain in Section~\ref{sub:Outline}, our proof of Theorem~\ref{thm:SurfacesWithBoundaryIntro} uses Kreck's modified surgery~\cite{KreckSurgeryAndDuality}, and owes an intellectual debt to Kreck's article~\cite{KreckOnTheHomeomorphism}, while going further than that article.

\begin{remark}
Compared to the work in the smooth category discussed in Remark~\ref{rem:smooth}, there are fewer results for smooth surfaces with nonempty boundary. Gompf~\cite{GompfNuclei} exhibited  a smooth infinite family of pairwise non smoothly isotopic punctured Klein bottles in $D^4$ with knot group $\Z_2$ and Euler number $0$ (and appealed to Kreck~\cite{KreckOnTheHomeomorphism} to prove that these $\Z_2$-surfaces are topologically isotopic) and, as we will discuss in more detail in Section~\ref{sub:BandMoves} below,  Lipshitz-Sarkar~\cite{LipshitzSarkar} recently also found examples of exotic surfaces in the $4$-ball with $h=3$.
\end{remark}

\subsection{One internal stabilisation is enough}

Let~$F_0,F_1\subseteq D^4$ be~$\Z_2$-surfaces of the same nonorientable genus~$h$, the same normal Euler number~$e$, and the same boundary~$K$.
Baykur and Sunukjian~\cite[Theorem~1]{BaykurSunukjian} proved that $F_0$ and $F_1$ become isotopic after some finite, but \textit{a priori} arbitrarily large, number of internal stabilisations. Here, an \emph{internal stabilisation} of a connected, nonorientable surface $F\subseteq M$ in a 4-manifold  is obtained by adding a tube. That is, performing ambient surgery on $F$, along an embedded copy of $D^1 \times D^2 \subseteq M$ with $S^0 \times D^2 \subseteq F$ and~$(\mathring{D}^1 \times D^2) \cap F = \emptyset$. This increases the nonorientable genus by two.  The fact that homotopy implies isotopy for 1-complexes in 4-manifolds means that, for $F$ a $\Z_2$-surface,  every tube choice for internal stabilisation is a trivial tube. Therefore adding a tube is the same as taking the connected sum~$(D^4,F)\#(S^4,T^2)$, where $T^2 \subseteq S^4$ is the standard unknotted torus. Equivalently, one could take the connected sum with an unknotted Klein bottle in $S^4$ that has trivial normal Euler number; see e.g.~\cite[Lemma 4]{BaykurSunukjian}.

We refine Baykur and Sunukjian's result, in the case of $\Z_2$-surfaces, to show that in fact one internal stabilisation suffices.

\begin{theoremalpha}\label{thm:one-is-enough}
Let~$F_0,F_1\subseteq D^4$ be~$\Z_2$-surfaces of the same nonorientable genus~$h$, the same normal Euler number~$e$, and the same boundary~$K$.
Then~$F_0 \# T^2$ and~$F_1 \# T^2$ are ambiently isotopic rel.\ boundary.
\end{theoremalpha}

Theorem~\ref{thm:one-is-enough} is proved in Section~\ref{sec:introproofs}.
Note that stabilising in this way automatically puts us into the case of non-extremal normal Euler number, so the outcome of a single stabilisation when $F$ is closed, or when $\partial F=K$ with $|\det{K}|=1$, is already covered by Theorems~\ref{thm:TopUnknottingIntro} and~\ref{thm:SurfacesWithBoundaryIntro} respectively.
On the other hand,  in Theorem~\ref{thm:one-is-enough} there is no assumption on $\det(K)$.
Theorem~\ref{thm:one-is-enough} can also be compared with~\cite{AucklySadykov}, which considers a type of external stabilisation.

\subsection{Involutions of connected sums of projective planes}
\label{sub:Actions}

We describe an application of our results to the study of involutions.
The result of Massey, Kuiper, and Arnol'd~\cite{MR341511, Kuiper, MR971226}, mentioned at the start of the introduction, can be applied to each of the connected summands of $\#^a \C P^2 \#^b \overline{\C P^2}$, to show that the conjugation-induced action of $\Z_2$ on $\#^a \C P^2 \#^b \overline{\C P^2}$ has orbit set~$S^4$ and fixed set the unknotted~$\Z_2$-surface with nonorientable genus $h$ and normal Euler number~$e$, where $a+b=h$ and $2(a-b)=-e$.
We refer to this action as the \emph{standard involution.}
We apply Theorem~\ref{thm:TopUnknottingIntro} to obtain a criterion for an involution on $\#^a \C P^2 \#^b \overline{\C P^2}$ to be topologically conjugate to the standard involution.

\setcounter{corollaryalpha}{3}
\begin{corollaryalpha}
\label{cor:Action}
Let $\tau$ be an involution on $Z :=\#^a \C P^2 \#^b \ol{\C P}^2$ that acts locally linearly and semifreely with fixed point set a nonorientable\footnote{
In fact,  as~$Z$ is nonspin,  a result of Nagami~\cite[Theorem 1.1]{Nagami} implies that $\widetilde{F}$ is necessarily nonorientable.
}
 surface $\widetilde{F} \subseteq Z$ of nonorientable genus $a+b$.
Assume that~$Z \setminus \widetilde{F}$ is simply-connected and that the pair $(h,e):=(a+b,2(b-a))$ satisfies the conditions of Theorem~\ref{thm:TopUnknottingIntro}.
Then the involution $\tau$ is topologically conjugate to the standard involution on $Z$.
\end{corollaryalpha}

\begin{proof}
For brevity,
we write $Z_{\widetilde{F}}$ for the exterior of $\widetilde{F} \subseteq Z$, $X:=Z/\tau$ for the orbit space of the involution, $p \colon Z \to X$ for the quotient map and $F:=p(\widetilde{F}) \subseteq X$ for the branch locus.
We also write $X_F:=X \setminus \nu F $ so that, since the action is semifree,  the restriction $p| \colon Z_{\widetilde{F}} \to X_F$ is a $2$-fold covering map.
Note that $F$ and $\widetilde{F}$ both have nonorientable genus $h$.

We claim that $X$ is homeomorphic to $S^4$.
Since $Z_{\widetilde{F}}$ is simply-connected and $p| \colon Z_{\widetilde{F}} \to X_F$ is a~$2$-fold covering map, the long exact sequence of a fibration implies that $\pi_1(X_F)=\Z_2$.
A Seifert-van Kampen argument then shows that
$\pi_1(X)=\pi_1(X_F)/\langle \mu_F \rangle$ is a quotient of $\Z_2$, where~$\mu_F \subseteq X_F$ denotes a meridian to $F \subseteq X$.
Since the Euler characteristic of a closed $3$-manifold vanishes,  a calculation using the additivity of the Euler characteristic yields
\begin{align*}
2\chi(X)
&=2(\chi(X_F)+\chi(F))
=\chi(Z_{\widetilde{F}})+2\chi(F) \\
&=(\chi(Z)-\chi(\widetilde{F}))+2\chi(\widetilde{F})
=\chi(Z)+\chi(\widetilde{F})
=(2+h)+(2-h)
=4.
\end{align*}
We argue that $X$ is simply-connected, i.e.~that $\mu_F$ is nontrivial in $\pi_1(X_F)$.
If $\mu_F$ were trivial in~$\pi_1(X_F)$, it would bound an immersed disc, which when glued to the meridional disc for $F$,  would produce an immersed sphere in
$X$ that is geometrically dual to $F$ in $X$. This sphere and~$F$ then represent independent and homologically essential $\Z_2$-coefficient homology classes.
This would imply, computing with $\Z_2$ coefficients, that $b_2^{\Z_2}(X)\geq 2$ and therefore~$\chi(X) \geq 1+2+1 = 4$, contradicting the calculation above.

Since~$X$ is a closed simply-connected $4$-manifold with $\chi(X)=2$, the 4-dimensional Poincar\'e conjecture implies that $X \cong S^4$~\cite{Freedman}.
This establishes the claim.

The claim implies that $Z$ is homeomorphic to the 2-fold cover of $S^4$ branched along the $\Z_2$-surface $F \subseteq S^4$.
Thus,~\cite[Theorem 1]{GeskeKjuchukovaShaneson} implies that the normal Euler number of $F$ is
$$e(F)=-2\sigma(\Sigma_2(F))=-2\sigma(Z)=-2(a-b)=e.$$
We apply Theorem~\ref{thm:TopUnknottingIntro} to deduce that there a homeomorphism of pairs~$\varphi \colon (S^4,F) \cong (S^4,U_{h,e})$, where $U_{h,e}$ is the unknotted embedding with normal Euler number $e$ of the closed nonorientable surface of genus $h$.
This homeomorphism can be lifted to a~$\Z_2$-equivariant homeomorphism of the universal covers of the exteriors and extended to
a~$\Z_2$-equivariant homeomorphism~$\widetilde{\varphi}$ of the branched covers.
In other words,~$\widetilde{\varphi}$ conjugates the~$\Z_2$-action on~$\Sigma_2(F)$ with the standard one on~$\Sigma_2(U_{h,e})  \cong \#^ a \C P^2 \#^b \overline{\C P}^2$.
\end{proof}

\subsection{Surfaces obtained by band moves}
\label{sub:BandMoves}

Given a knot~$K$, we describe a natural source of~$\Z_2$-surfaces in the~$4$-ball with boundary~$K$.
Assume that~$h$ band moves, at least one of which is nonorientable,  transform~$K$ into a knot~$J$ that bounds a genus $g$ orientable surface~$\Sigma\subseteq D^4$ whose exterior~$N_{\Sigma}=D^4 \setminus \nu \Sigma$ has fundamental group~$\Z$; we call~$\Sigma$ a~\emph{$\Z$-surface.}
We obtain a nonorientable~$\Z_2$-surface~$F \subseteq D^4$ of nonorientable genus~$h+2g$
by using  the~$\Z$-surface~$\Sigma$ to cap off the nonorientable cobordism~$C \subseteq S^3 \times [0,1]$ arising from the band moves.
To see that this is a~$\Z_2$-surface use the decomposition, $X_F=X_C \cup_{X_J} -N_\Sigma$, where~$X_C$ denotes the exterior of~$C$ and~$X_J$ denotes the exterior of~$J$.
Since~$\pi_1(X_J) \to \pi_1(X_C)$ is surjective,~$\pi_1(X_F)$ must be a quotient of~$\pi_1(N_\Sigma) \cong \Z$. Then since~$F$ is nonorientable, this forces~$\pi_1(X_F) \cong \Z_2$.

Lipshitz and Sarkar recently used this construction to produce examples of exotic nonorientable surfaces $F_0,F_1 \subseteq D^4$~\cite{LipshitzSarkar}.
In one of their examples,  they perform 3 band moves on the knot $K=12^n_{309}$, to arrive at a specific knot $J\subseteq S^3$ that Hayden-Sundberg~\cite{HaydenSundberg} had shown bounds two smoothly distinct $\Z$-discs.
Lipshitz and Sarkar showed that the construction above results in two smoothly embedded $\Z_2$-surfaces of nonorientable genus $h=3$ and
Euler number $-6$, that are not smoothly isotopic.
Since~$\det(K)=1$ and $h=3$, Theorem~\ref{thm:SurfacesWithBoundaryIntro} shows that these two surfaces are topologically ambiently isotopic rel.~boundary.

An alternative route to this conclusion  was already available, using the fact that any two $\Z$-discs with the same boundary are ambiently isotopic~\cite{ConwayPowellDiscs}.
In future applications, one might wish to cap off by a $\Z$-surface of higher genus than a slice disc. While isotopy uniqueness of such surfaces was also considered previously by two of the authors~\cite{ConwayPowell}, we suggest that Theorem~\ref{thm:SurfacesWithBoundaryIntro} will likely be more easily applicable here, as it avoids some of the more delicate issues involving intersection pairings from that study.

\subsection{Outline of the proof of Theorem~\ref{thm:SurfacesWithBoundaryIntro}}
\label{sub:Outline}
Our proof of Theorem~\ref{thm:SurfacesWithBoundaryIntro} uses Kreck's modified surgery theory~\cite{KreckSurgeryAndDuality}, and as mentioned above our approach was inspired by~\cite{KreckOnTheHomeomorphism}.
We review modified surgery theory in Sections~\ref{sec:Monoid} and~\ref{sec:ModifiedSurgery},  but for this introduction we introduce some aspects of this theory and explain how it applies (in principle) to proving that two spin~$4$-manifolds~$M_0$ and~$M_1$, with fundamental group $\Z_2$ and (possibly empty) boundary, are homeomorphic.
While we delay detailed definitions to Sections~\ref{sec:Monoid} and~\ref{sec:ModifiedSurgery},  we hope that the broad picture will give the reader some idea of what our proof of Theorem~\ref{thm:SurfacesWithBoundaryIntro} entails.

The choice of a spin structure on the~$M_i$ and of an identification~$\pi_1(M_i) \cong \Z_2$ determine~$2$-connected maps $\overline{\nu}_i \colon M_i \to B$ where the~$B:=\BTOPSpin \times B\Z_2$.
There is a fibration~$\xi \colon B \to \BSTOP$, and in Kreck's terminology $(B,\xi)$ is called the \emph{normal~$1$-type} of the $M_i$, while the~$\overline{\nu}_i$ are called \emph{normal $1$-smoothings} if they lift the stable normal bundles~$M_i \to \BSTOP$ of the~$M_i$.

In order to show that $M_0$ and $M_1$ are homeomorphic, the first step in modified surgery is to fix a homeomorphism~$f \colon \partial M_0 \to \partial M_1$ that preserves the normal $1$-smoothings, then find a $5$-dimensional cobordism~$(W;M_0,M_1)$ rel.~boundary together with a normal $1$-smoothing~$\overline{\nu} \colon W \to B$ that extends the $\overline{\nu}_i$.
The bordism $W$ can be obtained by proving that $(M_0 \cup_f -M_1,\overline{\nu}_0 \cup - \overline{\nu}_1)$ vanishes in the bordism group~$\Omega_4(B,\xi)$.

To such a $(B,\xi)$-cobordism~$(W,\overline{\nu})$, modified surgery associates an obstruction $\Theta(W,\overline{\nu})$ in a monoid~$\ell_5(\Z[\Z_2])$.
If this obstruction takes a particularly simple form (i.e.~is \emph{elementary}), then~$(W,\overline{\nu})$ is $(B,\xi)$-bordant rel.\ boundary to an~$s$-cobordism.
Since $\Z_2$ is a \emph{good} group with trivial Whitehead torsion the $5$-dimensional $s$-cobordism theorem~\cite[Theorem 7.1A]{FreedmanQuinn} then implies that the homeomorphism $f \colon \partial M_0 \to \partial M_1$ extends to a homeomorphism~$M_0 \cong M_1$.
We refer to~\cite[Definition 12.12]{DETBook} for the definition of a good group and to~\cite[Chapter~19]{DETBook} for a survey of good groups, including the fact that finite groups are good.

\medbreak

Having described the modified surgery programme to decide whether two spin~$4$-manifolds with fundamental group $\Z_2$ are homeomorphic, we restrict our attention to the case where~$M_0=X_{F_0}$ and~$M_1=X_{F_1}$ are exteriors of $\Z_2$-surfaces~$F_0, F_1\subseteq D^4$.

\begin{definition}
\label{def:Extendable}
Suppose $F_0, F_1\subseteq D^4$ have the same connected boundary $K\subseteq S^3$ and write~$S(\nu F_i )=\partial \ol{\nu} F_i  \sm \nu K\subseteq D^4$ for the sphere bundle of the normal bundle to $F_i$, or equivalently the boundary of the closed tubular neighbourhood to $F_i$, minus $\nu K$.
  Suppose $f \colon S(\nu F_0 ) \to S(\nu F_1 )$ is a homeomorphism. We say $f$ is \emph{$\nu$-extendable} if it extends to a homeomorphism $\ol{\nu} F_0  \cong \ol{\nu} F_1 $ of the closed tubular neighbourhoods that sends $F_0$ to $F_1$.

Suppose the homeomorphism $f$ restricts to the identity map $S(\nu K ) \to S(\nu K )$. Then we say $f$ is \emph{$\nu$-extendable rel.~boundary} if there exists moreover an extension $\ol{\nu} F_0  \cong \ol{\nu} F_1 $ that restricts to the identity map on~$\ol{\nu} K$.
\end{definition}

A homeomorphism $S(\nu F_0) \to S(\nu F_1)$ that is $\nu$-extendable rel.\ boundary is a necessary condition for $F_0$ and $F_1$ to be ambiently isotopic rel.\ boundary.
Such a homeomorphism exists if and only if $F_0$ and $F_1$ have the same nonorientable genera and relative normal Euler numbers.
Note also that if~$f$ is~$\nu$-extendable rel.\ boundary, then it extends both to a homeomorphism $\partial X_{F_0} \to \partial X_{F_1}$, which restricts to the identity on the knot exterior,  and extends to a homeomorphism $\ol{\nu} F_0 \to \ol{\nu}F_1$.


Since, for $i=0,1$, the exterior $X_{F_i}$ is spin (it is a codimension zero submanifold of $S^4$ which is itself spin) and has fundamental group~$\Z_2$,  it comes with a normal~$1$-smoothing $\overline{\nu}_i \colon X_{F_i} \to B$.
The strategy underlying the proof of Theorem~\ref{thm:SurfacesWithBoundaryIntro} is now as follows.

\begin{enumerate}
\item We prove that~$(X_{F_0},\overline{\nu}_0)$ and~$(X_{F_1},\overline{\nu}_1)$ are $(B,\xi)$-cobordant.
The key step is in Section~\ref{sec:Spin}, where we use $\Z_4$-valued quadratic forms due to Guillou-Marin~\cite{GuillouMarin} and Kirby-Taylor~\cite{KirbyTaylor} to prove that there exists a homeomorphism $f \colon \partial X_{F_0} \to \partial X_{F_1}$ such that the closed~$4$-manifold~$M:=X_{F_0} \cup_{f} -X_{F_1}$ is spin, has $\pi_1(M) \cong \Z_2$, and such that $f$ restricts to a homeomorphism $S(\nu F_0 )\cong S(\nu F_1 )$ that is $\nu$-extendable rel.~boundary.
Here we use that~$F_0$ and $F_1$ have the same nonorientable genus $h$ and the same normal Euler number~$e$.
It is then straightforward to show that $M$ has zero signature, and an application of the Atiyah-Hirzebruch spectral sequence
to calculate $\Omega_4(B,\xi)\cong \Omega_4^{\TOPSpin}(\Z_2) \cong \Z$ then implies~$M$ is null-bordant over $(B,\xi)$.
\item We prove that the modified surgery obstruction~$\Theta(W,\overline{\nu}) \in \ell_5(\Z[\Z_2])$ is elementary.
Use~$Q_{\widetilde{X}_{F_i}}$ to denote the $\Z$-intersection form of the universal cover of $X_{F_i}$ and $Q_{\widetilde{X}_{F_i}}^{\nd}$ for the induced nondegenerate form on $\operatorname{coker} (H_2(\partial \widetilde{X}_{F_i}) \to H_2(\widetilde{X}_{F_i}))$.
The arguments from~\cite{KreckOnTheHomeomorphism}, reformulated in the framework from~\cite{CrowleySixt}, show that this obstruction lies in a subset $\ell_5(v_0,v_1) \subseteq \ell_5(\Z[\Z_2])$ whose definition involves the quadratic form~$v_i:=(\Z[\Z_2]^{h},(1-T)\theta_i)$.
Here $T$ stands for the generator of $\Z_2=\langle T \mid T^2=1 \rangle$ and $\theta_i := \theta_{\widetilde{X}_{F_i}}^{\nd}$ denotes the quadratic form on~$\operatorname{coker} (H_2(\partial \widetilde{X}_{F_i}) \to H_2(\widetilde{X}_{F_i})) \cong \Z^h$ whose underlying even form is~$Q_{\widetilde{X}_{F_i}}^{\nd}$.
If~$|\det(K)|=1$ and if either~$h \leq 8$ or~$e$ is non-extremal, then Proposition~\ref{prop:Calculate} shows that~$\theta_0 \cong \theta_1$ and therefore $v:=v_0\cong v_1$.
Using Lemma~\ref{lem:InfamousLemma}, the obstruction in~$\ell_5(v,v)$ further reduces to an obstruction in $\ell_5(v_-,v_-)$, where $v_- \cong 2\theta_i \colon \Z^h \times \Z^h \to \Z$ is obtained by restricting $v$ to the $-1$ eigenspace.
Work of Crowley-Sixt~\cite[Section 6.3]{CrowleySixt} implies that~$\ell_5(v_-,v_-)$ is isomorphic to the group $\operatorname{Aut}(\partial v_-)$ of isometries of the boundary quadratic linking form of~$v_-$ modulo the action of isometries $\Aut(v_-)$ of $v_-$.
We appeal to a theorem of Nikulin~\cite{Nikulin} to deduce that this orbit space consists of a single orbit when~$e$ is non-extremal, and therefore every element of $\ell_5(v_-,v_-)$, and hence of $\ell_5(v,v)$, has an elementary representative.
Here the key point is that when $e$ is non-extremal and $h \geq 3$,  Proposition~\ref{prop:Calculate} shows that the quadratic form $\theta_i$ splits off a hyperbolic summand (even though the $\Z[\Z_2]$-intersection form does not).
In the appendix we present detailed explicit computations which show, in the extremal cases with $h \leq 3$, that the relevant subset of the $\ell$-monoid, $\ell_5(v_-,v_-)$,  is trivial, and hence in these cases the modified surgery obstruction is also elementary.
As mentioned above, Pencovitch generalised these calculations to include $h=4,5$ and showed that~$\ell_5(v_-,v_-)$ is nontrivial in the extremal case for $h=6,7$~\cite{Pencovitch}.
\item Modified surgery theory, and more specifically~\cite[Theorem 4]{KreckSurgeryAndDuality}, then ensures that~$W$ is bordant rel.~boundary to an~$s$-cobordism and,  as mentioned above, the $s$-cobordism theorem~\cite[Theorem~7.1A]{FreedmanQuinn} then implies that the homeomorphism $f \colon \partial X_{F_0} \to \partial X_{F_1}$ extends to a homeomorphism $\Phi \colon X_{F_0} \to X_{F_1}$.
The $\nu$-extendability condition means that we can extend $\Phi$ to a rel.\ boundary homeomorphism $(D^4,F_0) \cong (D^4,F_1)$.
It follows from the Alexander trick that orientation-preserving homeomorphisms of
$D^4$ that restrict to the identity on $S^3$ are rel.\ boundary isotopic to the identity rel.~boundary. We conclude that~$F_0$ and~$F_1$ are ambiently isotopic  rel.\ boundary.
\end{enumerate}

\begin{remark}
The key ingredients in establishing that $v_0 \cong v_1$ are the fact that for~$i=0,1$, the quadratic form $v_i$ is determined by the intersection form $Q_{\Sigma_2(F_i)}$ of the $2$-fold branched cover,  and that if we let $a,b \in \Z$ be the unique integers such that $h=a+b$ and~$\sigma(Q_{\Sigma_2(F_i)})=a-b$, then the classification of nonsingular symmetric bilinear forms implies that
\begin{equation}
\label{eq:Diago}
Q_{\Sigma_2(F_i)} \cong (1)^{\oplus a} \oplus (-1)^{\oplus b},
\end{equation}
provided~$|\det(K)|=1$, and either~$h \leq 8$ or~$e$ is non-extremal.

Two questions are of interest if one is to make further progress on Question~\ref{quest:UnknottingConjecture}.
The first question is to determine whether~\eqref{eq:Diago} holds for every extremal $e$.
If $F \subseteq S^4$ is smoothly embedded, then this follows from Donaldson's theorem~\cite{Donaldson}, whereas a negative answer in the locally flat setting would lead to a negative answer to Question~\ref{quest:UnknottingConjecture}.
The second question is whether the subset~$\ell_5(v_-,v_-)$ of~$\ell_5(\Z)$ is trivial for the forms $v_-$ that arise in the extremal cases when~$h \geq 4$; see Proposition~\ref{prop:ApplyNikulin}.  As illustrated in the appendix, the complexity of the computations grows rapidly with $h$, making further computations by hand increasingly challenging.
If~\eqref{eq:Diago} holds and if~$\ell_5(v_-,v_-)$ is trivial for every~$h$, then an affirmative answer to  Question~\ref{quest:UnknottingConjecture} would follow.
\end{remark}

The first of these two questions can be made more explicit for $h=9$.
\begin{question}
\label{question:CounterExample}
Is there a $\Z_2$-surface $F \subseteq S^4$ with $(h,e)=(9,\pm 18)$ whose $2$-fold branched cover has intersection form $E_8 \oplus (1)$?
\end{question}

As explained above,  a positive answer to this question would give an example of a knotted $\Z_2$-surface in $S^4$ and therefore to a negative answer to Question~\ref{quest:UnknottingConjecture}.
As mentioned above,  Donaldson's theorem implies that the answer to Question~\ref{question:CounterExample} is negative if $F$ is smoothly embedded.

\subsection{Isotopy of embeddings}\label{subsection:isotopy-embeddings}

In our analysis thus far, we considered submanifolds of $S^4$ or $D^4$ up to ambient isotopy.  We can instead consider the closely related problem of ambient isotopy of embeddings.
Let $F$ be a nonorientable surface with zero or one boundary components, let $N$ denote $S^4$ with zero or one copies of $\mathring{D}^4$ removed respectively, and let $g_0,g_1 \colon F \hookrightarrow N$ be proper locally flat embeddings with the same normal Euler number and, in the case $\partial F \cong S^1$ and $N=D^4$, with $g_0|_{\partial F} =g_1|_{\partial F}$.  Suppose that $g_i(F)$ is a $\Z_2$-surface, for $i=0,1$.
We now consider whether there is an ambient isotopy rel.\ boundary between $g_0$ and $g_1$.

Given a proper locally flat embedding $g \colon F \hookrightarrow N$,
Guillou-Marin~\cite{GuillouMarin} defined a quadratic refinement $q_{GM} \colon H_1(g(F);\Z_2) \to \Z_4$ of the intersection pairing $Q_{g(F)} \colon H_1(g(F);\Z_2) \times H_1(g(F);\Z_2) \to \Z_2$, in terms of the embedding $g$, which we recall in Definition~\ref{def:GuillouMarin}.


\setcounter{theoremalpha}{4}
\begin{theoremalpha}\label{thm:isotopy-of-embeddings}
 Suppose that for $i=0,1$, the images $F_i := g_i(F)$
  satisfy the hypotheses of one of Theorems~\ref{thm:TopUnknottingIntro},~\ref{thm:SurfacesWithBoundaryIntro}, or~\ref{thm:one-is-enough} $($in the latter case, we assume that $g_i(F)$ are the stabilised surfaces$)$.
The embeddings~$g_0$  and~$g_1$ are ambiently isotopic rel.~boundary if and only if the homeomorphism
  \[\psi:= g_1 \circ (g_0)^{-1} \colon F_0 \to F_1\]
  induces an isometry of the Guillou-Marin forms.
\end{theoremalpha}

Theorem~\ref{thm:isotopy-of-embeddings} is proved in Section~\ref{sec:introproofs}.

\begin{remark}
Suppose $g\colon F\hookrightarrow S^4$ is a smooth embedding such that $g(F)$ is a connected sum of standardly embedded projective planes and that $\varphi \colon F \to F$ is a self-diffeomorphism.
Hirose~\cite{Hirose-nonorientable}  proved that $\varphi$ is induced by a self-diffeomorphism $\Phi\colon S^4\to S^4$ sending the image $g(F)$ to itself if and only if $g\circ\varphi \circ g^{-1}$ induces an isometry of the Guillou-Marin form on $g(F)$.
In the case $N=S^4$, it is possible to deduce the case of  Theorem~\ref{thm:isotopy-of-embeddings} corresponding to Theorem~\ref{thm:TopUnknottingIntro} by combining Theorem~\ref{thm:TopUnknottingIntro} with an application of Hirose's theorem.
 However, since this method would not have applied automatically in the cases when $N=D^4$, we gave the independent proof above, taking advantage of the flexibility in our method of proof (Section~\ref{sec:Spin} in particular) to choose our preferred Guillou-Marin form preserving homeomorphism.

Another view on Theorem~\ref{thm:isotopy-of-embeddings} is that it recovers a topological version of Hirose's theorem, by applying it in the special case that $g_0(F) \subseteq S^4$ is a connected sum of standardly embedded projective planes and $g_1=g_0\circ\varphi$ for some homeomorphism $\varphi\colon F\to F$.
From this point of view, we have obtained extensions of Hirose's theorem to surfaces with boundary in the cases where~$N=D^4$ covered by Theorems~\ref{thm:SurfacesWithBoundaryIntro} and~\ref{thm:one-is-enough}, but again only in the topological category. It would be interesting to know whether the smooth analogues also hold, but this cannot be addressed with the inherently topological methods of this paper.
\end{remark}

\subsection*{Organisation}

In Section~\ref{sec:EulerNumber}, we review the normal Euler number of an embedded surface.
In Section~\ref{sec:Forms}, we review hermitian and quadratic forms.
In Section~\ref{sec:AlgtopExterior},  we collect some facts about the algebraic topology of $\Z_2$-surface exteriors. Section~\ref{sub:IntersectionForms} is concerned with intersection forms.
In Section~\ref{sec:Spin}, we prove that if $F_0$ and $F_1$ have the same normal Euler number then it is possible to form the union~$M:=X_{F_0}\cup -X_{F_1}$ along a homeomorphism of the boundaries so that $M$ is a spin $4$-manifold and $\pi_1(M) \cong \Z_2$.
In Section~\ref{sec:Monoid}, we review quasi-formations and the definition of the monoid~$\ell_{2q+1}(R)$.
In Section~\ref{sec:ModifiedSurgery}, we recall the set-up of Kreck's modified surgery theory.
In Section~\ref{sec:PlusMinusForms}, we formulate a criterion for a class $\Theta \in \ell_5(\Z[\Z_2])$ to be elementary.
In Section~\ref{sec:Obstruction}, we analyse the surgery obstruction and prove that it is elementary.

 \subsection*{Conventions}

We work in the topological category with locally flat embeddings unless otherwise stated.
From now on, all manifolds are assumed to be compact, connected, and based; if a manifold has a nonempty boundary, then the basepoint is assumed to lie in the boundary.

Let $F$ denote a nonorientable, compact surface with connected, nonempty boundary and nonorientable genus $h$.
The \emph{standard cell structure} for $F$ consists of a 0-cell $e^0$ and a 1-cell~$e^1_\partial$ to construct~$\partial F$, then $h$ 1-cells $\gamma_1,\dots,\gamma_h$ attached to the 0-cell, followed by a 2-cell attached via~$e^1_{\partial}\cdot\gamma_1\cdot \gamma_1\cdots\gamma_h\gamma_h$.
Note that $F$ deformation retracts onto the subcomplex $e^0 \cup \bigcup_{i=1}^h \gamma_i \simeq \bigvee^h S^1$.

For a space $X$, we write $H_n(X)$ for the $n$-th singular homology group of $X$ with $\Z$ coefficients. Any other coefficient system will be included in the notation.
We write $\subseteq$ for inclusions of subsets and $\subset$ to indicate proper subsets.
The notation $\Sigma_2(-)$ is used exclusively for $2$-fold branched covers, typically of a knot in $S^3$ or a surface in $D^4$. 

If~$R$ is a ring with involution~$x \mapsto \overline{x}$ and~$M$ is a left~$R$-module then
$M^*$ denotes the dual of~$M$,~$\Hom_R(M, R)$,
which is given a left~$R$-module structure using the involution on~$R$ in the usual way.

\subsection*{Acknowledgments}
We thank Matthias Kreck for helpful discussions about \cite{KreckOnTheHomeomorphism}, and  Tom Mrowka for asking us about involutions, which led to Corollary~\ref{cor:Action}. 
We warmly thank the anonymous referee, whose comments led to improvements in the article.
AC was partially supported by the NSF grant DMS~2303674.
PO was partially supported by the SNSF Grant~181199. MP was partially supported by EPSRC New Investigator grant EP/T028335/2 and EPSRC New Horizons grant EP/V04821X/2.

\subsection*{Open Access}
For the purpose of open access, the authors have applied a CC-BY Creative Commons attribution license to this version, and will do the same to any author-accepted manuscript arising.

\section{The normal Euler number}
\label{sec:EulerNumber}

In this section, we discuss the normal Euler number of a properly embedded surface.
The normal Euler number, an integer,  is an isotopy invariant of $\Z_2$-surfaces, which can distinguish between those with the same nonorientable genus.

We start by recalling the general definition of Euler numbers with twisted coefficients.
Suppose~$(X,Y)$ is homotopy equivalent to a  CW pair, with both $X$ and $Y$ compact and connected.
Let~$\xi=(\R^n\to E\xrightarrow{\pi} X)$ be a rank $n$ vector bundle.
The example to keep in mind is that of the normal bundle of a properly embedded surface $F \subseteq D^4$.
Write~$D(\xi)$ and~$S(\xi)$ for the disc and sphere bundles of~$\xi$.
Consider the obstruction-theoretic problem of finding a section $s$ for the sphere bundle $S(\xi)$ extending a given section $s'$ on the restriction of $S(\xi)$ to $Y$:
\[
\begin{tikzcd}
Y\ar[rr,"s'"]\ar[d,hook]\ar[drr, hook]&& S(\xi)\ar[d,"S(\pi)"]\\
X\ar[rru,dashed,"{s}" near start]\ar[rr,"\Id_X"']&&X.
\end{tikzcd}
\]
The primary obstruction to this problem is the \emph{relative Euler class}, which takes values in the local coefficient system given by the homotopy groups of the fibres $\epsilon(\xi, s')\in H^n(X, Y;\{\pi_{n-1}(S(\xi)_x)\})$.
Write~$x_0\in Y$ for a basepoint and $\pi_1(X):=\pi_1(X,x_0)$.
The monodromy around a loop either preserves or reverses an orientation on $S(\xi)_{x_0}$, and this determines a homomorphism $\rho\colon \pi_1(X)\to \Aut(\pi_{n-1}(S(\xi)_{x_0}))$.
This homomorphism is used
to form a twisted coefficient system for~$X$, giving groups $H^n(X,Y;\pi_{n-1}(S(\xi)_{x_0}))$, canonically isomorphic to $H^n(X, Y;\{\pi_{n-1}(S(\xi)_x)\})$; see~\cite[page 288, Theorem 4.9]{WhiteheadElements}.

A cohomology class $w\in H^1(X;\Z_2)$ determines a homomorphism $\pi_1(X) \to \Aut(\Z)$ by assigning, to a loop $\gamma$, multiplication by $w([\gamma])\in\{\pm 1\}$. This in turn determines a twisted coefficient system~$\Z^w$. We may identify $\Z^{w_1(\xi)}$ with $S(\xi)_{x_0}$ as $\Z[\pi_1(X)]$-modules by choosing an isomorphism of groups $\Z\cong \pi_{n-1}(S(\xi)_{x_0})$.
Via either such choice of identification, we see that the monodromy-induced map $\rho$ and the homomorphism determined by $w_1(\xi)$ agree.
Thus, once such a choice is specified, we consider the relative Euler class $\epsilon(\xi, s')$ to lie in $H^n(X, Y;\Z^{w_1(\xi)})$.

Now suppose that $X$ is a compact $n$-manifold with (possibly empty) boundary $Y$.
There is an isomorphism $H_n(X,Y;\Z^{w_1(TX)})\cong \Z$ and a choice of generator $[X]$ is called a \emph{twisted $($relative$)$ fundamental class}.
Suppose the total space $E$ of $\xi$ is orientable, and fix a fundamental class for this total space.
A choice of a twisted fundamental class $[X]$ determines local orientations for~$X$.
Together with the orientation on the total space of $\xi$, we infer an orientation on the fibres of $\xi$ and hence an isomorphism $\Z\cong \pi_{n-1}(S(\xi)_{x_0})$.
Thus we obtain an Euler class $\epsilon(\xi,s')\in H^n(X,Y;\Z^{w_1(\xi)})$.
Noting that $w_1(TX)+w_1(\xi)=w_1(E)=0$,  there is a $\Z$-valued cap product
\[
-\cap-\colon H_n(X,Y;\Z^{w_1(TX)})\times H^n(X,Y;\Z^{w_1(\xi)})\to H_0(X;\Z^{w_1(TX)+w_1(\xi)}) = H_0(X) = \Z.
\]
We define the \emph{relative Euler number} $e(\xi, s')$ of $\xi$ to be the integer given by the cap product
$$e(\xi, s'):=[X]\cap \epsilon(\xi,s')\in H_0(X)=\Z.$$
\emph{A priori} this depends on a choice of~$[X]$ and a choice of orientation on $\xi$. A change in the choice of~$[X]$ switches the sign of $[X]$ and also switches the identification  $\Z\cong \pi_{n-1}(S(\xi)_{x_0}),$ thus overall these cancel in the computation of $e(\xi, s')$, and the relative Euler number is independent of the choice of $[X]$. Tracking a change in the choice of orientation on the total space of $\xi$ through the construction above, shows this will ultimately change $e(\xi, s')$ to $-e(\xi, s')$.

\begin{remark}\label{rem:intersectioncount} In practice, (relative) Euler numbers may be computed by looking at how generic sections of the bundle intersect the 0-section in the total space. Suppose $\xi$ is a vector bundle over a compact, connected $n$-manifold with boundary $(X,Y)$, and that the total space of $\xi$ is oriented. Fix a nonvanishing section $s'\colon Y\to E|_Y$.
Let~$s\colon X\to E$ be any section that restricts to a section isotopic to $s'$ on $Y$ and meets the image of the 0-section~$s_0(X)$ transversely. At the preimage in~$X$ of each intersection $x\in s_0(X)\pitchfork s(X)$, arbitrarily choose a local orientation for~$X$ (this corresponds to choosing one of the two twisted fundamental classes $[X]\in H_n(X,Y;\Z^{w_1(TX)})$). This induces an orientation on $T_xs_0(X)\oplus T_xs(X)$, and we assign a sign to~$x$ by comparing this induced orientation to the ambient orientation on $T_xE$. The signed count of these intersections is  independent of the local orientation choices because the choice appears in the orientation of both~$T_xs_0(X)$ and~$T_xs(X)$.

It is known that the Euler number~$e(\xi, s')\in\Z$ is equal to the signed count of these intersections.
This follows because the relative Euler number is equal to the algebraic self-intersection number of the $0$-section $s_0(X)$ in the total space.
In the level of generality we are working, this is a somewhat intricate exercise in elementary algebraic topology and we will not perform it here.
When $Y=\emptyset$, and for oriented base space and bundle, this procedure is carried out in~\cite[Proposition 12.8]{BredonTopology}, for example. By working with twisted Poincar\'{e} duality, and relative to the subspace $Y$, the method for the closed, oriented case can be adapted to the situation we are considering, where both the base space and bundle are nonorientable (and thus the total space of the bundle is orientable).
\end{remark}

\begin{definition}\label{def:Eulernumber}
The \emph{normal Euler number} $e(F) \in \Z$ of an embedded closed surface~$F\subseteq S^4$ is the Euler number of its normal bundle~$\nu F$, whose total space is oriented using the fixed, standard orientation on~$S^4$.

The \emph{normal Euler number} $e(F) \in \Z$ of a properly embedded compact surface~$F\subseteq D^4$ with connected nonempty boundary $K\subseteq S^3$ is the relative Euler number $e(\xi,s')\in \Z$ of its normal bundle~$\nu F$, whose total space is oriented using the fixed, standard orientation on~$D^4$, and where $s'\colon K\to\nu K $ denotes a  nonvanishing section (which, in this case, is equivalent to a longitude of $K$) corresponding to the Seifert framing of $K$.
\end{definition}

\begin{example}\label{example:standard embedding}
In the introduction,  we defined the standard embeddings $\R P^2 \subseteq S^4$ using the involution of $\C P^2$ given by complex conjugation.
An alternative construction for the positive/negative standard embedding~$\R P^2 \subseteq S^4$, up to ambient isotopy, is as follows. Take an unknotted M\"obius band in $S^3\subseteq S^4$, with one positive/negative half-twist, push its interior into the northern hemisphere of $S^4$, then cap it off with the standardly embedded $2$-disc $D$ in the other hemisphere.  With this description one can compute that the normal Euler number of the standard positive embedding is $-2$ and the normal Euler number of the standard negative embedding is~$+2$.
To compute this, first extend a nowhere zero section of the normal bundle of the pushed-in M\"{o}bius band to a generic section of the normal bundle of $D$. This gives a generic section of the embedded~$\R P^2$, where geometric intersections with the 0-section happen only on $D$. One can compute this intersection number between $D$ and its generic section by considering linking between their intersections with $S^3$.

For example, there is a nowhere zero section of the pushed-in M\"{o}bius band that restricts on the boundary unknot to be a pushoff in the outwards facing normal direction to $\mathcal{M}\subseteq S^3$. For the positive case, this is shown in Figure~\ref{fig}. This pushoff links +2 with the unknot, so that the generic section over $D$ in the southern hemisphere has intersection number $-2$ with the 0-section. Here the minus sign arises because the southern hemisphere uses the opposite orientation to the standard one.

\begin{figure}
\begin{tikzpicture}
     \node[inner sep=0pt] at (0,0)
    {\includegraphics[width=.3\textwidth]{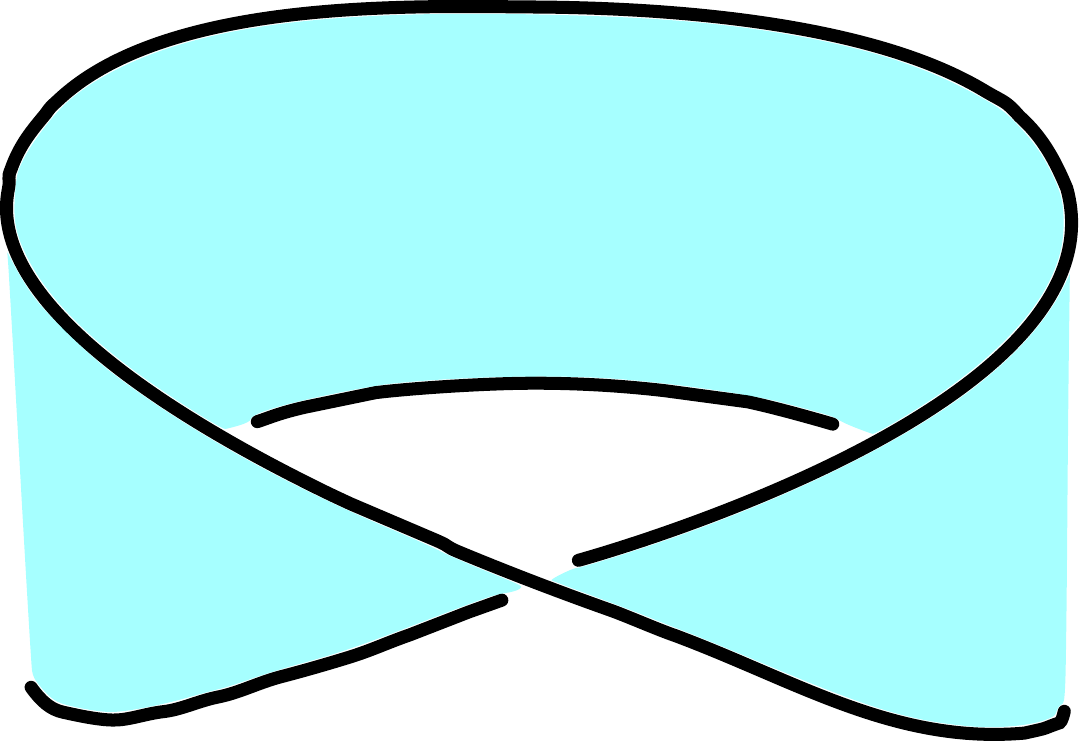}};
     \node[inner sep=0pt] at (6.5,0)
    {\includegraphics[width=.3\textwidth]{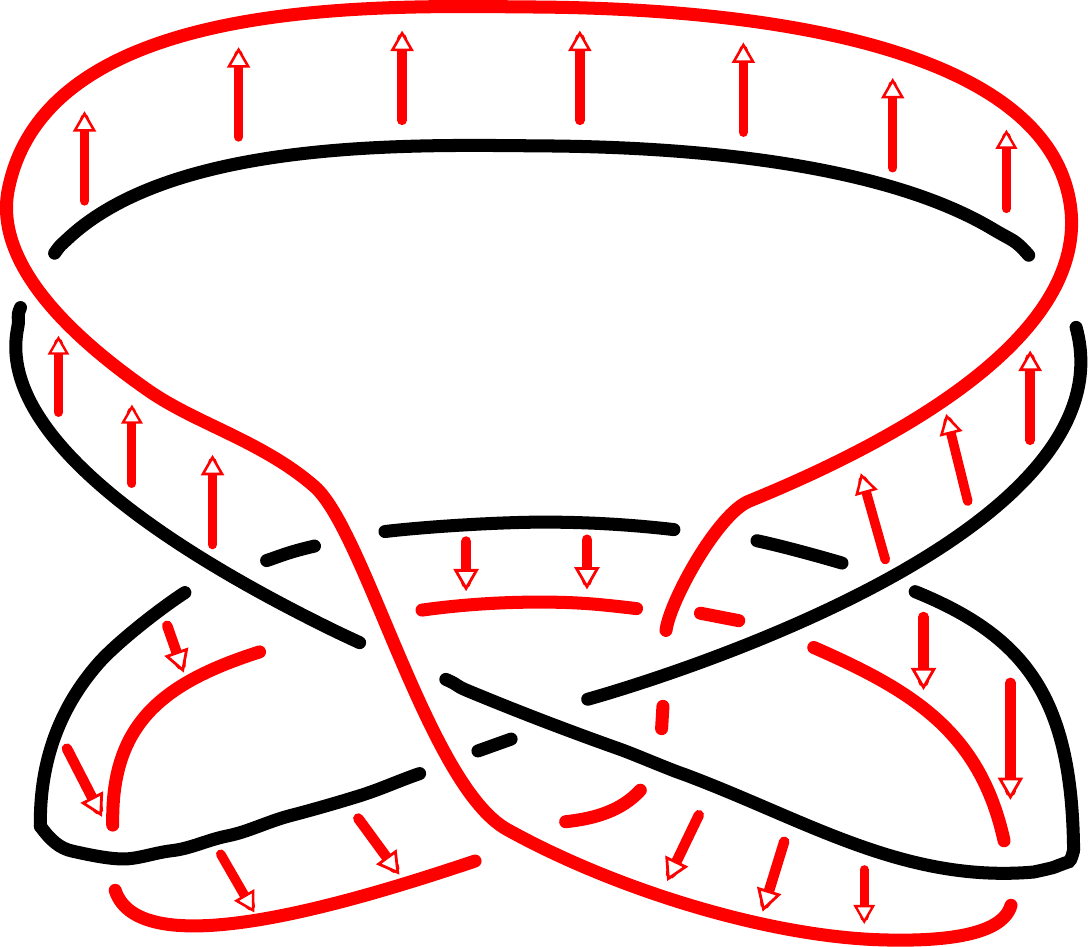}};
\end{tikzpicture}
\caption{Left: The unknotted, positively twisted M\"{o}bius band $\mathcal{M}\subseteq S^3$, before its interior is pushed into the northern hemisphere of $S^4$. Right: The red curve is the pushoff in the outwards-facing normal direction of $\mathcal{M}\subseteq S^3$. The pushoff direction is indicated by arrows. The pushoff has linking number~+2 with the boundary of the M\"{o}bius band. The pushoff is the restriction of a non-vanishing section of the normal bundle of the pushed-in M\"{o}bius band.}
\label{fig}
\end{figure}

It also follows from this description that the exteriors of these standard embeddings admit a handle decomposition with a single 0-handle and a single 1-handle,  and  therefore have fundamental group $\Z_2$.
This is the easiest way to see that all unknotted nonorientable surfaces, regardless of their nonorientable genus, are $\Z_2$-surfaces.
  \end{example}

Massey proved the following  theorem, as mentioned in the introduction. In particular it implies that the normal Euler number of a closed surface in $S^4$ is always even.

\begin{theorem}[Massey~{\cite{Massey}}]\label{thm:Massey}
  Let $F \subseteq S^4$ be an embedded, closed, nonorientable surface of nonorientable genus $h$. The normal Euler number of $F$ lies in the set $\{2h-4, 2h, 2h+4, \dots, 4-2h\}$. For each $h \geq 1$, this set is in bijective correspondence with the isotopy classes of the standard embeddings.
\end{theorem}

Technically speaking, Massey proved Theorem~\ref{thm:Massey} in the smooth category, but his proof extends to the topological category.
For example,  one of the crucial points in the proof is the relation between the Euler number~$e(F)$ and $\sigma(\Sigma_2(F))$, the signature of the $2$-fold branched cover.
This relation also holds in the topological category~\cite[Theorem 1]{GeskeKjuchukovaShaneson}.
Alternatively,  since, as noted in~\cite[Corollary after Theorem 5.2]{Matsumoto}, Massey's theorem follows from a congruence of Guillou-Marin~\cite{GuillouMarin},  one can alternatively note that the Guillou-Marin congruence holds in the topological category; see Theorem~\ref{thm:GM}.

In fact,  Massey's theorem can be generalised to surfaces with boundary, once one takes into account the signature $\sigma(K)$ of the knot $K=\partial F$.
The outcome reads as follows.
\begin{theorem}
\label{thm:MasseyBoundary}
If~$F \subseteq D^4$ is a nonorientable surface with boundary a knot~$K$,
then
\begin{equation}
\label{eq:SignatureEuler}
  \sigma(K)=\sign(\Sigma_2(F))+\smfrac{1}{2}e(F),
  \end{equation}
and the Euler number lies in the following range, where $h$ denotes the nonorientable genus of $F$:
\[ e(F) \in \lbrace -2h+2\sigma(K),-2h+4+2\sigma(K),\ldots, 2h-4+2\sigma(K),2h+2\sigma(K)\rbrace.\]
\end{theorem}
\begin{proof}
The equality in~\eqref{eq:SignatureEuler} can be found in~\cite[Theorem 2 and Corollary 5]{GordonLitherland}; see also~\cite[Theorem 6]{GilmerLivingston}.
The range of $e(F)$ is due to Yasuhara~\cite[Corollary 1.1.1]{YasuharaConnecting}.
\end{proof}

\section{Hermitian and quadratic forms}
\label{sec:Forms}

In Subsection~\ref{sub:QuadraticForms} we recall terminology regarding hermitian and quadratic forms.
In Subsection~\ref{sub:ModulesOverZZ2}, we collect some important facts about $\Z[\Z_2]$-modules.

\subsection{Quadratic forms}
\label{sub:QuadraticForms}

We begin by recalling some notation and terminology regarding hermitian and quadratic forms, which we will use throughout the article.
References include~\cite{RanickiExact} and~\cite[Subsection 3.2]{CrowleySixt}.

\medbreak
Let $R$ be a ring with involution and let $\varepsilon = \pm 1$.
Given an $R$-module $P$, consider the involution
$$
\varepsilon \top \colon \Hom_R(P,P^*) \to \Hom_R(P,P^*), \phi \mapsto (x \mapsto (y \mapsto \varepsilon \overline{\phi(y)(x)})).$$
The quadratic and symmetric \emph{$Q$-groups} are then respectively defined as
\[Q_\varepsilon(P):=\operatorname{coker}(1-\varepsilon \top) \ \ \text{ and }  \  \  Q^\varepsilon(P):=\operatorname{ker}(1-\varepsilon \top).\]
We will often abbreviate $Q_{\pm}(P) := Q_{\pm 1}(P)$.
We also note that for $P=R$, it is customary to identify $Q_\varepsilon(R)$ with the quotient $R/\{ r-\varepsilon\overline{r} \mid r \in R\}.$

\begin{definition}\leavevmode
\label{def:QuadraticForm}
\begin{enumerate}[(i)]
\item An \emph{$\varepsilon$-quadratic form} is a pair $(P,\psi)$, where $P$ is a stably free $R$-module and~$\psi \in Q_\varepsilon(P)$.
    \item    An \emph{$\varepsilon$-hermitian form} is a pair $(P,\lambda)$, where $P$ is a stably free $R$-module and~$\lambda \in Q^\varepsilon(P)$.
    \item An $\varepsilon$-quadratic form is \emph{nonsingular} if its \emph{symmetrisation} $(1+\varepsilon \top)\psi$ is an isomorphism.
 \item An \emph{isometry} $h \colon (P,\psi) \to (P',\psi')$ between quadratic forms consists of an isomorphism $h \colon P \to P'$ such that $h^*\psi'h=\psi \in Q_\varepsilon(P)$. Isometries of hermitian forms are defined similarly.
\end{enumerate}
\end{definition}

\begin{remark}
\label{rem:FormsAreBetter}
As is customary in surgery theory, we will go back and forth between thinking of quadratic forms (resp.~hermitian forms) as (equivalence classes of) $R$-linear maps $P \to P^*$ on the one hand and pairings $P \times P \to Q_\varepsilon(R)$ (resp.\ pairings~$P \times P \to R$) on the other.
We will also frequently use the  fact that the data of a quadratic form $(P,\psi)$ is  equivalent to the data of a triple~$(P,\lambda,\mu)$ where $\lambda  \colon P \times P \to R$ is a hermitian form and $\mu \colon P \to Q_\varepsilon(R)$ is a quadratic refinement of $\lambda$, a notion whose definition we recall next; see e.g.~\cite[Remark 8.78]{LueckMacko}.
\end{remark}

\begin{definition}\label{defn:quadratic-refinement}
Let $q \colon R \to Q_{\varepsilon}(R)$ be the quotient map and let $s \colon Q_{\varepsilon}(R) \to R$ be a section of $q$.
 A function  $\mu \colon P \to Q_{\varepsilon}(R)$ is a \emph{quadratic refinement} of a hermitian form $\lambda \colon P \times P \to R$ if
 \begin{enumerate}[(i)]
   \item $\mu(x+y) - \mu(x) - \mu(y) = q(\lambda(x,y))$ for every $x,y \in P$,
   \item $\lambda(x,x) = s(\mu(x)) + \varepsilon\ol{s(\mu(x))}$ for every $x \in P$,
   \item $\mu(rx) = r \mu(x) \ol{r} \in Q_{\varepsilon}(R)$ for every $x \in P$ and $r \in R$.
 \end{enumerate}
 Note that the right hand side of (ii) does not depend on the choice of section~$s$.
\end{definition}

\begin{definition}\label{defn:hyperbolic-form}
For a stably free $R$-module~$L$, the \emph{standard hyperbolic $\varepsilon$-quadratic form} is
\[ H_\varepsilon(L)=\left( L \oplus L^*, \left[ \begin{pmatrix} 0 & 1 \\ 0 & 0 \end{pmatrix} \right] \right).\]
A quadratic form $(P, \psi)$ is called {\em hyperbolic} if it is isometric to $H_\varepsilon(L)$ for some $L$.
\end{definition}

\begin{definition}\label{defn:lagrangian}
A \emph{sublagrangian} of an $\varepsilon$-quadratic form $(P,\psi)$ is a direct summand $j \colon L \hookrightarrow P$ that satisfies~$j^*\psi j=0 \in Q_\varepsilon(L)$.
A sublagrangian $j \colon L \hookrightarrow P$ satisfies $j(L) \subseteq L^\perp$, where the module $L^\perp=\ker(j^*(1+\varepsilon \top)\psi)$ is called the \emph{annihilator} of $j \colon L \hookrightarrow P$.
A sublagrangian~$j \colon L \hookrightarrow P$ is a  \emph{lagrangian} if it satisfies $L=L^\perp$.
\end{definition}

\begin{remark}
When $P$ is only assumed to be projective, this definition requires the additional assumption that $j^*(1+\varepsilon \top)\psi$ is surjective; cf.~\cite[p.~65]{RanickiExact}.
This is relevant to the study of projective $L$-theory, see e.g.~\cite[page 260]{WallSurgeryOnCompact} and the references within.
As we are working in the setting that $P$ is stably free, this condition follows automatically.
\end{remark}

We conclude this subsection by discussing the relationship between quadratic and hermitian forms.
Given an $R$-module $P$, the map $1+\varepsilon \top$ induces an $R$-linear map
\begin{equation}
\label{eq:symmetrisation}
 1+\varepsilon \top  \colon Q_\varepsilon(P) \to Q^\varepsilon(P).
 \end{equation}

 \begin{lemma}\label{lem:projectivelemma}
 When $\varepsilon=1$ and $P$ is projective, the symmetrisation map~\eqref{eq:symmetrisation} is injective.
 \end{lemma}

 \begin{proof}
 Assume first that $P$ is moreover free. Then an element in the kernel of \eqref{eq:symmetrisation} is represented by a skew-hermitian matrix $A$ over $R$. A skew-hermitian matrix may be written $A=B^T-\overline{B}=(1-\top)B^T$, where $B$ is the upper triangle of $A$. Thus $A$ determines the 0 element of $Q_+(P)$. It follows that in this case, the map \eqref{eq:symmetrisation} is injective.

Now drop the assumption the $P$ is free. As $P$ is projective, there exists an $R$-module~$Q$ such that $P\oplus Q$ is free. Suppose $\psi\colon P\to P^*$ represents an element of $\ker( 1+ \top \colon Q_+(P)\to Q^+(P))$. Then $\psi\oplus 0$ represents an element of $\ker( 1+ \top \colon Q_+(P\oplus Q)\to Q^+(P\oplus Q))$. As $P\oplus Q$ is free, by the argument in the first paragraph, $\psi\oplus0=(1-\top)\eta$ for some $\eta\colon P\oplus Q\to (P\oplus Q)^*\cong P^*\oplus Q^*$. Writing $i\colon P\hookrightarrow P\oplus Q$ for the inclusion, we have
\[
\psi=i^*(\psi\oplus 0)i=i^*(1-\top)\eta i=i^*(1-\top)\eta i= (1-\top)(i^*\eta i),
\]
and so we have that $\psi$ determines the 0 element of~$Q_+(P)$.
 \end{proof}

We obtain the following consequences.

 \begin{corollary}\label{cor:niceconsequences}
If a $(+1)$-hermitian form $(P,\lambda)$ admits a quadratic refinement, then it admits a unique quadratic refinement. Moreover, if $L\subseteq P$ is a direct summand and $\lambda$ vanishes on $L$, then~$L$ is a sublagrangian for that unique quadratic form with symmetrisation $(P,\lambda)$.
 \end{corollary}

 \begin{proof}
 The first claim is immediate from Lemma~\ref{lem:projectivelemma}, as $P$ is stably free, therefore projective (indeed the claim holds for forms defined more generally on projective modules).
Indeed,  if two quadratic forms $(P,\psi_1),(P,\psi_2)$ admit $(P,\lambda) \in Q^+(P)$ as a common symmetrisation, then Lemma~\ref{lem:projectivelemma} implies that $(P,\psi_1)=(P,\psi_2)\in Q_+(P)$ and therefore the corresponding quadratic refinements agree; here we use Remark~\ref{rem:FormsAreBetter} to go back and forth between quadratic forms and quadratic refinements.

 Write $(P,\psi)$ for the unique quadratic form such that $ (1+ \top)\psi=\lambda$.
 The second claim comes from the fact that as~$L$ is a direct summand of~$P$, it is also projective. The vanishing of $\lambda$ on~$L$ is equivalent to saying $0=j^*(1+ \top)\psi j=(1+ \top)j^*\psi j\in Q^+(L)$.
 As $L$ is projective, Lemma~\ref{lem:projectivelemma} shows~$0=j^*\psi j\in Q_+(L)$, and so $L$ is a sublagrangian for $(P,\psi)$.
 \end{proof}

\subsection{Modules over \texorpdfstring{$\Z[\Z_2]$}{Z[Z2]}}
\label{sub:ModulesOverZZ2}

Since we will be soon describing the $\Z[\Z_2]$-module structure of various homology groups, we start by recalling some facts about $\Z[\Z_2]$-modules.
From now on we write~$\Z_2=\langle T \mid T^2=1 \rangle$.

\begin{lemma}
\label{lem:StablyFreeImpliesFree}
For a finitely generated $\Z[\Z_2]$-module $P$,  the following assertions are equivalent:
\begin{enumerate}
\item $P$ is free,
\item $P$ is stably free,
\item $P$ is projective.
\end{enumerate}
\end{lemma}
\begin{proof}
For finitely generated modules over any ring,  free implies stably free, which in turn implies projective.
A theorem of Jacobinski~\cite{Jacobinski} ensures that stably free $\Z[\Z_2]$-modules are moreover free; see also~\cite{Johnson}.
For any finite group $G$,
every finitely generated (f.g.)~projective $\Z[G]$-module is stably free if and only if the reduced projective class group~$\widetilde{K}_0(\Z[G])$ is trivial~\cite[Lemma~2.1, Definition~2.3, and Example~2.4]{WeibelKBook}.
Rim proved that for a prime~$p$,  the group~$\widetilde{K}_0(\Z[\Z_p])$ is isomorphic to~$\widetilde{K}_0(\Z[\xi_p])$, where~$\xi_p$ denotes the~$p$-th primitive root of unity~\cite{RimModules}.
A result of Milnor then implies that~$\widetilde{K}_0(\Z[\xi_p])$ is isomorphic to the ideal class group of~$\Z[\xi_p]$~\cite[Section~1, Corollary 1.11]{MilnorIntroduction}.
This latter group is trivial for~$p=2$ since~$\Z[\xi_2]=\Z$.
Thus a f.g.\ $\Z[\Z_2]$-module is projective if and only if it is stably free.
\end{proof}

\begin{lemma}
\label{lem:Reiner}
Every f.g.\ $\Z[\Z_2]$-module $H$ that is torsion-free as an abelian group decomposes as $P \oplus H_{+1} \oplus H_{-1}$, where~$P$ is~$\Z[\Z_2]$-free, $H_{+1}$ has the
$\Z[\Z_2]$-action where $T$ operates by~$+1$, and~$H_{-1}$ has the~$\Z[\Z_2]$-action where $T$ operates by~$-1$.
\end{lemma}

\begin{proof}
Reiner~\cite{Reiner} proved that every f.g.\ $\Z[\Z_2]$-module $H$ that is torsion-free as an abelian group decomposes as $P \oplus H_{+1} \oplus H_{-1}$, where $H_{\pm 1}$ are as in the statement of the lemma and $P$ is~$\Z[\Z_2]$-projective.
Summands of f.g.\  modules are f.g.\  and Lemma~\ref{lem:StablyFreeImpliesFree} therefore implies that $P$ is in fact $\Z[\Z_2]$-free.
\end{proof}

\begin{remark}
\label{rem:pm}
In what follows,  $\Z_\pm$ denotes $\Z$ with the $(\Z[\Z_2],\Z[\Z_2])$-bimodule structure where both left and right module structures are determined by stipulating that $T$ induces multiplication by $\pm 1$.
One verifies for example the useful isomorphism $\Hom_{\Z[\Z_2]}(\Z_\pm,\Z[\Z_2]) \cong \Z_\pm$.
\end{remark}

\section{The algebraic topology of \texorpdfstring{$\Z_2$}{Z/2}-surface exteriors}
\label{sec:AlgtopExterior}

In this section we collect some facts about the algebraic topology of $\Z_2$-surface exteriors, which we will need in later proofs.
Subsections~\ref{sub:HomologyBoundary} and~\ref{sub:HomologyBoundarySurfaceExterior} are concerned with the homology of the boundary of~$\Z_2$-surface exteriors.
In Subsections~\ref{sub:HomologyExteriors} and~\ref{sub:HomologyBranched} we study the homology of $\Z_2$-surface exteriors and the corresponding branched covers.

\medbreak

Before getting started we record a result that we will use frequently.
If $\Z_2=\langle T \mid T^2=1 \rangle$ acts on a space $X$, its rational homology groups split into eigenspaces for the automorphism induced by $T$.
The eigenvalues are necessarily $\pm 1$ and we write $\mathcal{E}_\pm(H_i(X;\Q)) \subseteq H_i(X;\Q)$ for the corresponding eigenspaces.
Note that $H_i(X;\Q)=\mathcal{E}_+(H_i(X;\Q)) \oplus \mathcal{E}_-(H_i(X;\Q))$ because every~$x \in H_i(X;\Q)$ decomposes as $x=\tmfrac{1}{2}(x+Tx)+\tmfrac{1}{2}(x-Tx)$ with $\tmfrac{1}{2}(x\pm Tx) \in\mathcal{E}_\pm(H_i(X;\Q)).$

\begin{lemma}
\label{lem:Bredon}
Fix $i \geq 0$.
If $\Z_2$ acts on a space $X$, then $\mathcal{E}_+(H_i(X;\Q))\cong H_i(X/\Z_2;\Q)$ and therefore
\[H_i(X;\Q) \cong H_i(X/\Z_2;\Q) \oplus \mathcal{E}_-(H_i(X;\Q)).\]
\end{lemma}
\begin{proof}
Consider the set~$H_{i}(X;\Q)^{\Z_2}$ of fixed points of the~$\Z_2$ action on~$H_{i}(X;\Q)$.
A transfer map argument~\cite[Chapter~III, Theorem~7.2]{BredonIntroduction} shows that fixed set of the induced action on the rational homology is isomorphic to the rational homology of the orbit space, i.e.~$H_{i}(X;\Q)^{\Z_2}\cong H_{i}(X/\Z_2;\Q)$.
We therefore obtain the following sequence of isomorphisms:
\[
\mathcal{E}_+(H_{i}(X;\Q))= H_{i}(X;\Q)^{\Z_2} \cong H_{i}(X/\Z_2;\Q).
\]
The final sentence now follows from the decomposition $H_i(X;\Q)=\mathcal{E}_+(H_i(X;\Q)) \oplus \mathcal{E}_-(H_{i}(X;\Q))$.
\end{proof}

\subsection{The homology of circle bundles over nonorientable surfaces}
\label{sub:HomologyBoundary}

Given a properly embedded nonorientable surface $F \subseteq D^4$ with boundary a knot $K \subseteq S^3$, the boundary of the exterior $X_F:=D^4 \setminus \nu F $ decomposes as~$\partial X_F=X_K \cup -\mathring{Y}$ where $X_K$ is the exterior of $K$ and $\mathring{Y}$ is an oriented $3$-manifold that has the structure of a circle bundle over $F$ (see Notation~\ref{notation:Y} for details).
We calculate the homology of $\mathring{Y}$ as well as the homology of circle bundles over closed nonorientable surfaces.

We first introduce some notation we will use throughout this section.
\begin{notation}
\label{notation:Y}
Let~$F$ be a closed, nonorientable surface of nonorientable genus $h$ and let~$\xi$ be a rank 2 vector bundle over~$F$ with oriented total space and Euler number $e$.
Decompose~${F}=D^2\cup_{S^1}\mathring{{F}}$ and write~$\mathring{\xi}$ for the restriction to~$\mathring{F}$.
Write~$Y$ and $\mathring{Y}$ for the respective total spaces of the associated~$S^1$-bundles~$S(\xi)$ and~$S(\mathring{\xi})$, and let $p \colon Y \to F$ and $\mathring{p} \colon \mathring{Y} \to \mathring{F}$ denote the projection maps.
Let $\mathring{s}$ be a section of $\mathring{\xi}$, which exists because $\mathring{\xi}$ is rank 2 and $\mathring{F}$ is homotopy equivalent to a 1-complex.
\end{notation}

\begin{proposition}
\label{prop:HomologyY}
We have a commuting diagram
\[\xymatrix{
0 \ar[r]& \Z_2 \ar[r]\ar[d]^= & H_1(\mathring{Y}) \ar@{->>}[d]^{\iota_Y} \ar@{->>}[r]^{\mathring{p}}& H_1(\mathring{F}) \ar[r]\ar@{->>}[d]^{\iota_F}\ar@/^1pc/[l]^{\mathring{s}_*}& 0 \\
0 \ar[r]&  \Z_2 \ar[r]& H_1(Y) \ar@{->>}[r]^{p}& H_1(F) \ar[r]\ar@/^1pc/@{-->}[l]^{\sigma}& 0
}\]
where the vertical maps are inclusion induced, the top row is split by $\mathring{s}_*$, and the bottom is split if and only if $e$ is even.
Moreover, we have the following homology calculations.
\begin{itemize}
\item If $e$ is even,
there exists a unique map $\sigma \colon H_1(F) \to H_1(Y)$ that splits the bottom row and satisfies $\iota_Y \circ \mathring{s}_*=\sigma \circ \iota_F$.
In particular,  there are isomorphisms
\[
H_1(Y)\cong \Z_2 \oplus H_1(F) \  \ \text{  and  } \  \ H_1(\mathring{Y})\cong \Z_2 \oplus H_1(\mathring{F}).
\]
In both cases the~$\Z_2$-summand is generated by an~$S^1$-fibre of~$S(\xi)$.
\item
If $e$ is odd,  then
\[
H_1(Y)\cong \Z_4 \oplus \Z^{h-1} \  \ \text{  and  } \  \ H_1(\mathring{Y})\cong \Z_2 \oplus H_1(\mathring{F}).
\]
The~$\Z_2$-summand is generated by an~$S^1$-fibre of~$S(\xi)$ whereas the $\Z_4$ summand fits into the nonsplit short exact sequence
$$ 0 \to \Z_2 \to  {\underbrace{ H_1(Y)}_{\cong\, \Z_4\oplus \Z^{h-1}}} \xrightarrow{p_*} {\underbrace{ H_1(F)}_{\cong\, \Z_2\oplus \Z^{h-1}}} \to 0$$
in which the left hand side $\Z_2$ is generated by an~$S^1$-fibre of~$S(\xi)$.
\end{itemize}
\end{proposition}

\begin{proof}
Consider the total space $N$ of the disc bundle $D(\xi)$ over $F$.
We note that $H_1(N,Y)=0$ and $H_2(N,Y) \cong  \Z_2$ generated by a $D^2$-fibre of $D(\xi)$.
Indeed since $\partial N=Y$ and $N$ deformation retracts onto $F$, Poincar\'e duality implies that
$ H_i(N,Y) \cong H^{4-i}(N) \cong H^{4-i}(F)$
which vanishes for $i=1$ and is~$\Z_2$ for $i=2$.
In the punctured case,  write $\mathring{N}$ for the total space of $D(\mathring{\xi})$, so that excision gives~$H_i(\mathring{N},\mathring{Y}) \cong H_i(N,Y)$ for all $i \geq 0.$

Since $N$ (resp.\ $\mathring{N}$) deformation retracts via the bundle projection onto $F$ (resp.\ $\mathring{F}$),  we have $H_1(N) \cong H_1(F)$ (resp. $H_1(\mathring{N}) \cong H_1(\mathring{F})$).
Consider the following commutative diagram in which the rows (resp. columns) arise from the exact sequences of the pairs $(\mathring{N},\mathring{Y})$ and $(N,Y)$ (resp.~$(F,\mathring{F})$ and~$(Y,\mathring{Y})$):
\[\xymatrix{
& & H_2(Y,\mathring{Y}) \ar[r] \ar[d]^{\partial} & H_2(F,\mathring{F}) \ar[d]^{\partial} & \\
0 \ar[r]& \Z_2 \ar[r]\ar[d]^= & H_1(\mathring{Y}) \ar @{->>}[d]^{\iota_Y} \ar[r]^{\mathring{p}}& H_1(\mathring{F}) \ar[r]\ar@{->>}[d]^{\iota_F}\ar@/^1pc/[l]^{\mathring{s}_*}& 0 \\
0 \ar[r]&  \Z_2 \ar[r]& H_1(Y) \ar[r]^{p}& H_1(F) \ar[r]\ar@/^1pc/@{-->}[l]^{\sigma}& 0.
}\]
Here the map $\mathring{s}_*$ splits $\mathring{p}$ and is induced by the choice of  section~$\mathring{s}\colon \mathring{F} \to \mathring{Y}$.
It remains to prove the statements related to the splitting $\sigma$.

By excision, $H_2(F,\mathring{F}) \cong H_2(D^2,S^1) \cong \Z$ generated by a choice of relative fundamental class~$[D^2]$.
Using our standard cell structure for $\mathring{F}$, with $1$-cells $\gamma_1,\dots,\gamma_h$, we see that $H_1(\mathring{F})\cong\Z^h$. We choose a basis with first basis element $[\gamma_1\gamma_2\dots\gamma_h]$ to obtain a decomposition $H_1(\mathring{F}) \cong \Z \oplus \Z^{h-1}$, so that~$\partial([D^2]) = (2,\mathbf{0})$.
Using the splitting of the top sequence we write $H_1(\mathring{Y}) \cong \Z_2 \oplus H_1(\mathring{F}) \cong \Z_2 \oplus \Z \oplus \Z^{h-1}$.  Again by excision, we know that $H_2(Y,\mathring{Y}) \cong H_2(S^1 \times D^2,S^1 \times S^1) \cong \Z$, generated by $[\{1\} \times D^2]$.

We claim that $\partial([\{1\} \times D^2]) =
(\ol{e},2,\mathbf{0}) \in \Z_2 \oplus \Z \oplus \Z^{h-1}$, where $\ol{e}$ is the Euler number $e$ of~$\xi$ reduced modulo 2.
The last two entries are necessitated by commutativity of the top right square.
To see that the first entry is the Euler number mod 2, recall that the Euler number is the obstruction to extending the section $\mathring{s} \colon \mathring{F} \to \mathring{Y} \subseteq Y = S(\xi) \subseteq D(\xi)$ to a nonvanishing section on all of $F$.
Indeed, if we extend, the Euler number equals the number of transverse intersections of any section over all of~$F$ with the zero section.
 This means that the boundary of $\{1\} \times D^2 \subseteq S^1 \times D^2 \cong Y \setminus \mathring{Y}$ maps to $\pm e$ copies of the circle fibre.
This concludes the proof of the claim.

Assume $e$ is even.
A chase in the following diagram shows that $\sigma$ is well-defined:
\[\xymatrix{
& & \Z \ar[r] \ar[d]^{(0,2,\mathbf{0})} & \Z \ar[d]^{(2,\mathbf{0})} & \\
0 \ar[r]& \Z_2 \ar[r]\ar[d]^= & \Z_2 \oplus \Z \oplus \Z^{h-1} \ar @{->>}[d]^{\iota_Y} \ar[r]^-{\mathring{p}}&  \Z \oplus \Z^{h-1} \ar[r] \ar@{->>}[d]^{\iota_F} \ar@/^1pc/[l]^{\mathring{s}_*} & 0 \\
0 \ar[r]&  \Z_2 \ar[r]& H_1(Y) \ar[r]^{p}& H_1(F) \ar[r]\ar@/^1pc/@{-->}[l]^{\sigma}& 0.
}\]
In more detail, the map~$\sigma$ is defined to make the right hand square commute; i.e.\ $\sigma(x)=\iota_Y \circ \mathring{s}_*(\mathring{x})$ where $\iota_F(\mathring{x})=x$.
To check that $\sigma$ is indeed  a splitting, given $x \in H_1(F)$ and $\mathring{x} \in H_1(\mathring{F})$ such that~$\iota_F(\mathring{x} ) =x$. Then
$p \circ \sigma(x)
 = p  \circ  \sigma\circ \iota_F(\mathring{x} )
 = p  \circ  \iota_Y \circ \mathring{s}_*(\mathring{x} )
 = \iota_F  \circ  \mathring{p} \circ  \mathring{s}(\mathring{x} )
 = \iota_F(\mathring{x} ) = x.$
The uniqueness of $\sigma$ follows similarly: if $\sigma'$ satisfies $\iota_Y \circ \mathring{s}_*=\sigma' \circ \iota_F$, then for any $x \in H_1(F)$ and any~$\mathring{x}  \in H_1(\mathring{F})$ with~$\iota_F(\mathring{x})=x$, we have $\sigma(x)=\iota_Y\circ \mathring{s}_*(\mathring{x})=\sigma'(\iota_F(\mathring{x}))=\sigma'(x)$.

If $e$ is odd, then we deduce that  $H_1(Y)=\im(\iota_Y)$ is isomorphic to $\Z_2 \oplus \Z \oplus \Z^{h-1}/\langle (1,2,0)\rangle$, which is itself isomorphic to $\Z_4 \oplus \Z^{h-1}$.
The proposition follows.
\end{proof}

\begin{proposition}
\label{prop:HomologyGroupsmathringY}
The nontrivial integral homology groups of $\mathring{Y}$ are
\[
H_0(\mathring{Y})\cong\Z,\quad H_1(\mathring{Y})\cong \Z_2 \oplus \Z^h,\quad  \text{ and }\quad H_2(\mathring{Y})\cong\Z^{h-1},
\]
and the   nontrivial integral homology groups of~$Y$ are
\[H_0(Y)\cong \Z, \quad
H_1(Y)\cong\begin{cases}
\Z_4 \oplus \Z^{h-1} &\quad \text{ if~$e$ is odd}\\
\Z_2 \oplus \Z_2 \oplus \Z^{h-1} &\quad \text{ if~$e$ is even,}
\end{cases} \quad
H_2(Y)\cong \Z^{h-1}, \quad
 H_3(Y)\cong \Z.\]
\end{proposition}

\begin{proof}
Since~$\mathring{Y}$ is connected~$H_0(\mathring{Y})\cong\Z$.  Proposition~\ref{prop:HomologyY} proves that~$H_1(\mathring{Y})\cong \Z_2 \oplus \Z^h$.
Since~$\mathring{Y}$ is a~$3$-manifold with torus boundary, we have~$H_3(\mathring{Y})=0$ and~$0=\chi(\mathring{Y})=1-h+b_2(\mathring{Y})$.
On the other hand, Poincar\'e duality and the universal coefficient theorem imply that~$H_2(\mathring{Y})$ is free.
Combining these facts shows that~$H_2(\mathring{Y}) \cong \Z^{h-1}$, as claimed.
An analogous argument gives the closed case.
\end{proof}

\begin{construction}
\label{cons:Covers}
Assume the Euler number $e$ of the bundle $\xi$ is even, so that Proposition~\ref{prop:HomologyY} gives
\[
H_1(Y)\cong \Z_2 \oplus H_1(F) \  \ \text{  and  } \  \ H_1(\mathring{Y})\cong \Z_2 \oplus H_1(\mathring{F}),
\]
where in both cases the~$\Z_2$-summand is generated by an~$S^1$-fibre of~$S(\xi)$.
Projecting to this~$\Z_2$-summand leads to surjections~$\varphi \colon \pi_1(Y) \twoheadrightarrow \Z_2$ and~$\varphi \colon \pi_1(\mathring{Y}) \twoheadrightarrow \Z_2$.
We consider the~$\Z_2$-covers~$Y^\varphi$ and~$\mathring{Y}^\varphi$ corresponding to the kernels of these epimorphisms.
From now on when we write~$H_i(Y;\Z[\Z_2])$ and~$H_i(\mathring{Y};\Z[\Z_2])$, we will mean the homology groups of these covers, viewed as $\Z[\Z_2]$-modules.
\end{construction}


\begin{lemma}\label{lem:newlemma}
When the Euler number $e$ of $\xi$ is even, it is twice that of the rank 2 bundle~$\xi^\varphi$ associated with~$Y^\varphi$.
\end{lemma}

\begin{proof}
Consider $N:=N_{e,h}$, the $D^2$-bundle over $F$ with Euler number $e$.
This bundle satisfies~$\partial N=Y$.
Note that $N\sm \nu F \cong Y\times [0,1]\simeq Y$, so that $\varphi$ determines a $2$-fold branched cover $p\colon N_2(F)\to N$, branched over $F\subseteq N$, and with $Y^\varphi=\partial N_2(F)$.
 Orient $N_2(F)$ so that~$p$ is degree~1.
%
Write $\widetilde{F} \subseteq N_2(F)$ for the branch set of this branched cover.
Write $e'$ for the Euler number of~$Y^\varphi$. This agrees with the normal Euler number of $\widetilde{F} \subseteq N_2(F)$.
The spaces $N_2(F)$ and $N$ have the structure of tubular neighbourhoods over $\widetilde{F}$ and $F$ respectively, and we choose these tubular neighbourhoods so that the double branched covering map $p$ is a fibre preserving $D^2$-bundle map of the normal bundles $p\colon \nu \widetilde{F} \to \nu F $.
Implicitly using the homeomorphism $p\colon \widetilde{F}\to F$ in the base, there is a change of local coefficients homomorphism
$p_{\#}\colon H^2(\widetilde{F};\{\pi_{1}(S(\xi)_x)\})\to H^2(F;\{\pi_{1}(S(\xi)_x)\})$ that sends~$e(\nu \widetilde{F})$ to~$e(\nu F )$; cf.~\cite[Lemma 1]{Massey}.
As $p$ restricts to a double cover on the $S^1$ fibres, this map is multiplication by $\pm 2$. The sign ambiguity of an Euler number depends only on the choice of orientation on the ambient $4$-manifolds (see Section~\ref{sec:EulerNumber}). The stipulation earlier that~$p\colon N_2(F)\to N$ is a degree 1 map, fixes the sign to be positive.
Thus, $2e'=2e(\widetilde{F})=e(F)=e$, as claimed.
\end{proof}

\begin{remark}\label{rem:newremark} The Euler number $e$ of a $D^2$-bundle $N$ over $F$ is, by definition, equal to both the Euler number of the associated sphere bundle, and the normal Euler number $e(F)$ of $F\subseteq N$, the~0-section of the $D^2$-bundle in the total space. Now assume $e$ is even and write $\widetilde{F} \subseteq N_2(F)$ for the branch set of the $2$-fold branched cover described in the proof of Lemma~\ref{lem:newlemma}. It follows that Lemma~\ref{lem:newlemma} is equivalent to the statement that the normal Euler numbers satisfy~$e(F)=2e(\widetilde{F})$.
\end{remark}

\begin{proposition}
\label{prop:HomologyYCover}
Assume the Euler number $e$ of the bundle $\xi$ is even.
The   nontrivial homology groups of the~$\Z_2$-cover~$\mathring{Y}^\varphi$ are
$$
H_0(\mathring{Y}^\varphi)\cong \Z, \quad
H_1(\mathring{Y}^\varphi)\cong\Z_2 \oplus H_1(\mathring{F}), \quad  \text{ and }\quad
H_2(\mathring{Y}^\varphi)\cong \Z^{h-1}.
$$
If we write $e=2x$,  then the   nontrivial homology groups of the~$\Z_2$-cover~$Y^\varphi$ are
$$
H_0(Y^\varphi)\cong \Z, \quad
H_1(Y^\varphi)\cong \begin{cases}
\Z_4 \oplus \Z^{h-1}&\quad \text{ if~$x$ is odd}\\
\Z_2 \oplus \Z_2 \oplus \Z^{h-1} &\quad \text{ if~$x$ is even,}
\end{cases} \quad
H_2(Y^\varphi)\cong\Z^{h-1}, \quad
 H_3(Y^\varphi)\cong \Z.
$$
\end{proposition}
\begin{proof}
The definition of the surjection $\varphi \colon \pi_1(Y) \twoheadrightarrow \Z_2$ implies that~$Y^\varphi$ is again a circle bundle with base a surface of nonorientable genus~$h$, and is the sphere bundle of a rank two vector bundle~$\xi^{\varphi}$ over~$F$.
By Lemma~\ref{lem:newlemma}, the Euler number $e'$ of the rank 2 bundle $\xi^\varphi$ associated with~$Y^\varphi$  is half that of $\xi$. In other words, $e'=x$.
We now have that $Y^\varphi$ is an Euler number $x$ circle bundle with base a surface of nonorientable genus~$h$, so the result follows in this closed case from Propositions~\ref{prop:HomologyY} and~\ref{prop:HomologyGroupsmathringY}.
For the statement about $\mathring{Y}^{\varphi}$, the reasoning is analogous, but simpler, as the homology computed in the punctured cases of Propositions~\ref{prop:HomologyY} and~\ref{prop:HomologyGroupsmathringY} are the same, no matter the parity of the Euler number.
\end{proof}

\subsection{The homology of the boundary of a surface exterior}
\label{sub:HomologyBoundarySurfaceExterior}

We fix the following notation for the remainder of this section.

\begin{notation}
\label{not:SurfaceNotation}
Let~$F \subseteq D^4$ be a~$\Z_2$-surface of nonorientable genus~$h$ and normal Euler number $e$.
Write~$X_F$ for the exterior of $F$ and $K:=\partial F$ for its boundary.
\end{notation}

\begin{proposition}
\label{prop:HomologyBoundaryXF}
The   nontrivial homology groups of $\partial X_F$ are
$$
H_0(\partial X_F)\cong \Z, \quad
H_1(\partial X_F)\cong \Z_2 \oplus \Z_2 \oplus \Z^{h-1},  \quad
H_2(\partial X_F)\cong \Z^{h-1},
\quad  \text{ and }\quad
 H_3(\partial X_F) \cong \Z.
$$
\end{proposition}
\begin{proof}
Since~$H_1(X_K)=\Z$ and $H_i(X_K)=0$ for~$i>1$, the Mayer-Vietoris sequence for the decomposition~$\partial X_F=\mathring{Y} \cup X_K$ shows that the homology of $\partial X_F$ is independent of the knot~$K$.
In particular we can assume without loss of generality that~$K$ is the unknot.
This implies that the Mayer-Vietoris calculation for~$\partial X_F=\mathring{Y} \cup X_K$
 calculates the homology of the circle bundle~$Y=\mathring{Y} \cup (S^1 \times D^2)$.
Since $e$ is even (recall Theorem~\ref{thm:MasseyBoundary}) the result now follows from Proposition~\ref{prop:HomologyGroupsmathringY}.
\end{proof}

Let $\widetilde{X}_F$ denote the universal cover of $X_F$ and let $\Sigma_2(K)$ denote the $2$-fold branched cover of a knot $K$.

\begin{proposition}
\label{prop:HomologyBoundaryUniversalCover}
Write the normal Euler number of $F$ as~$e=2x$.
The nontrivial homology groups of the~$\Z_2$-cover~$\partial \widetilde{X}_F$ are
\begin{align*}
H_0(\partial \widetilde{X}_F) &\cong  \Z, \quad
H_2(\partial  \widetilde{X}_F) \cong \Z^{h-1}, \quad
 H_3(\partial  \widetilde{X}_F) \cong \Z, \text{ and} \\
H_1(\partial  \widetilde{X}_F) &\cong \begin{cases}
\Z_4 \oplus \Z^{h-1} \oplus H_1(\Sigma_2(K)) &\quad \text{ if~$x$ is  odd}\\
\Z_2 \oplus \Z_2 \oplus \Z^{h-1} \oplus H_1(\Sigma_2(K)) &\quad \text{ if~$x$ is  even.}
\end{cases}
\end{align*}

Additionally the inclusion~$\mathring{Y} \hookrightarrow \partial X_F$ induces a $\Z[\Z_2]$-isomorphism
$$ H_2(\mathring{Y},\Z[\Z_2]) \xrightarrow{} H_2(\partial X_F;\Z[\Z_2]).$$
\end{proposition}

\begin{proof}
Use~$X_J^2 \to X_J$ to denote the~$2$-fold cover of the exterior of a knot $J$.
Additionally using~$\Sigma_2(J)$ to denote the $2$-fold branched cover of $J$,  a Mayer-Vietoris argument shows that~$H_1(X_J^2)=H_1(\Sigma_2(J)) \oplus \Z$ and~$H_2(X_J^2)=0$.

Note that $\partial \widetilde{X}_F=\mathring{Y}^\varphi \cup X_K^2$.
This has the same homology as $Y^{\varphi}$ from Proposition~\ref{prop:HomologyYCover}, except in degree one.
In our setting,  if we assume that~$|\det(K)|=1$, we see that~$X_K^2$ has the same homology as~$X_U^2=S^1 \times D^2$ where~$U \subseteq S^3$ denotes the unknot.
In particular, the Mayer-Vietoris sequence for~$\partial \widetilde{X}_F=\mathring{Y}^\varphi \cup X_K^2$ calculates the homology of~$Y^\varphi=\mathring{Y} \cup (S^1 \times D^2)$.
The result then follows from Proposition~\ref{prop:HomologyYCover}.
In general, when $|\det(K)| \neq 1$, the Mayer-Vietoris computation shows that one obtains an additional summand of $H_1(\partial  \widetilde{X}_F)$ isomorphic to $H_1(\Sigma_2(K))$. To see this, the relevant part of the sequence is
\[H_1(\partial X^2_K) \to H_1(\mathring{Y}^{\varphi}) \oplus H_1(X^2_K) \to H_1(\partial \wt{X}_F) \to 0.\]
Since $H_1(X^2_K) \cong \Z \oplus H_1(\Sigma_2(K))$ and the map $H_1(\partial X^2_K) \to H_1(X^2_K) \to  H_1(\Sigma_2(K))$, given by inclusion then projection, is the zero map, it follows that $H_1(\partial \wt{X}_F) \cong H_1(Y^{\varphi}) \oplus H_1(\Sigma_2(K))$, as asserted.
\end{proof}

\subsection{The homology of exteriors of~$\Z_2$-surfaces}
\label{sub:HomologyExteriors}

We continue with the conventions from Notation~\ref{not:SurfaceNotation}.

\begin{proposition}
\label{prop:ExteriorHomology}
The   nontrivial integral homology groups of~$X_F$ are
\[
H_0(X_F)\cong\Z,\quad H_1(X_F)\cong\Z_2,\quad  \text{ and }\quad H_2(X_F)\cong\Z^{h-1}.
\]
\end{proposition}
\begin{proof}
Since~$X_F$ is connected we have~$H_0(X_F)=\Z$ and since~$X_F$ is a~$4$-manifold with boundary, we know that~$H_4(X_F)=0$.
The fact that~$X_F$ is a~$\Z_2$-surface implies that~$H_1(X_F)=\Z_2$ and that the inclusion induced map~$\pi_1(\partial X_F) \to \pi_1(X_F)$ is surjective.
We deduce that~$H_1(X_F,\partial X_F)=0$.
Poincar\'e duality and the universal coefficient theorem then imply that~$H_3(X_F)=0$ and that~$H_2(X_F)$ is free.
Since~$\mathring{Y}$ is a~$3$-manifold with torus boundary,  it has vanishing Euler characteristic.
We deduce that
$$1=\chi(D^4)=\chi(X_F)+\chi(\overline{\nu}F)-\chi(\mathring{Y})=(1+b_2(X_F))+(1-h)-0$$
from which it follows that~$b_2(X_F)=h-1$.
Thus~$H_2(X_F)\cong \Z^{h-1}$.
\end{proof}

We record a lemma that will be helpful when calculating the $\Z[\Z_2]$-homology of $X_F$.
Recall that $\Z_{\pm}$ denotes $\Z$ considered as a $\Z[\Z_2]$-module where $T$ acts as multiplication by $\pm 1$.

 \begin{lemma}~
\label{lem:homologicalalgebra}
For~$k \geq 1$, we have
$\operatorname{Ext}^k_{\Z[\Z_2]}(\Z_\pm,\Z[\Z_2])=0$.
\end{lemma}
\begin{proof}
Consider the following projective resolution for $\Z_\pm$:
\begin{equation}
\label{eq:ProjectiveResolutionZpmbody}
\ldots \xrightarrow{1\pm T}\Z[\Z_2] \xrightarrow{1\mp T} \Z[\Z_2] \xrightarrow{1\pm T} \Z[\Z_2] \xrightarrow{1\mp T} \Z[\Z_2] \to \Z_\pm \to 0.
\end{equation}
Applying $\Hom_{\Z[\Z_2]}(-,\Z[\Z_2])$ and removing the rightmost term, we obtain the  chain complex
\[
\cdots\xleftarrow{1\pm T}\Z[\Z_2] \xleftarrow{1\mp T} \Z[\Z_2] \xleftarrow{1\pm T} \Z[\Z_2] \xleftarrow{1\mp T} \Z[\Z_2] \leftarrow 0.
\]
Taking the homology of this chain complex gives the result.
\end{proof}

\begin{proposition}
\label{prop:ZZ2HomologyExterior}
The nontrivial~$\Z[\Z_2]$-homology groups of~$X_F$ are
\[
H_0(X_F;\Z[\Z_2])\cong\Z_+\quad  \text{ and } \quad H_2(X_F;\Z[\Z_2])\cong\Z_- \oplus \Z[\Z_2]^{h-1}.
\]
\end{proposition}
\begin{proof}
Since $\widetilde{X}_F$ is connected, we certainly have~$H_0(X_F;\Z[\Z_2])=\Z_+$, and since $\widetilde{X}_F$ has nonempty boundary we have~$H_4(\widetilde{X}_F)=0$.
Since $\widetilde{X}_F$ is simply-connected, the fact that~$\pi_1(\partial X_F) \to \pi_1(X_F)$ is surjective implies that $H_1(\widetilde{X}_F,\partial \widetilde{X}_F)=0$.
Duality and the universal coefficient theorem now imply that $H_3(\widetilde{X}_F)=0$.
Our task is therefore to calculate~$H_2(X_F;\Z[\Z_2])$.
We first calculate~$H_2(X_F;\Z[\Z_2])=H_2(\widetilde{X}_F)$ as an abelian group.
Duality and the universal coefficient theorem imply that~$H_2(\widetilde{X}_F)$ is free abelian.
We deduce that~$1+b_2(\widetilde{X}_F)= \chi(\widetilde{X}_F)=2\chi(X_F)=2(1+b_2(X_F))$ and it follows that as an abelian group, we have
$$H_2(X_F;\Z[\Z_2])\cong \Z \oplus \Z^{2b_2(X_F)}=\Z \oplus \Z^{2(h-1)}.$$
We determine the~$\Z[\Z_2]$-module structure of~$H_2(X_F;\Z[\Z_2])$.

Next,  we deduce from a result of Wall~\cite[Lemma 5.2]{WallPoincare} that the double~$D:= X_F \cup_{\id} -X_F$ has~$H_2(D;\Z[\Z_2])\cong \Z_-\oplus \Z_- \oplus F$ with~$F$ free.
\begin{claim}
\label{claim:Double}
We have that
$H_2(D;\Z[\Z_2]) \cong H_2(X_F;\Z[\Z_2]) \oplus H_2(X_F;\Z[\Z_2])^*$.
\end{claim}
\begin{proof}

Using excision we have~$H_2(D,X_F;\Z[\Z_2])=H_2(X_F,\partial X_F;\Z[\Z_2])$.
Duality implies that $H_2(X_F,\partial X_F;\Z[\Z_2]) \cong H^2(X_F;\Z[\Z_2])$.
We now apply the universal coefficient spectral sequence (or UCSS for short) to $H^2(X_F;\Z[\Z_2])$.
We refer to~\cite[Theorem 2.3]{LevineKnotModules} for details about this spectral sequence but note that in this instance it takes the form
$$ E_2^{p,q}=\operatorname{Ext}^q_{\Z[\Z_2]}(H_p(X_F;\Z[\Z_2]),\Z[\Z_2]) \Rightarrow H^*(X_F;\Z[\Z_2])$$
with differentials of degree $(1-r,r).$
Using Lemma~\ref{lem:homologicalalgebra} and the UCSS, we deduce that $H^2(X_F;\Z[\Z_2])$ is isomorphic to~$H_2(X_F;\Z[\Z_2])^*$.
The long exact sequence of the pair~$(D,X_F)$ now gives
$$0 \to H_2(X_F;\Z[\Z_2]) \to H_2(D;\Z[\Z_2]) \to H_2(X_F;\Z[\Z_2])^* \to 0.$$
The fold map $D \to X_F$ is a retract and so its induced map on $H_2(-;\Z[\Z_2])$  splits this short exact sequence.
This concludes the proof of the claim.
\end{proof}

Using the fact that~$H_2(\widetilde{X}_F) \cong \Z \oplus \Z^{2(h-1)}$ is finitely generated as a~$\Z[\Z_2]$-module and torsion free as an abelian group, Lemma~\ref{lem:Reiner} implies
that as a~$\Z[\Z_2]$-module~$H_2(\widetilde{X}_F)$ decomposes as~$P \oplus P_{+1} \oplus P_{-1}$.
Here,~$P$ is~$\Z[\Z_2]$-free,~$P_{+1}$ has the~$\Z[\Z_2]$-action where $T$ operates by~$+1$, and~$P_{-1}$ has the~$\Z[\Z_2]$-action where~$T$ operates by~$-1$.
Since Lemma~\ref{lem:Bredon} gives~$\E_+(H_2(\widetilde{X}_F;\Q)) \cong H_2(X_F;\Q) \cong \Q_+^{h-1}$
and we know that~$H_2(D;\Z[\Z_2]) \cong \Z_-^2 \oplus F$ (with~$F$ free) decomposes as~$ H_2(X_F;\Z[\Z_2]) \oplus H_2(X_F;\Z[\Z_2])^*$
 (thanks to Claim~\ref{claim:Double}), this forces
$$H_2(X_F;\Z[\Z_2])\cong \Z_-\oplus \Z[\Z_2]^{h-1}.$$
Here we also used that $\Hom_{\Z[\Z_2]}(\Z_\pm,\Z[\Z_2]) \cong \Z_\pm$; recall Remark~\ref{rem:pm}.
This concludes the proof of the proposition.
\end{proof}


\subsection{The homology of branched covers of~$\Z_2$-surfaces}
\label{sub:HomologyBranched}

We continue with the conventions from Notation~\ref{not:SurfaceNotation}, and additionally we write~$\widetilde{F} \subseteq \Sigma_2(F)$ for the branch set in the $2$-fold branched cover.
The next proposition relates the homology of the branched and unbranched covers, and follows~\cite[p.~65]{KreckOnTheHomeomorphism}.

\begin{proposition}
\label{prop:BranchedUnbranchedv1}
There is an exact sequence of~$\Z[\Z_2]$-modules
\begin{equation}
\label{eq:BranchedUnbranchedSequencev1}
0\to H_2(\mathring{Y}^\varphi)  \to H_2(\widetilde{X}_F) \xrightarrow{j} H_2(\Sigma_2(F)) \xrightarrow{\partial}  \Z_2 \to  0.
\end{equation}
The~$\Z_2$ term has~$\Z[\Z_2]$-action given by extending the trivial~$\Z_2$-action linearly.
\end{proposition}

\begin{proof}
This proposition follows from a Mayer-Vietoris argument for~$\Sigma_2(F)$, as we now describe.
We have a decomposition~$\Sigma_2(F)=\widetilde{X}_F \cup_{\mathring{Y}^\varphi} \overline{\nu} \widetilde{F}$.
The gluing takes place along~$\mathring{Y}^\varphi\cong\partial \overline{\nu} \widetilde{F}$.
We assert that the Mayer-Vietoris sequence for~$\Sigma_2(F)=\widetilde{X}_F \cup_{\mathring{Y}^\varphi} \overline{\nu} \widetilde{F}$ reduces to
\begin{equation}
\label{eq:IntermediateMV}
0\to  H_2(\mathring{Y}^\varphi) \to H_2(\widetilde{X}_F) \to H_2(\Sigma_2(F)) \to H_1(\mathring{Y}^\varphi) \xrightarrow{\psi} H_1(\overline{\nu} \widetilde{F}) \to 0.
\end{equation}
For the leftmost side of the sequence, we used~$H_3(\Sigma_2(F)) \cong H^1(\Sigma_2(F),\partial \Sigma_2(F))=0$, where the latter holds because~$\Sigma_2(F)$ is simply-connected and $\Sigma_2(K)=\partial \Sigma_2(F)$ is connected.
For the rightmost zero,  we used that~$H_1(\partial \overline{\nu} \widetilde{F}) \cong H_1(\mathring{Y}^\varphi) \cong \Z_2 \oplus H_1(\widetilde{F})$ (recall Proposition~\ref{prop:HomologyYCover}) maps surjectively onto~$H_1(\overline{\nu} \widetilde{F})\cong H_1(\widetilde{F})$,  and we used  that~$H_1(\widetilde{X}_F)=0$.
The assertion now follows from the fact that~$H_2(\overline{\nu} \widetilde{F})\cong H_2(\widetilde{F})=0$.

Next we explain how to go from the exact sequence~\eqref{eq:IntermediateMV} to the short exact sequence~\eqref{eq:BranchedUnbranchedSequencev1} in the statement of the proposition.
To study the rightmost side of the sequence, we consider the following portion of the exact sequence of the pair $(\overline{\nu} \widetilde{F},\mathring{Y}^\varphi)$:
\[ \cdots \to H_2(\overline{\nu} \widetilde{F},\mathring{Y}^\varphi) \to H_1(\mathring{Y}^\varphi)  \xrightarrow{\psi} H_1(\overline{\nu} \widetilde{F})\to H_1(\overline{\nu} \widetilde{F},\mathring{Y}^\varphi) \to \cdots\]
Here $\overline{\nu} \widetilde{F}$ is what we called $\mathring{N}$ in the proof of Proposition~\ref{prop:HomologyY}.
In that proposition, using excision, we saw that $H_i(\mathring{N},\mathring{Y})$ vanishes for $i=1$ and that for $i=2$ we have that $H_2(\mathring{N},\mathring{Y})\cong\Z_2$, generated by the $D^2$-fibre of $\overline{\nu} \widetilde{F}$.
It follows that $\psi$ is surjective with $\ker(\psi)\cong\Z_2$ generated by the meridian to~$\widetilde{F} \subseteq \Sigma_2(F)$.

We have now established the exact sequence in~\eqref{eq:BranchedUnbranchedSequencev1}. We conclude by explaining the last assertion of the proposition: the~$\Z_2$ term in~\eqref{eq:BranchedUnbranchedSequencev1} has the trivial~$\Z_2$-action.
We have already argued that this~$\Z_2$ is generated  by the meridian to~$\widetilde{F} \subseteq \Sigma_2(F)$.
By definition of the branched cover, this meridian is~$(1+T)\widetilde{\mu}$, where~$\widetilde{\mu} \subseteq \widetilde{X}_F$ is a path lifting a meridian~$\mu \subseteq X_F$ to~$F \subseteq D^4$.
Since~$T(1+T)\widetilde{\mu} = (1+T)\widetilde{\mu}$, the generator is unchanged under the action of~$T$, so we obtain the required statement.
\end{proof}

Next, we use the exact sequence of Proposition~\ref{prop:BranchedUnbranchedv1} to deduce the homology of the~$2$-fold cover of~$D^4$ branched along a~$\Z_2$-surface~$F$.

\begin{proposition}
\label{prop:BranchedHomology}
The branched cover $\Sigma_2(F)$ is simply-connected and its   nontrivial homology groups are
\[
H_0(\Sigma_2(F))\cong \Z \quad \text{ and } \quad H_2(\Sigma_2(F))\cong \Z^h.
\]
\end{proposition}
\begin{proof}
The branched cover~$\Sigma_2(F)$ is simply-connected because $\widetilde{X}_F$ is simply-connected.
The fact that~$H_0(\Sigma_2(F))\cong\Z$ follows because~$\Sigma_2(F)$ is connected, and~$H_4(\Sigma_2(F))=0$ because~$\Sigma_2(F)$ has nonempty boundary.
 Since~$\Sigma_2(F)$ is simply-connected,~$H_1(\Sigma_2(F))=0$ and~$H_2(\Sigma_2(F))$ is free. It therefore suffices to prove that~$b_2(\Sigma_2(F))=h$. Recall from Proposition~\ref{prop:ExteriorHomology} that~$H_2(X_F;\Z[\Z_2])\cong\Z[\Z_2]^{h-1} \oplus \Z_-$ and therefore~$b_2(\widetilde{X}_F)=2(h-1)+1=2h-1$.
 Recall also from Proposition~\ref{prop:HomologyYCover} that~$H_2(\partial \overline{\nu} \widetilde{F}) \cong H_2(\mathring{Y}^\varphi)$ is free abelian of rank~$h-1$. The exact sequence in~\eqref{eq:BranchedUnbranchedSequencev1} now implies that~$b_2(\Sigma_2(F))=(2h-1) - (h-1) = h$, as claimed.
\end{proof}

Recall that the homology groups of~$\Sigma_2(F)$ are endowed with the structure of a~$\Z[\Z_2]$-module by the action of the deck transformation.
We compute the~$\Z[\Z_2]$-module structure of~$H_2(\Sigma_2(F))$.

\begin{proposition}
\label{prop:BranchedHomology-module-structure}
There is an isomorphism~$H_2(\Sigma_2(F)) \cong \Z^{h}_-$ of~$\Z[\Z_2]$-modules.
\end{proposition}
\begin{proof}
The proof is very similar to that of Proposition~\ref{prop:ZZ2HomologyExterior}.
Namely, we first use fixed points to study~$H_2(\Sigma_2(F);\Q)$ and then use Lemma~\ref{lem:Reiner} to conclude.
\begin{claim}
\label{claim:KrecknicalLemma}
There is an isomorphism~$H_2(\Sigma_2(F);\Q) \cong \Q^{h}_-$ of~$\Q[\Z_2]$-modules.
\end{claim}
\begin{proof}
We showed in Proposition~\ref{prop:BranchedHomology} that~$H_2(\Sigma_2(F);\Q) \cong \Q^{h}$ as a vector space over~$\Q$.
Lemma~\ref{lem:Bredon} then implies that
$$H_2(\Sigma_2(F);\Q) \cong H_2(D^4;\Q) \oplus \mathcal{E}_-(H_2(\Sigma_2(F);\Q)) \cong  \mathcal{E}_-(H_2(\Sigma_2(F);\Q)) .$$
It follows that $H_2(\Sigma_2(F);\Q) \cong \Q_-^{h},$ as claimed.
\end{proof}
We now conclude the proof of Proposition~\ref{prop:BranchedHomology-module-structure}.
Using the fact that~$H_2(\Sigma_2(F)) \cong \Z^{h}$ is finitely generated as a~$\Z[\Z_2]$-module and torsion free as an abelian group,
Lemma~\ref{lem:Reiner} implies that as a~$\Z[\Z_2]$-module~$H_2(\Sigma_2(F))$ decomposes as~$P \oplus P_{+1} \oplus P_{-1}$.
Here,~$P$ is~$\Z[\Z_2]$-free,~$P_{+1}$ has the~$\Z[\Z_2]$-action where $T$ operates by~$+1$,  and~$P_{-1}$ has the~$\Z[\Z_2]$-action where~$T$ operates by~$-1$.
Arguing as in the proof of Proposition~\ref{prop:ExteriorHomology},
Claim~\ref{claim:KrecknicalLemma} implies that~$P=0$ and~$P_{+1}=0$. Therefore~$H_2(\Sigma_2(F))\cong\Z_-^h$, as required.
\end{proof}

\section{Intersection forms}
\label{sub:IntersectionForms}

In this section, we describe some facts about the various types of intersection forms associated with a~$\Z_2$-surface.
We continue with the notation from Section~\ref{sub:HomologyBoundarySurfaceExterior}, which we recall briefly.

\begin{notation}
Let~$F \subseteq D^4$ be a~$\Z_2$-surface of nonorientable genus~$h$ and normal Euler number $e$.
Write~$X_F$ for the exterior of $F$ and $K:=\partial F \subseteq S^3$ for its boundary.
\end{notation}

\subsection{The~$\Z$-intersection form of the branched cover}

\begin{proposition}
\label{prop:BranchedNonsingular}
The intersection form $Q_{\Sigma_2(F)}$ is nondegenerate. If $|\det(K)|=1$, then $Q_{\Sigma_2(F)}$ is moreover nonsingular.
\end{proposition}

\begin{proof}
Since $\Sigma_2(K)=\partial \Sigma_2(F)$ is a rational homology 3-sphere,  we know that~$H_2(\Sigma_2(K))=0$.
Thus $H_2(\Sigma_2(F)) \to H_2(\Sigma_2(F),\Sigma_2(K))$ is injective, and we deduce that $Q_{\Sigma_2(F)}$ is nondegenerate.
If in addition $|\det(K)|=1$, then $H_1(\Sigma_2(K))=0$ and it follows that $Q_{\Sigma_2(F)}$ is nonsingular.
\end{proof}

%


Recall that we say $e:= e(F)$ is \emph{extremal} if $e \in \{-2h+2\sigma(K), 2h+2\sigma(K)\}$, i.e.\ $e$ is one of the two extremal values in Theorem~\ref{thm:MasseyBoundary}.

\begin{proposition}
\label{prop:Indefinite}
Assume that~$|\det(K)|=1$.
The intersection form $Q_{\Sigma_2(F)}$ is definite if and only if the normal Euler number $e$ is extremal.
\end{proposition}
\begin{proof}
A nonsingular symmetric bilinear form is definite if and only if its rank equals the absolute value of its signature.
By Proposition~\ref{prop:BranchedNonsingular}, we know that $Q_{\Sigma_2(F)}$ is nonsingular.
Thus $Q_{\Sigma_2(F)}$ is definite if and only if $\operatorname{rk}(H_2(\Sigma_2(F))=h$ equals $|\sigma(\Sigma_2(F))|=|\sigma(K)-\smfrac{1}{2}e(F)|$.
By Theorem~\ref{thm:MasseyBoundary} this is equivalent to~$e$ being extremal.
\end{proof}

\begin{proposition}
\label{prop:Odd}
The intersection form~$Q_{\Sigma_2(F)}$ is odd.
\end{proposition}

\begin{proof}
Edmonds~\cite[Corollary~4]{Edmonds} proved that the branched cover~$\Sigma_2(F)$ is not spin.
We note that Edmonds states his proof in the smooth category, but that the proof may be adapted to work for a properly embedded surface in a $4$-manifold, replacing his discussions of smooth $\Z_2$ actions by locally linear $\Z_2$ actions.
Since~$\Sigma_2(F)$ is simply-connected, this implies that $Q_{\Sigma_2(F)}$ is odd.
\end{proof}

\begin{proposition}
\label{prop:IntersectionFormBranched}
Assume that $|\det(K)|=1$
and write $\sigma:=\sign(\Sigma_2(F)) = \sigma(K)-\smfrac{1}{2} e$.
If $h \leq 8$ or~$e$ is non-extremal, then
$$ Q_{\Sigma_2(F)} \cong (1)^a \oplus (-1)^b~$$
where $a,b \in \N_0$ are nonnegative integers such that~$h=a+b$ and~$\sigma=a-b$, which exist because~$h \equiv \sigma \mod{2}$ and $|\sigma| \leq h$.
\end{proposition}

\begin{proof}
We know from Propositions~\ref{prop:BranchedNonsingular}, \ref{prop:Odd}, and~\ref{prop:BranchedHomology} that $Q_{\Sigma_2(F)}$ is nonsingular, odd and has rank~$h$.
Assume that~$h \leq 8$ and that $e$ is extremal.
By Proposition~\ref{prop:Indefinite}, $Q_{\Sigma_2(F)}$ is definite.
Since~$Q_{\Sigma_2(F)}$ is odd and definite,  the classification of (low rank) definite, nonsingular symmetric bilinear forms implies that~
$Q_{\Sigma_2(F)} \cong \operatorname{sgn}(\sigma)(1)^{\oplus h}$~\cite[Chapter II, Lemma 6.2 and Remark~1]{HusemollerMilnor}.
We now assume that~$e$ is non-extremal so that by Proposition~\ref{prop:Indefinite}, $Q_{\Sigma_2(F)}$ is indefinite.
Since~$Q_{\Sigma_2(F)}$ is odd and indefinite,  the classification of odd, indefinite, nonsingular,  symmetric, bilinear forms~\cite[Theorem~II.4.3]{HusemollerMilnor} in terms of their rank and signature implies that~
$$Q_{\Sigma_2(F)} \cong (1)^{\oplus a} \oplus (-1)^{\oplus b},$$ where~$a+b=h$ is the rank of $Q_{\Sigma_2(F)}$ and~$a-b=\sigma$ is the signature of~$Q_{\Sigma_2(F)}$.
\end{proof}

\subsection{The~$\Z$-intersection form of the universal cover}

This section is concerned with $Q_{\widetilde{X}_F}$,  the~$\Z$-intersection form of the universal cover of $X_F$.
In fact, we are particularly interested in the radical $\operatorname{rad}(Q_{\widetilde{X}_F})$ and the form induced on the quotient $H_2(\widetilde{X}_F)/\operatorname{rad}(Q_{\widetilde{X}_F})$.
Here,  recall that~$\operatorname{rad}(Q_{\widetilde{X}_F})$ consists of those $x \in H_2(\widetilde{X}_F)$ such that $Q_{\widetilde{X}_F}(x,y)=0$ for all $y \in H_2(\widetilde{X}_F).$
\begin{proposition}
\label{prop:RadicalFormUniversalCover} The inclusion induced map~$i_*\colon H_2(\partial \widetilde{X}_F)\to H_2(\widetilde{X}_F)$ is injective.
In particular, this implies that, as an abelian group,  the radical of~$Q_{\widetilde{X}_F}$ is
\[
\operatorname{rad}(Q_{\widetilde{X}_F})=\im(i_* \colon H_2(\partial \widetilde{X}_F)\to H_2(\widetilde{X}_F))\cong H_2(\partial \widetilde{X}_F)\cong\Z^{h-1}.
\]
Additionally, the~$\Z[\Z_2]$-module structure on~$H_2(\widetilde{X}_F)/\operatorname{rad}(Q_{\widetilde{X}_F})$ induced by deck transformations is extended linearly from~$T \in \Z_2$,  acting by multiplication by~$-1$.
\end{proposition}

\begin{proof}
The adjoint of $Q_{\widetilde{X}_F}$ is given by the composition
$$ H_2(\widetilde{X}_F) \xrightarrow{j_*} H_2(\widetilde{X}_F,\partial \widetilde{X}_F) \xrightarrow{\PD,\cong} H^2(\widetilde{X}_F) \xrightarrow{\ev,\cong} \Hom_\Z(H_2(\widetilde{X}_F),\Z).$$
Thus
the radical of~$Q_{\widetilde{X}_F}$ is equal to~$\ker(j_*)=\im(i_*\colon H_2(\partial \widetilde{X}_F) \to H_2( \widetilde{X}_F))$.
Using Poincar\'e duality and the universal coefficient theorem, we see that~$H_3(\widetilde{X}_F,\partial \widetilde{X}_F)\cong H^1(\widetilde{X}_F)=0$.
The long exact sequence of the pair~$(\widetilde{X}_F,\partial \widetilde{X}_F)$ implies that
$i_* \colon H_2(\partial \widetilde{X}_F) \to H_2( \widetilde{X}_F)$ is injective.
As~$i_*$ is injective, this image is isomorphic to~$H_2(\partial \widetilde{X}_F)$, which by Proposition~\ref{prop:HomologyBoundaryUniversalCover} is isomorphic to~$\Z^{h-1}$.

Finally, the exact sequence
$$0\to H_2(\mathring{Y}^\varphi)  \to H_2(\widetilde{X}_F) \xrightarrow{j} H_2(\Sigma_2(F)) \xrightarrow{\partial}  \Z_2 \to  0$$
 from Proposition~\ref{prop:BranchedUnbranchedv1} shows that we can identify~$H_2(\widetilde{X}_F)/\operatorname{rad}(Q_{\widetilde{X}_F})$ with a~$\Z[\Z_2]$-submodule of~$H_2(\Sigma_2(F))$,  of index $2$ as an abelian subgroup.
 By Proposition~\ref{prop:BranchedHomology-module-structure} we have~$H_2(\Sigma_2(F))\cong \Z_-^h$.
\end{proof}

\begin{remark}
  In fact, as a $\Z[\Z_2]$-module, we have that $\operatorname{rad}(Q_{\widetilde{X}_F}) \cong \Z^{h-1}_+$. We omit the proof since we shall not require this information later.
\end{remark}

Write $Q_{\widetilde{X}_F}^{\nd}$ for the nondegenerate form induced by~$Q_{\widetilde{X}_F}$ on~$H_2(\widetilde{X}_F)/\operatorname{rad}(Q_{\widetilde{X}_F})$.
The next proposition refines Proposition~\ref{prop:BranchedUnbranchedv1} to include information on the intersection forms. The statement was given (without proof) on \cite[p.~65]{KreckOnTheHomeomorphism}.

\begin{proposition}
\label{prop:BranchedUnbranched}
There is an exact sequence of~$\Z[\Z_2]$-modules
\begin{equation}
\label{eq:BranchedUnbranchedSequence}
0\to \operatorname{rad}(Q_{\widetilde{X}_F}) \to H_2(\widetilde{X}_F) \xrightarrow{j} H_2(\Sigma_2(F)) \xrightarrow{\partial}  \Z_2 \to  0.
\end{equation}
The inclusion-induced map~$j \colon H_2(\widetilde{X}_F) \to H_2(\Sigma_2(F))$ induces an isometry~$Q_{\widetilde{X}_F}^{\nd} \cong Q_{\Sigma_2(F)}|_{\im(j)}$, and the~$\Z_2$ term has~$\Z[\Z_2]$-action given by extending the trivial~$\Z_2$-action linearly.

In particular the form~$Q_{\widetilde{X}_F}$ splits as
$$Q_{\widetilde{X}_F} \cong (0)^{(h-1)} \oplus Q_{\Sigma_2(F)}|_{\im(j)},$$
where $\im(j)$ is free abelian of rank $h$.
\end{proposition}
\begin{proof}
Proposition~\ref{prop:BranchedUnbranchedv1} established the exact sequence
$$
0\to H_2(\mathring{Y}^\varphi)  \to H_2(\widetilde{X}_F) \xrightarrow{j} H_2(\Sigma_2(F)) \xrightarrow{\partial}  \Z_2 \to  0.$$
We also proved in Proposition~\ref{prop:HomologyBoundaryUniversalCover} that the inclusion $\mathring{Y} \subseteq \partial X_F $ induces a $\Z[\Z_2]$-isomorphism $H_2(\mathring{Y}^\varphi) \to  H_2(\partial \widetilde{X}_F)$.
It now follows from Proposition~\ref{prop:RadicalFormUniversalCover} that
\[\im(H_2(\mathring{Y}^\varphi) \to H_2( \widetilde{X}_F))
=\im(H_2(\partial \widetilde{X}_F) \to H_2(\widetilde{X}_F))
=\operatorname{rad}(Q_{\widetilde{X}_F}).\]
This gives rise to the exact sequence~\eqref{eq:BranchedUnbranchedSequence}.
The sentence concerning the $\Z_2$ action was already proved in Proposition~\ref{prop:BranchedUnbranchedv1} and since $j$ is inclusion induced, it gives rise to an isometry $Q_{\widetilde{X}_F}^{\nd} \cong Q_{\Sigma_2(F)}|_{\im(j)}$, as asserted.
The fact that $\im(j) \cong \Z^h$ follows from~\eqref{eq:BranchedUnbranchedSequence} together with Proposition~\ref{prop:BranchedHomology}, according to which $H_2(\Sigma_2(F))\cong \Z^h$.

We conclude by proving the splitting of the intersection form.
Note that~$H_2(\Sigma(F))$ is free abelian, and therefore as a subgroup of this, so is~$H_2(\widetilde{X}_F)/\operatorname{rad}(Q_{\widetilde{X}_F})$. It follows that there is a splitting~$H_2(\wt{X}_F) \cong \operatorname{rad}(Q_{\widetilde{X}_F}) \oplus H_2(\wt{X}_F)/\operatorname{rad}(Q_{\widetilde{X}_F})$, and the form also splits since one summand is the radical.
\end{proof}

Next we describe~$Q_{\widetilde{X}_F}^{\nd}$ in terms of~$Q_{\Sigma_2(F)}$.
Recall that a properly embedded surface~$F \subseteq X$ in a~$4$-manifold~$X$ is \emph{characteristic} if~$[F] \in H_2(X,\partial X;\Z_2)$ is Poincar\'e dual to~$w_2(X) \in H^2(X;\Z_2)$.

\begin{lemma}
\label{lem:Characteristic}
If~$G \subseteq D^4$ is a properly embedded nonorientable surface, then the preimage $\widetilde{G} \subset \Sigma_2(G)$ of~$G$ under the branched covering map is characteristic.
The same assertion holds for closed nonorientable surfaces in $S^4$.
\end{lemma}

\begin{proof}
By Proposition~\ref{prop:Odd}, the intersection form of $\Sigma_2(G)$ is odd, and so $\Sigma_2(G)$ is not spin.
It is a folklore result that a properly embedded surface in a
compact, oriented $4$-manifold is characteristic if and only if the exterior of the surface is spin; see e.g.~\cite[Proposition~7.18]{FriedlNagelOrsonPowell} for a proof.
Since~$G \subseteq D^4$, and~$D^4$ is spin, the exterior~$X_G$ is spin, and thus so is the universal cover~$\widetilde{X}_G$. But~$\widetilde{X}_G$ is the exterior of~$\widetilde{G} \subseteq \Sigma_2(G)$, so the lemma is proved.
\end{proof}

\begin{example}
As an example, if~$G \subseteq S^4$ is unknotted then~$\Sigma_2(G)\cong \#^a \C P^2 \#^b \ol{\C P}^2$ for some~$a,b\in\Z_{\geq 0}$, and~$[\widetilde{G}]\in H_2(\Sigma_2(G);\Z_2)$ is the sum of the~$a+b$ standard classes~$[\C P^1]$, one from each connect summand.
It is not difficult to compute that this $[\widetilde{G}]$ is characteristic by verifying that $Q_{\Sigma_2(G)}(x,x)=Q_{\Sigma_2(G)}(x,[\widetilde{G}])$ mod $2$ for every $x \in H_2(\Sigma_2(G);\Z_2)$.
\end{example}

We return to our fixed $\Z_2$-surface $F$.

\begin{proposition}
\label{prop:MinusFormCalculated}
The inclusion~$j \colon \widetilde{X}_F\hookrightarrow \Sigma_2(F)$ induces an isometric injection
$$j \colon (H_2(\widetilde{X}_F)/\operatorname{rad}(Q_{\widetilde{X}_F}),{Q_{\widetilde{X}_F}^{\nd}})  \hookrightarrow (H_2(\Sigma_2(F)),Q_{\Sigma_2(F)}).$$
Moreover
\[
\im(j)=\{x\in H_2(\Sigma_2(F))\,\,|\,\, Q_{\Sigma_2(F)}(x,x)\equiv 0\mod 2\} \cong \Z^h.
\]
\end{proposition}

\begin{proof}
The first assertion was already proved in Proposition~\ref{prop:BranchedUnbranched}, so  we now prove the statement relating the image of~$j$ to mod~$2$ intersections in~$\Sigma_2(F)$.
We know from Proposition~\ref{prop:BranchedUnbranched} that~$ \Z^h\cong \im(j)=\ker(\partial)$.
As a consequence, the assertion reduces to proving that~$\partial(x)\equiv Q_{\Sigma_2(F)}(x,x) \pmod 2$ for every~$x \in H_2(\Sigma_2(F))$.
The definition of the connecting homomorphism in the Mayer-Vietoris exact sequence for~$\Sigma_2(F) = \widetilde{X}_F \cup \ol{\nu}\widetilde{F}$ implies that~$\partial(x)=Q_{\Sigma_2(F)}(x,[\widetilde{F}])\in \Z_2$.
By Lemma~\ref{lem:Characteristic}, the class~$[\widetilde{F}]\in H_2(\Sigma_2(F);\Z_2)$ is Poincar\'{e} dual to~$w_2(\Sigma_2(F))$, and is thus a characteristic vector for~$Q_{\Sigma_2(F)}$ over~$\Z_2$.
Thus~$\partial(x)\equiv Q_{\Sigma_2(F)}(x, [\widetilde{F}]) \equiv Q_{\Sigma_2(F)}(x,x) \pmod2$, which concludes the proof.
 \end{proof}

The next proposition applies Proposition~\ref{prop:MinusFormCalculated} to calculate~$Q_{\widetilde{X}_F}^{\nd}$ and its corresponding quadratic form~$\theta_{\widetilde{X}_F}^{\nd}$.
 The answer will depend on the following quadratic form,  defined for every~$n \in \Z$:
$$
X_n= \operatorname{sgn}(n)\left[\begin{pmatrix}
2&2&2&\ldots &2 \\
0&1&1&\cdots &1\\
0&0&\ddots&\ddots &\vdots \\
\vdots&\ddots&\ddots&\ddots &1\\
0&\ldots&0&0 &1
\end{pmatrix}\right] \in Q_+(\Z^n).$$
The symmetrisation of this form is definite, which is relevant to the next proposition.

Recall that $h$ is the rank of $H_2(\Sigma_2(F))$ and $\sigma$ is the signature of $Q_{\Sigma_2(F)}$.
It follows that $|\sigma| \leq h$ and that $h \equiv \sigma \mod{2}$ when $Q_{\Sigma_2(F)}$ is nonsingular.
Thus $h-|\sigma|$ is nonnegative and even, and if~$\sigma=0$ then $h$ is even.

\begin{proposition}
\label{prop:Calculate}
Assume that $|\det(K)|=1$ and write $\sigma:=\sign(\Sigma_2(F))$.
If~$h \leq 8$ or if~$e$ is non-extremal, then
$$\theta_{\widetilde{X}_F}^{\nd} \cong
\begin{cases}
2H_+(\Z) \oplus H_+(\Z)^{\oplus \frac{h}{2}-1}  &\quad \text{ if } \sigma=0,   \\
X_\sigma \oplus H_+(\Z)^{\oplus \frac{1}{2}(h-|\sigma|)} &\quad \text{ if } \sigma\neq0.
\end{cases}
$$
\end{proposition}
\begin{proof}
We calculate~$Q_{\widetilde{X}_F}^{\nd}$, as this entirely determines the quadratic form~$\theta_{\widetilde{X}_F}^{\nd}$.
We will use Proposition~\ref{prop:MinusFormCalculated}, which shows that the inclusion~$\widetilde{X}_F \subseteq \Sigma_2(F)$ induces an isometry between the form~$Q_{\widetilde{X}_F}^{\nd}$ and the restriction of~$Q_{\Sigma_2(F)}$ to
\begin{equation}\label{eqn:zero-mod-2}
  \im(j)=\lbrace x \in H_2(\Sigma_2(F)) \mid Q_{\Sigma_2(F)}(x,x)=0 \  \operatorname{ mod } \  2\rbrace.
\end{equation}
Assume $e$ is non-extremal.
Using Proposition~\ref{prop:IntersectionFormBranched}, and recalling that $a,b \in \N_0$ are defined to be the unique integers such that~$h=a+b$ and~$\sigma=a-b$,  the classification of odd indefinite nonsingular symmetric bilinear forms~\cite[Theorem~II.4.3]{HusemollerMilnor},  which says that such forms are isometric if and only if they have the same rank and signature, gives
\begin{equation}
\label{eq:BranchedAgain}
Q_{\Sigma_2(F)}\cong(1)^{\oplus a} \oplus (-1)^{\oplus b}
\cong
\begin{cases}
  \operatorname{sgn}(\sigma)I_{|\sigma|} \oplus H^+(\Z)^{\oplus \frac{1}{2}(h-|\sigma|)} & \quad \text{ if } \sigma \neq 0, \\
 (1) \oplus (-1) \oplus H^+(\Z)^{\frac{h}{2}-1} & \quad \text{ if } \sigma = 0.
 \end{cases}
\end{equation}
Indeed these forms have signature~$\sigma$, rank~$h$, and are indefinite and odd. Such symmetric bilinear forms are unique up to isometry.
When $e$ is extremal and $h \leq 8$ we proved in Proposition~\ref{prop:IntersectionFormBranched} that $Q_{\Sigma_2(F)}\cong\operatorname{sgn}(\sigma) (1)^{\oplus h}$ and, since $|\sigma|=h \neq 0$, the isomorphisms in~\eqref{eq:BranchedAgain} hold trivially.

We now calculate~$Q_{\Sigma_2(F)}|_{\im(j) \times \im(j)}$.
The hyperbolic summands of~$Q_{\Sigma_2(F)}$ lead to (the same number of) hyperbolic summands in~$Q_{\Sigma_2(F)}|_{\im(j) \times \im(j)}$. Using~\eqref{eqn:zero-mod-2},  this is because hyperbolics are even and therefore all squares are zero mod~$2$.

We first assume that~$\sigma=0$,  and analyse the restriction of~$B:=(1)\oplus (-1)$ to~$\im(j)$.
Use~$e_1,e_2$ to denote the canonical basis of~$\Z^2$. By~\eqref{eqn:zero-mod-2}, \[\im(j) = \{v_1e_1 + v_2e_2 \mid v_1 +v_2 \equiv 0 \mod{2}\}.\]
Whence observe that~$\im(j)$ admits~$2e_1,e_1+e_2$ as a basis.
Evaluating~$B$ on these basis vectors yields~$\bsm 4&2 \\ 2&0 \esm \cong 2H^+(\Z).$
Passing to quadratic forms yields the result in this case.

We conclude with the case~$\sigma \neq 0$.
We need to analyse the restriction of~$B':=\operatorname{sgn}(\sigma) I_{|\sigma|}$ to~$\im(j)$.
Use~$e_1,\ldots,e_{|\sigma|}$ to denote the canonical basis of~$\Z^{|\sigma|}$.  By~\eqref{eqn:zero-mod-2},
\[\im(j) = \Big\{\sum_{i=1}^{|\sigma|} v_ie_i \,\Big| \, \sum_{i=1}^{|\sigma|} v_i  \equiv 0 \mod{2}\Big\}.\]
  We can therefore compute  that~$\im(j)$ admits~$2e_1,e_1+e_2,e_1+e_3,\ldots,e_1+e_{|\sigma|}$ as a basis.
Evaluating~$B'$ on these basis vectors yields the matrix~$X_\sigma+X_\sigma^T$.
Passing to the quadratic forms yields the result in this case.
\end{proof}

\subsection{The~$\texorpdfstring{\Z[\Z_2]}{Z[Z/2]}$-intersection form of the exterior }

We describe the nondegenerate Hermitian form~$\lambda_{X_F}^{\nd}$ induced by the equivariant intersection~$\lambda_{X_F} \colon H_2(\widetilde{X}_F) \times H_2(\widetilde{X}_F) \to \Z[\Z_2]$ on the quotient~$H_2(\widetilde{X}_F)/\operatorname{rad}(\lambda_{X_F})$.

\begin{proposition}
\label{prop:InterectionFormExterior}
The radical of the equivariant intersection form~$\lambda_{X_F}$ of~$X_F$ equals that of the standard intersection form of~$\widetilde{X}_F$ under the isomorphism~$H_2(\wt{X}_F) \cong H_2(X_F;\Z[\Z_2])$. On the quotient~$H_2(\widetilde{X}_F)/\operatorname{rad}(\lambda_{X_F}) \cong \Z_-^h$ we have~$\lambda_{X_F}^{\nd}=(1-T)Q_{\widetilde{X}_F}^{\nd}$.
\end{proposition}
\begin{proof}
The same reasoning as in Proposition~\ref{prop:RadicalFormUniversalCover}, but now using that~$H_3(X_F,\partial X_F;\Z[\Z_2]) \cong H^1(X_F; \Z[\Z_2]) =0$, shows that the radical of~$\lambda_{X_F}$ is the image of~$H_2(\partial X_F;\Z[\Z_2])$ and this is precisely~$\operatorname{rad}(Q_{\widetilde{X}_F})$.
Use Proposition~\ref{prop:RadicalFormUniversalCover} again
 to observe that
\begin{equation}
\label{eq:ModOutTheSame}
H_2(X_F;\Z[\Z_2])/\operatorname{rad}(\lambda_{X_F})=
H_2(\widetilde{X}_F)/\operatorname{rad}(\lambda_{X_F})
=H_2(\widetilde{X}_F)/\operatorname{rad}(Q_{\widetilde{X}_F})
\cong \Z^h_-.
\end{equation}
The last part of the assertion now follows from a direct calculation using the definition of~$\lambda_{X_F}$. Indeed, for any pair of elements~$x,y \in H_2(\widetilde{X}_F)/\operatorname{rad}(\lambda_{X_F})$, we have
$$ \lambda_{X_F}^{\nd}(x,y)=Q_{\widetilde{X}_F}(x,y)+Q_{\widetilde{X}_F}(x,Ty)T = Q_{\widetilde{X}_F}(x,y)+Q_{\widetilde{X}_F}(x,-y)T =(1-T)Q_{\widetilde{X}_F}^{\nd}(x,y).$$
This concludes the proof of Proposition~\ref{prop:InterectionFormExterior}.
\end{proof}

\section{A spin union of surface exteriors}
\label{sec:Spin}

We will show that if two $\Z_2$-surfaces in $D^4$ with the same connected, nonempty boundary have the same nonorientable genus and normal Euler number, then it is possible to glue their exteriors together so that the union is both spin and has fundamental group $\Z_2$.

\begin{notation}
In this section, it will be necessary to be particularly careful with tubular neighbourhoods.
We will use the notation $\nu_{A\subseteq B}$ for the normal vector bundle of a submanifold $A\subseteq B$ (as opposed to $\nu A\subseteq B$, which is an open tubular neighbourhood).
We will assume that every surface~$F \subseteq D^4$ comes equipped with a choice of tubular neighbourhood, i.e.\ a choice of continuous map~$\nu_{F \subseteq D^4} \to D^4$ from the total space of the normal vector bundle of $F$ whose restriction to the unit disc bundle is a homeomorphism to its image:
$$ \alpha \colon (D(\nu_{F \subseteq D^4}),S(\nu_{F \subseteq D^4}))  \xrightarrow{\cong}  (\alpha(D(\nu_{F \subseteq D^4})), \alpha(S(\nu_{F \subseteq D^4}))), $$
and which when composed with  the zero section gives the original embedding of $F$.
We refer to the disc bundle $\alpha(D(\nu_{F \subseteq D^4}))$ as a \emph{tubular neighbourhood} of  $F \subseteq D^4$.

Moreover, for each knot $K\subseteq S^3$, we choose a tubular neighbourhood and assume that, if a surface $F\subseteq D^4$ has connected nonempty boundary $K\subseteq S^3$, then the tubular neighbourhood of $F$ extends the given one of $K$. In this way, if two such surfaces have the same connected boundary~$K$, their tubular neighbourhoods will agree precisely upon restriction to the boundary.
Throughout this section it will also be helpful to recall that~$\partial X_{F} \cong X_K \cup_{\partial \ol{\nu}K} \alpha(S(\nu_{F \subseteq D^4}))$.
Moreover, one should recall that if a manifold is orientable (resp.~spin), then so is every codimension~0 submanifold, and so is the boundary. In particular, as $D^4$ is spin so are $X_{F}$, the open tubular neighbourhood $\nu F$, and $\partial X_F$.
\color{black}
\end{notation}

The \emph{meridian} of a $\Z_2$-surface $F \subseteq D^4$ refers to the homotopy class $T \in \pi_1(X_F)$ of the boundary of any $D^2$-fibre of the tubular neighbourhood~$\alpha(D(\nu_{F \subseteq D^4}))$, together with a basing path.
The meridian of~$F$
\begin{itemize}
\item generates $\pi_1(X_F) \cong H_1(X_F) \cong \Z_2$;
\item lies in the image of the inclusion induced map $\iota \colon \pi_1(\partial X_F) \to \pi_1(X_F)$.
\end{itemize}
We will also refer to an element $\mu \in \pi_1(\partial X_F)$ as a meridian if it satisfies $\iota(\mu)=T$.
Given two nonorientable surfaces $F_0,F_1 \subseteq  D^4$, a homomorphism $\varphi \colon \pi_1(\partial X_{F_0})  \to \pi_1(\partial X_{F_1})$ is said to \emph{preserve meridians} if there exists a meridian $\mu_0 \in \pi_1(\partial X_{F_0})$ with $\varphi \circ \iota_1  (\mu_0)=T_1$.

\begin{definition}
\label{def:NiceSection}
A section~$s \colon F \to S(\nu_{F \subseteq D^4})$ is \emph{nice} if the composition
\[
 H_1(F) \xrightarrow{{s}_*} H_1(S(\nu_{F \subseteq D^4})) \xrightarrow{\alpha_*} H_1(\partial X_F)\xrightarrow{\iota_*} H_1(X_F),
 \]
is trivial, where~$\iota\colon \partial X_F\to X_F$ is the inclusion.
\end{definition}

Recall from Definition~\ref{def:Extendable} that a homeomorphism $f \colon S(\nu F_0 ) \to S(\nu F_1 )$ that restricts to the identity map on $S(\nu K)$ is \emph{$\nu$-extendable rel.~boundary}
 if it extends to a homeomorphism $\ol{\nu} F_0  \cong \ol{\nu} F_1 $ of the closed tubular neighbourhoods that sends $F_0$ to~$F_1$,  and restricts to the identity map on~$\nu K$.
The main result of this section is the following.

\begin{proposition}\label{prop:UnionSpin}
Let~$F_0, F_1\subseteq D^4$ be two~$\Z_2$-surfaces of nonorientable genus~$h$, with the same connected nonempty boundary $K\subseteq S^3$, and each with normal Euler number~$e$. Fix a spin structure~$\s\in \Spin (\partial X_{F_1})$ that extends to $X_{F_1}$.
Then there exists an orientation preserving homeomorphism~$f\colon \partial X_{F_0} \to \partial X_{F_1}$ such that:
\begin{enumerate}
\item\label{item:1spin} the spin structure~$f^*\s\in\Spin (\partial X_{F_0})$ extends to~$X_{F_0}$;
\item\label{item:2spin} on $X_K\subseteq X_{F_0}$, the map $f$ restricts to the identity map $\Id_{X_K}$;
\item\label{item:3spin} on $\partial \ol{\nu} F_0 \sm \nu K = \alpha_0(S(\nu_{F_0 \subseteq D^4}) \subseteq X_{F_0})$, the map $f$ restricts to a homeomorphism  from 
$\alpha_0(S(\nu_{F_0 \subseteq D^4}))$ to $\alpha_1(S(\nu_{F_1 \subseteq D^4}))$,  that is $\nu$-extendable rel.~boundary;
\item\label{item:4spin} the map~$f_*\colon \pi_1(\partial X_{F_0})\to \pi_1(\partial X_{F_1})$ preserves meridians;
\item\label{item:5spin} there exist nice sections $s_0 \colon F_0 \to S(\nu_{F_0 \subseteq D^4})$ and $s_1 \colon F_1 \to S(\nu_{F_1 \subseteq D^4})$ which are preserved by
 $f$, in the sense that $f(\alpha_0(s_0(F_0))) = \alpha_1(s_1(F_1)) \subseteq \alpha_1(S(\nu_{F_1 \subseteq D^4}))$.
 \color{black}
\end{enumerate}
\end{proposition}

Some preparation is required for the proof of this proposition, which is the goal of this section.
 But, assuming the proposition, we record the main consequence.

\begin{corollary}
\label{cor:UnionPi1Z2}
The union~$X_{F_0}\cup_f -X_{F_1}$ is spin and has fundamental group~$\Z_2$.
\end{corollary}

\begin{proof}[Proof of Corollary~\ref{cor:UnionPi1Z2}]
It follows immediately from Proposition~\ref{prop:UnionSpin} that the union is spin.
To see that the fundamental group is as claimed, we apply the Seifert--van-Kampen theorem, computing the push-out
 of the diagram
\[
\begin{tikzcd}
\langle T_1 \mid {T_1}^2=1\rangle = \pi_1(X_{F_1})
&& \ar[ll, "(\iota_1\circ f)_*"'] \pi_1(\partial X_{F_0})\ar[rr, "(\iota_0)_*"]
&& \pi_1(X_{F_0}) = \langle T_0 \mid {T_0}^2=1\rangle,
\end{tikzcd}
\]
where~$\iota_j\colon \partial X_{F_j}\to X_{F_j}$ denote the inclusion maps.
To understand this diagram we consider the decomposition $\partial X_{F_0} \cong X_K \cup_{\partial \ol{\nu}K} \alpha_0(S(\nu_{F_0 \subseteq D^4}))$, which computes the fundamental group as the amalgamated free product
\[\pi_1(\partial X_{F_0}) \cong \pi_1(X_K) \ast_{\pi_1(\partial \ol{\nu}K)} \pi_1(S(\nu_{F_0 \subseteq D^4})).\]
We claim that an element of $\pi_1(\partial X_{F_0})$ maps trivially under~$(\iota_0)_*$ if and only if it maps trivially under~$(\iota_1\circ f)_*$.
Let $s_0 \colon F_0 \to S(\nu_{F_0 \subseteq D^4})$ be a nice section.
Recalling the standard cell structure for~$F$ from our standing conventions at the end of the introduction, we have closed curves $\gamma_1,\dots,\gamma_h$ on $F_0$.
Together with a meridian $\mu_{F_0}$, we have a generating set $\lbrace \mu_{F_0},s_0(\gamma_1),\dots,s_0(\gamma_h)\rbrace$ for~$\pi_1(\partial X_{F_0})$.
Since $s_0$ is a nice section, and since $f$ preserves nice sections, we have
$$(\iota_0)_*(s_0(\gamma_i)) = 1 \in \pi_1(X_{F_0}) \quad  \text{ and } \quad (\iota_1 \circ f)_*(s_0(\gamma_i)) = 1 \in \pi_1(X_{F_1}).$$
The meridian $\mu_{F_0}$ lies in the image $\pi_1(\partial \ol{\nu} K) \to \pi_1(S(\nu_{F_0 \subseteq D^4}))$.
Thus an element $g \in \pi_1(\partial X_{F_0})$ can be expressed as a word  in elements of $\pi_1(X_K)$ and the $s_0(\gamma_i)$.
Consider the word $\wt{g} \in \pi_1(X_K)$ obtained by removing all the instances of $s_0(\gamma_i)$ from $g \in \pi_1(\partial X_{F_0})$.
Since $s_0$ and $s_1$ are nice sections and $f$ preserves nice sections, the images $(\iota_0)_*(g) \in \pi_1(X_{F_0}) =\Z_2$ and $(\iota_1 \circ f)_*(g) \in \pi_1(X_{F_1}) = \Z_2$ are both computed by taking the image of $\wt{g}$ under the composition $\pi_1(X_K) \to H_1(X_K) \cong \Z \to \Z_2$, given by the Hurewicz map and reduction modulo two.
Thus each element of $\pi_1(\partial X_{F_0})$ maps trivially under~$(\iota_0)_*$ if and only if it maps trivially under~$(\iota_1\circ f)_*$,  as claimed.

As~$f$ is meridian preserving, $\mu_{F_0} \in \pi_1(\partial X_{F_0})$ is such that~$(\iota_0)_*(\mu_{F_0})=T_0$ and ~$(\iota_1 \circ f)_*(\mu_{F_0})=T_1$.
Hence in the push-out we have \[\pi_1(X_{F_0}\cup_f X_{F_1})\cong\langle T_0,T_1 \mid {T_0}^2=1, {T_1}^2=1, T_0=T_1\rangle\cong\Z_2,\] as desired.
\end{proof}

The focus for the proof of Proposition~\ref{prop:UnionSpin} is on proving condition~\eqref{item:1spin}; conditions~\eqref{item:2spin},~\eqref{item:3spin},~\eqref{item:4spin}, and~\eqref{item:5spin} will be achieved along the way. The idea for showing~\eqref{item:1spin} is intuitively straightforward but the argument is a little long, so some more informal discussion is in order before we begin.

A~$\Z_2$-surface exterior~$X_F$ is spin, and there are exactly two choices of spin structure
because~$H^1(X_F;\Z_2)=\Z_2$. A spin structure on a manifold can be described by its restriction to the 1-skeleton. The inclusion of the sphere bundle of the normal bundle $\alpha(S(\nu_{F \subseteq D^4}))\subseteq \partial X_F$ induces an isomorphism on homology with $\Z_2$-coefficients, so in fact a spin structure on $\partial X_F$ can be described by just by its restriction to the 1-skeleton of~$\alpha(S(\nu_{F \subseteq D^4}))$; this 1-skeleton consists of the meridian to $F$ and a push-off of the 1-skeleton of $F$. The two spin structures on $X_F$ differ on the meridian, in particular showing these are distinct spin structures on restriction to $\partial X_F$. There are $|H^1(F;\Z_2)|=2^h$ possible spin structures on the push-off of the 1-skeleton of $F$. As we shall see, exactly one of these is the restriction of the two spin structures on $\partial X_F$ that extend to $X_F$.
Thus, intuitively, the task is to produce a meridian-preserving homeomorphism~$f\colon \partial X_{F_0}\to \partial X_{F_1}$ that also preserves this distinguished spin structure on the surface push-off.
Having equal normal Euler numbers is what guarantees the last condition is possible.

We need the following definition.

\begin{definition}
\label{def:QuadraticRefinement}
Given a symmetric bilinear form~$(P,\lambda)$ on a finite dimensional~$\Z_2$-vector space~$P$, a \emph{$\Z_4$-quadratic refinement} is a function~$q\colon P\to \Z_4$ such that~$q(x+y)=q(x)+q(y)+2\lambda(x,y)$, for all~$x,y\in P$. We call the triple~$(P,\lambda, q)$ a \emph{$\Z_4$-quadratic form over~$\Z_2$}.
\end{definition}

Even though $\lambda(x,y) \in \Z_2$, we can make sense of $2\lambda(x,y) \in \Z_4$.
One quickly deduces from the condition on $q$ that $q(0)=0$ and that $2q(x) = 2\lambda(x,x) \in \Z_4$.

Here is a more detailed proof structure for Proposition~\ref{prop:UnionSpin}~\eqref{item:1spin}.

\begin{enumerate}
\item\label{item:summary-spin-proof-1} Let $F\subseteq D^4$ be a compact,  properly embedded surface with boundary $K\subseteq S^3$ and write~$\mathring{Y}$ for the total space of the sphere bundle of the normal bundle, so that $\partial X_F\cong X_K\cup \mathring{Y}$.
We show that a choice of section~$s\colon F\to \mathring{Y}$ can be used to specify a 1:1 correspondence between the set of spin structures on~$\partial X_F$ and the set of~$\Z_4$-quadratic refinements of the~$\Z_2$-intersection form on~$s(F)$, up to isometry, together with the restricted spin structure on the meridian.
Here, the~$\Z_4$-refinement $q_{KT}(\mathfrak{s})$ associated to a spin structure $\mathfrak{s}$ is due to Kirby-Taylor~\cite{KirbyTaylor}.
\item\label{item:summary-spin-proof-2} An embedded surface~$F\subseteq D^4$ has a canonical~$\Z_4$-quadratic refinement of the~$\Z_2$-intersection form of~$F$, called the \emph{Guillou-Marin quadratic refinement}~\cite{GuillouMarin}, and denoted $q_{GM}$.
We prove (Corollary~\ref{cor:sumitup}) that a spin structure $\mathfrak{s}$ on $\partial X_F$ extends to $X_F$ if and only if the Kirby-Taylor refinement on $s(F)\subseteq \partial X_{F}$ corresponding to $\mathfrak{s}$ coincides with the Guillou-Marin refinement on~$F\subseteq D^4$.
\item\label{item:summary-spin-proof-3}
Guillou-Marin proved (Corollary~\ref{cor:realiseisometry})
that if two surfaces $F_0,F_1\subseteq D^4$ with $\partial F_0 =\partial F_1$ have equal nonorientable genus and normal Euler number,  then there exists a homeomorphism~$F_0\cong F_1$ inducing an isometry of Guillou-Marin $\Z_4$-quadratic forms.
\item\label{item:summary-spin-proof-4} To complete the proof, we let $\mathfrak{s}$ be a spin structure on $\partial X_{F_1}$ that extends to $X_{F_1}$.
We show how to use Guillou-Marin's homeomorphism $F_0\cong F_1$ to construct a homeomorphism~$f\colon \partial X_{F_0} \to \partial X_{F_1}$. We show that
\[ q_{GM_0}=q_{GM_1}\circ f_*  =q_{KT_1}(\mathfrak{s}) \circ f_*=q_{KT_0}(f^*\mathfrak{s}).
\]
The first equality uses~\eqref{item:summary-spin-proof-3}, and the second uses~\eqref{item:summary-spin-proof-2}. Applying the reverse implication of~\eqref{item:summary-spin-proof-2} to $\partial X_{F_0}$ implies that
$f^*\mathfrak{s}$ extends over $X_{F_0}$.
\end{enumerate}

\subsection{The Kirby-Taylor quadratic refinements}
\label{sub:KirbyTaylor}
Given an abstract compact surface, we recall a construction of Kirby-Taylor~\cite[Definition 3.5]{KirbyTaylor} that relates a certain set of spin structures to the set of~$\Z_4$-quadratic refinements of the~$\Z_2$-intersection form on the surface.

\medbreak

Let~$\Sigma$ be a compact, connected surface.
Let~$\zeta=(\R\to E \to \Sigma)$ be a line bundle with~$w_1(\zeta)=w_1(\Sigma)$. Note that $T\Sigma\oplus \zeta$ admits a spin structure. To see this, consider an immersion $\iota \colon \Sigma \looparrowright \R^3$. Note that $\iota^*(T\R^3) = T\Sigma \oplus\nu_{\iota(\Sigma) \subseteq \R^3}$.
By the Whitney sum formula $w_1(\nu_{\iota(\Sigma) \subseteq \R^3}) = w_1 (T\Sigma)=w_1(\zeta)$, so the normal bundle is isomorphic to $\zeta$.  We then compute that $w_1(T\Sigma \oplus \zeta) = w_1(T\Sigma) + w_1(\zeta) =0$ and
\[0 = \iota^* (w_2(T\R^3)) = w_2(\iota^*(T\R^3)) = w_2(T\Sigma \oplus \nu_{\iota(\Sigma) \subseteq \R^3}) = w_2(T\Sigma \oplus \zeta).\]
So $T\Sigma \oplus \zeta$ admits a spin structure as desired. Fix a spin structure~$\mathfrak{s}$ on~$T\Sigma\oplus \zeta$.
Using the~$0$-section, identify~$\Sigma$ with a subspace of~$E$.
The reader will verify that there is a canonical bundle isomorphism~$\zeta\cong \nu_{\Sigma\subseteq E}$.
\footnote{For the reader comparing with~\cite[p.~203]{KirbyTaylor}, we note that they stipulate a choice of~spin structure on~$TE$ at this point. We think this is a minor error in that article, and they should rather stipulate a choice of spin structure on~$TE|_\Sigma$, the tangent bundle of the total space restricted to $\Sigma$, as this is what corresponds to a $\Pin^{-}$ structure on $T\Sigma$; see the proof of Theorem~\ref{thm:KTslick} below. This is then equivalent to our stipulation of a~spin structure on~$T\Sigma\oplus \zeta$, by the observation that~$\zeta\cong \nu_{\Sigma\subseteq E}$ and the splitting of~$TE|_\Sigma\cong T \Sigma \oplus \nu_{\Sigma\subseteq E}$.}

Let~$\gamma \subseteq \Sigma$ be an embedded circle.
Make a choice of orientation on~$\gamma$. This choice determines a framing~$T\gamma \cong \gamma \times \R$.
We now describe two ways of framing the rank three vector bundle~$T\Sigma|_\gamma \oplus \nu_{\Sigma \subseteq E}|_\gamma$.
The framing~$\mathcal{F}$ is defined to be the one determined by the choice of spin structure on~$T\Sigma \oplus \zeta$,  via the following sequence of identifications:
\begin{align*}
\mathcal{F} \colon T\Sigma|_\gamma \oplus \nu_{\Sigma\subseteq E}|_\gamma & \cong T\Sigma|_\gamma \oplus \zeta|_\gamma  \cong (T\Sigma \oplus \zeta)|_\gamma \cong  \gamma \times \R^3.
\end{align*}
The framing~$\mathcal{G}_f$ is obtained by combining the orientation-induced framing~$T\gamma \cong \gamma \times \R$ with some additional choice of a framing~$f \colon \nu_{\gamma \subseteq E} \cong \gamma \times \R^2$ in the following identifications:
\begin{align*}
\mathcal{G}_f \colon T\Sigma|_\gamma \oplus \nu_{\Sigma\subseteq E}|_\gamma
& \cong T\gamma \oplus \nu_{\gamma \subseteq \Sigma} \oplus \nu_{\Sigma\subseteq E}|_\gamma
\cong T\gamma \oplus \nu_{\gamma \subseteq E}
\cong \gamma \times \R^3.
\end{align*}

\begin{definition}
The framing~$f\colon \nu_{\gamma \subseteq E}\xrightarrow{\cong} \gamma\times \R^2$, above, is called \emph{odd} (with respect to the chosen spin structure~$\mathfrak{s}$) if~$\mathcal{G}_f$ determines the same orientation on~$\nu_{\gamma \subseteq E}$ as~$\mathcal{F}$, but the framings~$\F$ and~$\G_f$ disagree up to homotopy.
\end{definition}

\begin{definition}
\label{def:KirbyTaylor}
Let~$\Sigma$ be a surface, let~$\zeta=(\R\to E \to \Sigma)$ be a line bundle over~$\Sigma$ with~$w_1(\zeta)=w_1(\Sigma)$ and choose a spin structure~$\mathfrak{s}$ for~$T\Sigma\oplus \zeta$.
Let~$\gamma \subseteq E$ be an embedded circle. Choose an orientation on~$\gamma$ and an odd framing of~$\nu_{\gamma \subseteq E}$ with respect to~$\mathfrak{s}$.
With respect to these data, define \emph{the Kirby-Taylor quadratic refinement of~$\gamma$} as
$$q_{KT}(\mathfrak{s})(\gamma):= \#
\left\{\begin{array}{cc}
 \text{right-handed half-twists made by }\nu_{\gamma\subseteq \Sigma}\subseteq\nu_{\gamma \subseteq E}
\\
 \text{as it completes a full rotation around}~\gamma \text{ in the positive direction}
 \end{array}\right\} \in\Z_4.$$
\end{definition}

Note that a count of right-handed half-twists only makes sense once an orientation on $\gamma$ and a framing~$f\colon \nu_{\gamma \subseteq E}\cong\gamma\times \R^2$ have been chosen.

The Kirby-Taylor quadratic refinement depends neither on the choice of orientation on~$\gamma$, nor on the particular choice of odd framing; see~\cite[p.~203]{KirbyTaylor}. The dependency on the spin structure~$\mathfrak{s}$ is witnessed by the concept of an odd framing.

\begin{theorem}[Kirby-Taylor{~\cite{KirbyTaylor}}]\label{thm:KTslick}
The triple~$(H_1(\Sigma;\Z_2),Q_\Sigma,q_{KT}(\mathfrak{s}))$ is a~$\Z_4$-quadratic form over~$\Z_2$.
The assignment described in Definition~\ref{def:KirbyTaylor} determines a bijection of sets
\[
q_{KT} \colon \Spin(T\Sigma\oplus \zeta)\xrightarrow{1:1} \left\{\begin{array}{cc}\text{$\Z_4$-quadratic refinements of}\\ \text{$(H_1(\Sigma;\Z_2),Q_{\Sigma})$}\end{array}\right\}.
\]
\end{theorem}

\begin{proof}
As Kirby-Taylor do not state the result in the way we have, we collect the components of their work which yield our statement. First,~\cite[Theorem~3.2]{KirbyTaylor} is the statement that there is a natural 1:1 correspondence
\begin{equation}\label{eq:bijection}
\Pin^-(T\Sigma)\xrightarrow{1:1} \left\{\begin{array}{cc}\text{$\Z_4$-quadratic refinements of}\\ \text{$(H_1(\Sigma;\Z_2),Q_{\Sigma})$}\end{array}\right\},
\end{equation}
where the reader is referred to that article for the definition of a~$\Pin^-$ structure.
The natural correspondence is described on~\cite[p.~203]{KirbyTaylor}, and begins by using~\cite[Lemma~1.7]{KirbyTaylor} in which there is described a natural bijection~$\Psi_{4k+1}(\xi)\colon\Pin^-(\xi)\to \Spin(\xi\oplus(4k+1)\det(\xi))$, for any vector bundle~$\xi$ and integer~$k\geq 0$. We use~$\xi=T\Sigma$, and~$k= 0$.
We then note that as~$w_1(T\Sigma)=w_1(\zeta)$ and~$T\Sigma$ is orientable,  we obtain $w_1(\det(T\Sigma))=w_1(\zeta)$ and so there is an isomorphism~$\det(T\Sigma)\cong\zeta$.
We thus obtain a bijection of sets
$$\Pin^-(T\Sigma)\to\Spin(T\Sigma\oplus\zeta)$$ (depending only on the choice of isomorphism~$\det(T\Sigma)\cong\zeta$). The remainder of the construction in~\cite[p.~203]{KirbyTaylor} is the definition of~$q_{KT}(\mathfrak{s})$ we have given above and so the bijection of line \eqref{eq:bijection} factors as
\[
\Pin^-(T\Sigma)\xrightarrow{1:1} \Spin(T\Sigma\oplus\zeta)\xrightarrow{q_{KT}} \left\{\begin{array}{cc}\text{$\Z_4$-quadratic refinements of}\\ \text{$(H_1(\Sigma;\Z_2),Q_{\Sigma})$}\end{array}\right\}.
\]
It follows that~$q_{KT}$ is a bijection as claimed.
\end{proof}

\begin{remark}
We remark at this point that Kirby-Taylor do not explicitly allow for the possibility that the surface $\Sigma$ has boundary. However, we do explicitly allow this possibility. To justify this, we note that their proofs that we cited in the proof of Theorem~\ref{thm:KTslick} remain valid (without modification) when~$\Sigma$ has boundary.
\end{remark}

The following is a standard lemma and we omit the proof.

\begin{lemma}\label{lem:spin1skeleton}
Let $\zeta$ be a spin vector bundle over a space $B$. Let $i\colon K\to B$ be a map of spaces inducing an isomorphism $H_1(K;\Z_2)\cong H_1(B;\Z_2)$. Then the induced map on the sets of spin structures $\Spin(\zeta)\to \Spin (i^*\zeta)$ is a bijection.
\end{lemma}


We use the previous theorem by Kirby-Taylor in the proof of the following.

\begin{proposition}\label{prop:characterisation}
Let~$F$, $F_1$, and $F_2$ be nonorientable surfaces, with connected
nonempty boundary.
\begin{enumerate}
\item
Let~$\xi=(\R^2\to \mathring{E}\to F)$ be a
vector bundle with orientable total space.
Write~$\mathring{Y}$ for the total space of the associated~$S^1$-bundle~$S(\xi)$.
Fix a choice of section~$s\colon F\to \mathring{Y}$ and a choice of~$S^1$-fibre~$\mu \subseteq \mathring{Y}$.
There is a bijection of sets
\[
\begin{array}{rcl}
\Psi_{s,\mu}\colon\Spin(\mathring{Y})&\xrightarrow{1:1}& \Spin(T\mathring{Y}|_\mu)\times\left\{\begin{array}{cc}\text{$\Z_4$-quadratic refinements of}\\ \text{$(H_1(s(F);\Z_2),Q_{s(F)})$}\end{array}\right\}\\
\s&\mapsto& (\s|_{\mu},q_{KT}(\s)).
\end{array}
\]
To define $q_{KT}(\s)$, we take $\zeta := \nu_{s(F) \subseteq \mathring{Y}}$ and use the spin structure on $\mathring{Y}$ to induce a spin structure on $Ts(F) \oplus \nu_{s(F) \subseteq \mathring{Y}} \cong T\mathring{Y}|_{s(F)}$.
\item Let~$\Phi\colon\xi_1\cong \xi_2$ be an isomorphism of nonorientable 2-plane vector bundles with orientable total space, covering a homeomorphism~$\varphi \colon F_1 \to F_2$.
Write~$\mathring{Y}_i$ for the total space of the associated~$S^1$-bundles~$S(\xi_i)$. Fix a choice of section~$s_1\colon F_1 \to \mathring{Y}_1|_{F_1}$ and a choice of~$S^1$-fibre~$\mu_1 \subseteq \mathring{Y}_1$.
Write~$s_2=\Phi \circ s_1 \circ \varphi^{-1}$ and $\mu_2:=\Phi(\mu_1)$.
Then there is a commutative diagram
\[
\begin{tikzcd}
\Spin(\mathring{Y}_2)\ar[rr, "\Psi_{s_2,\mu_2}"]\ar[d, "{\Phi^*}"]
&&\Spin(T\mathring{Y}_2|_{\mu_2})\times \left\{\parbox{4.5cm}{\centering~$\Z_4$-quadratic refinements of\\~$(H_1(s_2(F_2);\Z_2),Q_{s_2(F)})$}\right\} \ar[d, "{\widehat{\Phi}}"] \\
\Spin(\mathring{Y}_1)\ar[rr, "\Psi_{s_1,\mu_1}"]
&&\Spin(\mathring{Y}_1|_{\mu_1})\times \left\{\parbox{4.5cm}{\centering~$\Z_4$-quadratic refinements of\\~$(H_1(s_1(F_1);\Z_2),Q_{s_1(F)})$}\right\},
\end{tikzcd}
\]
where the~$\Psi_{s_i,\mu_i}$ are the bijections of sets from the first item, and
\[
{\widehat{\Phi}}(\mathfrak{t}, q):= ((\Phi|_{\mu_2})^*\mathfrak{t},  q \circ \Phi_*).
\]
\end{enumerate}
\end{proposition}

\begin{proof}
We begin with the first item.
Take the standard cell structure for~$F$
(see the conventions at the end of the introduction),
with its basis of loops~$\gamma_1,\dots, \gamma_h$.
Let $A\subseteq \mathring{Y}$ be a connected 1-complex comprising a single $0$-cell, the~$S^1$-fibre loop~$\mu$, and the collection of loops~$s(\gamma_i)$.
Applying Lemma~\ref{lem:spin1skeleton} to the inclusion map of $A$, we obtain a bijection $\Spin(\mathring{Y})\to \Spin(T\mathring{Y}|_A)$.
Here we used Proposition~\ref{prop:HomologyY} and the universal coefficient theorem to compute the $\Z_2$-coefficient homology of $\mathring{Y}$.
Now, as~$A$ is a bouquet of circles, there is a bijection
\[
\Spin(T\mathring{Y}|_A)\xrightarrow{1:1} \Spin(T\mathring{Y}|_\mu)\times\prod_{i=1}^h\Spin(T\mathring{Y}|_{s(\gamma_i)}).
\]
Hence it now suffices to prove that the latter factor is in 1:1 correspondence with the~$\Z_4$-quadratic refinement factor from the statement of the proposition.

For this, first note that the collection of loops~$s(\gamma_i)$ form the~$1$-skeleton of~$s(F)$ and so, as~$s(F)$ is homotopy equivalent to this~$1$-complex, restriction determines a bijection
\[
\Spin(T\mathring{Y}|_{s(F)})\xrightarrow{1:1} \prod_{i=1}^h\Spin(T\mathring{Y}|_{s(\gamma_i)}).
\]
It is clear that there is a bijection between~$\Spin(T\mathring{Y}|_{s(F)})$ and~$\Spin(Ts(F)\oplus \nu_{s(F)\subseteq \mathring{Y}})$. We now wish to apply Theorem~\ref{thm:KTslick} using the surface~$\Sigma=s(F)$ and the line bundle~$\nu_{s(F)\subseteq \mathring{Y}}$. To see that this theorem is indeed applicable, observe that as~$\mathring{Y}$ is orientable, the product formula for Stiefel-Whitney classes implies that~$w_1(Ts(F))= w_1(\nu_{s(F)\subseteq \mathring{Y}})$.
Thus, by Theorem~\ref{thm:KTslick}, the Kirby-Taylor quadratic refinement determines a bijection of sets
\[
\Psi_{s,\mu} \colon \Spin(Ts(F)\oplus \nu_{s(F)\subseteq \mathring{Y}})\xrightarrow{1:1} \left\{\begin{array}{cc}\text{$\Z_4$-quadratic refinements of}\\ \text{$(H_1(s(F);\Z_2),Q_{s(F)})$}\end{array}\right\},
\]
and this completes the proof of the statement in the first item.

We will now prove the second claim, namely that~$\Psi_{s_1,\mu_1}\circ\Phi^*={\widehat{\Phi}}\circ \Psi_{s_2,\mu_2}$.
In other words, for all~$\s\in \Spin (\mathring{Y}_2)$, we will show
$
((\Phi^*\mathfrak{s})|_{\mu_1}, q_{KT_1}(\Phi^*\s)) = ((\Phi|_{\mu_2})^*(\s|_{\mu_2}), q_{KT_2}(\s)\circ \Phi_*).$
As~$\mu_2=\Phi(\mu_1)$, by definition, equality in the first argument is clear.
Given~$\s\in \Spin (\mathring{Y}_2)$,  we will now show the equality
\begin{equation}
\label{eq:EndSecondClaim}
q_{KT_1}(\Phi^*\s)=q_{KT_2}(\s)\circ \Phi_* \colon H_1(s_1(F_1)) \to \Z_4,
\end{equation}
which will complete the proof of the second claim.
Let $\gamma_1 \subseteq F_1$ be an embedded circle and let $\gamma_2=\varphi(\gamma_1)$ be the corresponding loop in $F_2$.
This way~$s_1(\gamma_1) \subseteq s_1(F_1)  \subseteq \mathring{Y}_1$ is an embedded circle and, by definition of $s_2=\Phi \circ s_1 \circ \varphi^{-1}$,  we deduce that~$\Phi(s_1(\gamma_1))=s_2(\gamma_2) \subseteq \mathring{Y}_2$ is similarly an embedded circle.
Pick a framing~$f_2 \colon \nu_{s_2(\gamma_2) \subseteq \mathring{Y}_2} \cong s_2(\gamma_2) \times \R^2$ that is odd with respect to the spin structure~$\mathfrak{s}$.
Since~$\Phi$ is a bundle isomorphism,  one checks that the framing defined by the composition
\[
\begin{tikzcd}
f_1 \colon  \nu_{s_1(\gamma_1) \subseteq \mathring{Y}_1}
\ar[r, "\Phi", "\cong"']
&\nu_{s_2(\gamma_2) \subseteq \mathring{Y}_2}
\ar[r, "f_2", "\cong"']
 &s_2(\gamma_2) \times \R^2
\ar[rr, "{\Phi^{-1}\times\Id_{\R^2}}","\cong"']
 &&s_1(\gamma_1) \times \R^2
 \end{tikzcd}
\]
is odd  with respect to the spin structure~$\Phi^*\mathfrak{s}$.
Choose an orientation on $\gamma_1$ and use the induced orientations on~$s_1(\gamma_1)$ and $s_2(\gamma_2)$. Now~$q_{KT_1}(\Phi^*\s)(s_1(\gamma_1))$ is calculated by counting right-handed half-twists of~$\nu_{s_1(\gamma_1)\subseteq s_1(F_1)} \subseteq \nu_{s_1(\gamma_1)\subseteq \mathring{Y}_1}$ with respect to the odd framing $f_1$. On the other hand, $q_{KT_2}(\s)\circ \Phi_*(s_1(\gamma_1))=q_{KT_2}(\s) (\Phi(s_1(\gamma_1)))$  is calculated
 by counting right-handed half-twists of~$\nu_{\Phi(s_1(\gamma))\subseteq s_2(F_2)}\subseteq\nu_{\Phi(s_1(\gamma)) \subseteq \mathring{Y}_2}$ with respect to the odd framing~$f_2$.
These counts agree thanks to our choice of framings. This establishes~\eqref{eq:EndSecondClaim} and therefore proves the second claim.
\end{proof}

\begin{remark}\label{rem:closedversion}
A similar statement to Proposition~\ref{prop:characterisation} holds when the surfaces are closed. To modify the statement to the case of $F$ closed, one must consider the punctured surface $\mathring{F}$ in order to choose the initial section $s\colon \mathring{F}\to \mathring{Y}|_{\mathring{F}}$ (similarly for $F_1$ and $F_2$). The remainder of the statement holds verbatim and the proof proceeds entirely similarly.
\end{remark}

\subsection{The Guillou-Marin quadratic refinement}
\label{sub:GuillouMarin}

Let $F \subseteq D^4$ be a properly embedded nonorientable surface with connected nonempty boundary.
We review the definition of the Guillou-Marin~\cite{GuillouMarin} quadratic refinement $H_1(F;\Z_2) \to \Z_4$, refining the~$\Z_2$-valued intersection form of~$F$.

Let~$\gamma$ be an embedded circle in the interior of~$F$ and let~$D\looparrowright D^4$ be an immersed disc with
$\partial D=\gamma$ and such that $D$ is \emph{transverse to $F$ along $\gamma$}; we refer the reader to the diagram on~\cite[p.15]{FreedmanQuinn} for the local model we are describing, where $D=B$, $\gamma=c$, and $F=A$, in their terminology.
Moreover, assume the interior of~$D$ meets $F$ transversely, and away from $\gamma$.
Choose an orientation on $\gamma$, and thus a positive tangent direction.
There are two possible choices of framing~$\nu_{D\subseteq D^4}\cong D\times \R^2$, one for each orientation of the fibres.
 Choose the framing with the property that, writing~$\R^2=\R\langle e_1, e_2\rangle$, the frame
\[
\text{(outwards-facing normal to $\gamma$ in $D$)}\times \text{(positive tangent to~$\gamma$)}\times e_1\times e_2\in \operatorname{Fr}(T_pD^4)
\]
agrees with the ambient orientation of $D^4$ at all points~$p\in\gamma$.

\begin{definition}\label{def:GuillouMarin}
Given an embedded loop~$\gamma\subseteq F$, choose an orientation on~$\gamma$ and a disc~$D$ as above. We then denote by~$W(D)$ the number of right-handed half-twists completed by the sub-bundle~$\nu_{\gamma\subseteq F}\subseteq \nu_{D\subseteq D^4}|_\gamma$ in one complete rotation around~$\gamma$.
Define\footnote{Guillou-Marin define $q_{GM}$ for closed surfaces in closed $4$-manifolds but their definition extends to properly embedded surfaces in $4$-manifolds with boundary, see e.g.~\cite{GilmerLivingston}.}
\[
q_{GM}(\gamma):=W(D)+2(D\cdot F)\in\Z_4.
\]
\end{definition}

\subsection{The exterior Guillou-Marin form, in the presence of a nice section}

In this section we show it is possible to compute the Guillou-Marin quadratic refinement of a $\Z_2$-surface using immersed discs whose interiors lie entirely in the complement of~$F$.
Before doing this,  we introduce a mechanism that will allow us to choose immersed discs so that their interiors, minus a boundary collar, miss the tubular neighbourhood.

\medbreak
Throughout this subsection, let~$F \subseteq D^4$ be a properly embedded surface with connected nonempty boundary $K\subseteq S^3$. Recall the standing convention that we have fixed a tubular neighbourhood~$\alpha\colon (D(\nu_{F \subseteq D^4}),S(\nu_{F \subseteq D^4}))\xrightarrow{\cong} (\alpha(D(\nu_{F \subseteq D^4})),\partial X_F \sm \mathring{X}_K)$, where~$\nu_{F \subseteq D^4}$ denotes the normal bundle of~$F$ and~$D(\nu_{F \subseteq D^4})$ and~$S(\nu_{F \subseteq D^4})$ respectively denote its disc and circle bundles.

\begin{lemma}
\label{lem:NiceExistence}
The bundle~$S(\nu_{F \subseteq D^4})$ admits a nice section.
\end{lemma}

\begin{proof}
Write~$h$ for the nonorientable genus of~$F$, write $\bigvee_{i=1}^h\gamma_i \subseteq F$ for the $1$-subcomplex that is a deformation retract of $F$, using the standard cell decomposition of $F$, and let $b_i:=[\gamma_i]\in H_1(F)$ be the corresponding homology generators.
Note that $w_1(F)(b_i) \neq 0$ for all $i=1,\dots,h$.
Choose a section~$s\colon F \to {S}(\nu_{F \subseteq D^4})$
and consider the composition
\[
 \Z\langle b_1,\dots,b_h\rangle\cong H_1(F) \xrightarrow{{s}_*} H_1(S(\nu_{F \subseteq D^4})) \xrightarrow{\alpha_*} H_1(\partial X_F)\xrightarrow{\iota_*}H_1(X_F),
 \]
 where~$\iota\colon \partial X_F\to X_F$ is the inclusion.
 By Proposition~\ref{prop:HomologyY} we may write this map as
 \begin{equation}\label{equation:composition-of-maps}
 \begin{tikzcd}
\Z\langle b_1,\dots,b_h\rangle\ar[r, "s_*"]
&\Z_2\langle \mu\rangle\oplus \Z\langle s_*(b_1),\dots,s_*(b_h)\rangle\ar[rr, "{\iota_*\circ\alpha_*}"]
&&\Z_2.
\end{tikzcd}
 \end{equation}
Here,~$\mu$ is a generator of the first homology of a fibre of~$S(\nu_{F \subseteq D^4})$.
 We now intend to define a new section~$s'\colon F\to S(\nu_{F \subseteq D^4})$, which will be the promised nice section.
 Because~$F$ is homotopy equivalent to a wedge of circles, it is sufficient to describe our new section individually on the 1-cells~$\gamma_i$.
  Recall that a nonorientable circle bundle~$\xi$ over a circle has exactly two sections up to isotopy.
The total space~$E(\xi)$ is a Klein bottle with homology~$H_1(E(\xi))\cong \Z\oplus\Z_2$, and with respect to some such choice of isomorphism, the two possible sections represent the classes~$(1,0)$ and~$(1,1)$. In our case, the choices of~$\gamma_i$ and~$\mu$ fix such an isomorphism and so the image in~$H_1(S(\nu_{F \subseteq D^4}))$ of the two possible sections of~$S(\nu_{F \subseteq D^4})|_{\gamma_i}$ are~$s_*(b_i)$ and~$s_*(b_i)+\mu$.
  For each~$i$ such that the composition in \eqref{equation:composition-of-maps}~$\iota_*\circ\alpha_*\circ s_*(b_i)$ is nontrivial, modify~$s|_{\gamma_i}$ to~$s'|_{\gamma_i}$ by switching to the other section of this nontrivial circle bundle.
  On all other~$\gamma_i$, i.e.\ where $\iota_*\circ\alpha_*\circ s_*(b_i)$ is trivial, set~$s|_{\gamma_i}=s'|_{\gamma_i}$. This determines the new section~$s'\colon F\to S(\nu_{F \subseteq D^4})$.

We now verify that~$s'$ is nice.
   Note that for~$i$ such that~$\iota_*\circ\alpha_*\circ s_*(b_i)$ is nontrivial, we have that~$s'_*(b_i)=s_*(b_i)+\mu\in H_1(S(\nu_{F \subseteq D^4}))$.
Thus~$\iota_*\circ\alpha_*\circ s'_*(b_i)$ is now trivial for all~$i$, and so~$s'$ is a nice section, as desired.
\end{proof}

Now let $F \subseteq D^4$ be a $\Z_2$-surface and suppose that we have chosen a tubular neighbourhood~$\alpha$ for~$F\subseteq D^4$, together with a nice section~$s\colon F\to S(\nu_{F \subseteq D^4})$.
For any embedded circle~$\gamma\subseteq s(F)$,  the facts that~$s$ is a nice section and that $F$ is a $\Z_2$-surface, imply~$\alpha(\gamma)$
is null-homotopic in~$X_F$. So there exists a properly immersed disc~$D'\looparrowright X_F$ that is transverse to $s(F)$ along $\partial D'=\alpha(\gamma)$.


\begin{definition}\label{def:Dprimeconvention}
Let~$\alpha$ be a tubular neighbourhood for~$F\subseteq D^4$, together with a nice section~$s\colon F\to S(\nu_{F \subseteq D^4})$.
Let~$D'\looparrowright X_F$ be a properly immersed disc that is transverse to~$s(F)$ along $\partial D'=\alpha(\gamma)$.
for some embedded loop~$\gamma\subseteq s(F)$.
Choose an orientation on~$\gamma$ and choose a framing~$\nu_{D'\subseteq X_F}|_{\alpha(\gamma)} \cong \alpha(\gamma) \times\R\langle e_1, e_2\rangle$ such that the frame
\[
\text{(outwards facing normal to $\alpha(\gamma)$ in~$D'$)}\times \text{(positive tangent to~$\alpha(\gamma)$)}\times e_1\times e_2\in \operatorname{Fr}(T_pX_F)
\]
agrees with the ambient orientation of~$X_F$ at all points~$p\in\alpha(\gamma)$.
Define~$W(D')$ as the number of right-handed half-twists of~$\nu_{\alpha(\gamma)\subseteq \alpha(s(F))}\subseteq \nu_{D'\subseteq X_F}|_{\alpha(\gamma)}$ in one complete loop around~$\gamma$ in the positive direction.
Define
\[
\widehat{q}_{GM}(\gamma):=W(D').
\]
As there is a basis of~$H_1(s(F);\Z_2)$ by oriented embedded loops, this extends linearly to a function $\widehat{q}_{GM}\colon H_1(s(F);\Z_2)\to\Z_4$, which we call the \emph{exterior} Guillou-Marin function.
The proof that this is well defined on $H_1(s(F);\Z_2)$ is similar to the one for the Guillou-Marin form; see Theorem~\ref{thm:GM}.
\end{definition}

The next proposition relates $\widehat{q}_{GM}$ to the Guillou-Marin form, and in particular shows that its isometry type does not depend on the choice of the tubular neighbourhood $\alpha$ of $F$.


\begin{proposition}\label{prop:GMarethesame}
Let~$F \subseteq D^4$ be a~$\Z_2$-surface,  and fix a nice section~$s\colon F\to S(\nu_{F \subseteq D^4})$.
Then the exterior Guillou-Marin quadratic function is a~$\Z_4$-quadratic refinement of the $\Z_2$-intersection form~$(H_1(s(F);\Z_2),Q_{s(F)})$.

Moreover, the isometry of $\Z_2$-intersection forms~$s_*\colon H_1(F;\Z_2)\to H_1(s(F);\Z_2)$ induces an isometry between the Guillou-Marin quadratic form and the exterior Guillou-Marin quadratic form.
\end{proposition}

\begin{proof}
Let~$\alpha$ be a tubular neighbourhood for~$F\subseteq D^4$, let $\delta \subseteq F$ be an embedded loop, and let~$D'\looparrowright X_F$ be a properly immersed disc that is transverse to $s(F)$ along~$\partial D'=\alpha(s(\delta))$.
For each~$p\in\delta$, write~$I_p$ for the radial line between~$p$ and~$s(p)$ in the fibre~$\nu_{F \subseteq D^4}|_p$. The union~$A:=\coprod_{p\in\delta} I_p$ is an embedded annulus in~$\nu_{F \subseteq D^4}$.
The union~$D:=D'\cup_{\alpha(s(\delta))}\alpha(A)$ may be used to compute~$q_{GM}(\delta)$.
The interior of the disc~$D$ does not intersect~$F$ and therefore~$q_{GM}(\delta)=W(D)$.
As~$\alpha(A)$ is an embedded annulus with trivial normal bundle in~$D^4$, and the framing conventions for computing~$W(D)$ and~$W(D')$ are the same, we also have~$W(D)=W(D')$. Thus,~$q_{GM}(\delta)=\widehat{q}_{GM}(s(\delta))$. This shows both that $\widehat{q}_{GM}$ is a~$\Z_4$-quadratic refinement of the $\Z_2$-intersection form on~$s(F)$ and that $s_*$ is indeed an isometry of~$\Z_4$-quadratic forms.
\end{proof}

There is an isometry between the (exterior) Guillou-Marin form for $F$ and the particular Kirby-Taylor forms that correspond to the spin structures on $\partial X_F$ that extend to $X_F$:

\begin{lemma}\label{lem:optimistic}
Let~$F\subseteq D^4$ be a~$\Z_2$-surface, and let~$s\colon  F\to S(\nu_{F \subseteq D^4})$ be a nice section. Write~$\mathring{Y}\subseteq \partial X_F$ for the boundary of the tubular neighbourhood of $F$.
Choose a spin structure~$\mathfrak{s}$ for~$\mathring{Y}$ that extends to a spin structure on~$X_F$.
Then there is an equality of~$\Z_4$-quadratic forms
\[
 ((H_1(s( F);\Z_2),Q_{s( F)},\widehat{q}_{GM}) = ((H_1(s( F);\Z_2),Q_{s( F)},q_{KT}(\mathfrak{s})),
\]
where the latter form~$q_{KT}(\mathfrak{s})$ is furnished by Proposition~\ref{prop:characterisation}.
\end{lemma}

\begin{proof}
To ease notation, in this proof we will write~$X:=X_F$ and $\Sigma:=s( F)$, and we will henceforth identify $\Sigma$ with $\alpha(\Sigma)$. Let~$\gamma\subseteq \Sigma$ be an oriented, embedded loop.
We unravel the definition of~$q_{KT}(\mathfrak{s})(\gamma)$ from Proposition~\ref{prop:characterisation}. In the notation of that section, set $\zeta=\nu_{\Sigma\subseteq \mathring{Y}}$.
The spin structure~$\mathfrak{s}$ on~$\mathring{Y}$ gives a framing of
\[
T \mathring{Y}|_{\Sigma}\cong T\Sigma\oplus \nu_{\Sigma\subseteq  \mathring{Y}}.
\]
This, together with the choice of orientation on~$\gamma$, is used to define the concept of an odd framing for the 2-plane bundle
\[
\nu_{\gamma\subseteq \Sigma}\oplus \nu_{\Sigma\subseteq  \mathring{Y}}\cong \nu_{\gamma\subseteq  \mathring{Y}}.
\]
Then~$q_{KT}(\mathfrak{s})(\gamma)$ is the number of right-hand half twists modulo~$4$ made by~$\nu_{\gamma\subseteq \Sigma} \subseteq  \nu_{\gamma\subseteq  \mathring{Y}}$ as it rotates around~$\gamma$, with respect to a choice of odd framing on this ambient~$2$-plane bundle.

We compare with the process for computing the exterior Guillou-Marin form.
Choose a properly immersed disc~$D'\looparrowright X_{F}$ for computing~$\widehat{q}_{GM}(\gamma)$.
Write
\begin{equation}
\label{eq:DeffForKT=GM}
f\colon \nu_{D'\subseteq X}|_{\gamma} \cong \gamma\times\R^2
\end{equation}
for the framing coming from restricting the framing on~$\nu_{D'\subseteq X}$ determined by the convention described in Definition~\ref{def:Dprimeconvention}.
Then~$\widehat{q}_{GM}(\gamma)$ is the number of right-hand half-twists modulo~$4$ in~$\nu_{\gamma\subseteq \Sigma} \subseteq \nu_{D'\subseteq X}|_\gamma$ with respect to~$f$.
Since $D'$ is transverse to~$s(F)$ along $\gamma$,  there is a bundle isomorphism~$\nu_{\gamma\subseteq  \mathring{Y}}\cong \nu_{D'\subseteq X}|_{\gamma}$.

As both $q_{KT}(\mathfrak{s})(\gamma)$ and~$\widehat{q}_{GM}(\gamma)$ are counts of right-handed half twists of~$\nu_{\gamma\subseteq \Sigma} \subseteq \nu_{D'\subseteq X}|_\gamma$, to show they are equal it suffices to prove that~$f$ is an odd framing.
In other words, we must show that the following two framings induce the same orientation, but disagree as framings:
\[
\F \colon T \mathring{Y}|_{\gamma} \cong \gamma \times \R^3,\qquad\text{and}\qquad \G_f \colon T \mathring{Y}|_{\gamma}  \cong \gamma\times\R^3,
\]
where~$\F$ is obtained by using the spin structure~$\mathfrak{s}$ on~$\mathring{Y}$ and, recall,~$\G_f$ is defined as the sequence of bundle isomorphisms
\[
T\mathring{Y}|_{\gamma} \cong T\gamma \oplus \nu_{\gamma \subseteq \mathring{Y}}
\cong T\gamma \oplus \nu_{D'\subseteq X}|_{\gamma}
 \cong \gamma \times \R^3.
\]

Writing~$n$ for the outwards-facing normal to $\gamma$ in~$D'$,~$t$ for the tangent to~$\gamma$, and~$\langle e_1, e_2\rangle$ for the frame specified by~$f$, the frame~$n\times t\times e_1\times e_2$ agrees with the ambient orientation of~$X$, by definition of~$f$. Using the ``outwards normal first'' convention, the orientation on~$X$ induces an orientation on~$\mathring{Y}$. Our frame includes~$n$ as the first vector already, so the frame~$t\times e_1\times e_2$, (i.e.~the frame specified by~$\G_f$) induces this ambient orientation on~$\mathring{Y}$. Now consider that~$\F$ comes from a spin structure on~$\mathring{Y}$, and this spin structure extends to a spin structure on~$X$, by hypothesis. It must be that the orientations underlying these spin structures agree with the orientations on~$X$ and~$ \mathring{Y}$. Hence the orientation underlying~$\F$ is the ambient orientation on~$\mathring{Y}$, and the orientations coming from~$\F$ and~$\G_f$ agree, as claimed.

Next, we show that ~$\F$ and~$\G_f$ disagree, as framings.
Using the outwards-facing normal~$n$, complete each of~$\F$ and~$\G_f$ to a framing of
\[
TX|_\gamma\cong\nu_{\gamma\subseteq D'}\oplus T \mathring{Y}|_\gamma.
\]
We refer to the completed framings as~$\F'$ and $\G_f'$, respectively.
We now show that these two framings disagree, as this will imply that~$\F$ and $\G_{f}$ also disagree.

The framing~$\F'$ extends to a framing of~$TX|_{D'}$ because the spin structure on~$\mathring{Y}$ used to define~$\alpha$ extends to~$X$ (this uses the ``outwards-facing normal first'' convention, again). The framing~$\G_f'$ splits as a framing of~$TD'|_\gamma\cong \nu_{\gamma\subseteq D'}\oplus T\gamma$ and of~$\nu_{D'\subseteq X}|_\gamma$. The latter of these framings is~$f$ and extends to~$\nu_{D'\subseteq X}$, by definition; recall~\eqref{eq:DeffForKT=GM}.
The former of these framings is given by~$n\times t$. An ordered pair of framings of a~$2$-plane bundle over a circle, which correspond to the same orientation, determines an element of~$\pi_1(\SO(2))$. Fix the (unique) framing that extends to a framing of~$\nu_{D'\subseteq X}$ as our reference framing; this framing determines the trivial class in $\pi_1(\SO(2))$. On the other hand, the framing~$n\times t$ determines a generator of~$\pi_1(\SO(2))\cong\Z$, and so determines the Lie framing.
 Therefore this latter framing does not extend to a framing of~$\nu_{D'\subseteq X}$.
Thus, overall,~$\G_f'$ does not extend to~$TX|_{D'}$. This proves that~$\F'$ and~$\G'_f$ are different framings of~$TX|_\gamma$.

But the framings~$\F'$ and~$\G'_f$ were obtained from~$\F$ and~$\G_f$ by the addition of the same line bundle (given by the outwards facing normal), from which we deduce that~$\F$ and~$\G_f$ are different framings.  This completes the proof that $f$ is an odd framing, which completes the proof of Lemma~\ref{lem:optimistic}.
\end{proof}

\begin{corollary}\label{cor:sumitup}
Let~$F \subseteq D^4$ be a~$\Z_2$-surface, and write $\mathring{Y}\subseteq \partial X_F$ for the boundary of the tubular neighbourhood of $F$.
Let $\s\in\Spin(\mathring{Y})$ be a spin structure and let $s$ be a nice section.
Then $\s$ extends to $X_F$ if and only if the quadratic refinement~$q_{KT}(\mathfrak{s})$ associated to the spin structure $\mathfrak{s}$ by Proposition~\ref{prop:characterisation} agrees with the exterior Guillou-Marin quadratic refinement.
\end{corollary}

\begin{proof}
Write $S$ for the set of spin structures on $\mathring{Y}$ that extend to $X_F$.
Write $T$ for the set of~$\mathfrak{t}\in \Spin(\mathring{Y})$ such that $q_{KT}(\mathfrak{t})(\gamma) = \widehat{q}_{GM}(\gamma)$ for every $\gamma \in H_1(s(F);\Z_2)$. The corollary is equivalent to the statement $S=T$.
%
  Lemma~\ref{lem:optimistic} shows that if $\mathfrak{t}$ extends to $X_F$ then we have an identification between $\widehat{q}_{GM}$ and $q_{KT}(\mathfrak{t})$, and so $S\subseteq T$.

 Now note that by Proposition~\ref{prop:characterisation}~(1), exactly two spin structures on $\mathring{Y}$ correspond to the Guillou-Marin quadratic refinement, so we have $|T|=2$. The corollary will be proved if we can show $|S|=2$.
 We know $|S| \leq |T| =2$ because $S \subseteq T$.
  As noted earlier there are two spin structures on $X_F$ and these induce at most two spin structures on $\mathring{Y}$.  We write $S=\{\s_1,\s_2\}$ where~$\s_1$ and~$\s_2$ are the restrictions of the two distinct spin structures on $X_F$ to $\mathring{Y}$.
 These spin structures on $X_F$ differ on their restriction to a meridian $\mu\subseteq X_F$, which is nontrivial in $H_1(\mathring{Y};\Z_2)$.
 Thus by Proposition~\ref{prop:characterisation}~(1) we have $\s_1\neq \s_2$ and so $|S|=2$ as required.
\end{proof}

A $\Z_4$-quadratic form $(P,Q,q)$ over $\Z_2$ has a Brown-Kervaire invariant $\beta(P,Q,q) \in \Z_8$; see e.g.~\cite{Brown,KirbyMelvin}.
We will need to use the following theorem.

\begin{theorem}[Guillou-Marin{~\cite{GuillouMarin}}, Kirby-Taylor{~\cite{KirbyTaylor}}, Gilmer-Livingston{~\cite{GilmerLivingston}}]
\label{thm:GM}
The Guillou-Marin function~$q_{GM}$ is well defined on homology and gives rise to a~$\Z_4$-quadratic refinement of the~$\Z_2$-intersection form~$(H_1(F;\Z_2);Q_F)$.
Moreover, if~$F$ has normal Euler number $e$ $($which is necessarily even by Theorem~\ref{thm:MasseyBoundary}$)$, then
\[
\beta(H_1(F;\Z_2),Q_F,q_{GM})\equiv -\frac{1}{2}e +4\operatorname{Arf}(\partial F)\pmod 8.
\]
\end{theorem}
\begin{proof}
None of the sources given above state the theorem exactly as we do, so we now discuss how this version arises from the literature.
Guillou-Marin~\cite{GuillouMarin} proved the theorem for closed nonorientable manifolds, only in the smooth category. Gilmer-Livingston~\cite[Theorem 7]{GilmerLivingston} proved the theorem, as stated above, but again only in the smooth category. The arguments of Guillou-Marin that their form is well-defined and gives rise to a~$\Z_4$-quadratic enhancement extend, without changes, to the topological setting, and the case that the surface has non-empty boundary, using the existence of normal bundles~\cite[Theorem~9.3A]{FreedmanQuinn} and topological transversality in dimension~4~\cite[\textsection9.5]{FreedmanQuinn}.

To extend the stated congruence to the topological category, begin by recalling that the Guillou-Marin formula for smooth \emph{closed} surfaces was explicitly generalised in~\cite[Corollary~9.3]{KirbyTaylor} to the topological category. To state the generalisation for our purposes, let $\Sigma\subseteq S^4$ be a closed surface and define a \emph{characteristic structure} (\cite[Definition~6.1]{KirbyTaylor}) to be a $\Pin^-$-structure on~$S^4\sm \Sigma$ that does not extend over $\Sigma$. Such structures are in 1:1 correspondence with~$H^1(S^4;\Z_2)=0$, so there is a unique such structure, in this case coming from the unique spin structure on $S^4$. This determines a~$\Pin^-$ structure on $\Sigma$ in a certain way described in \cite[Lemma~6.2]{KirbyTaylor}, and hence a $\Z_4$-quadratic refinement of~$(H_1(\Sigma;\Z_2),Q_{\Sigma})$, using~\eqref{eq:bijection}. Proposition~\ref{prop:GMarethesame} and Lemma~\ref{lem:optimistic} combine to show that this Kirby-Taylor refinement agrees with the Guillou-Marin refinement (this proposition and lemma are technically for surfaces with boundary, but can be applied to $(S^4,\Sigma)$, by first removing a ball pair~$(D^4,D^2)$). Given the discussion above, the result~\cite[Corollary~9.3]{KirbyTaylor} reads as
\begin{equation}\label{eq:KTcorrect}
2\cdot \beta(H_1(\Sigma;\Z_2),Q_\Sigma,q_{GM})\equiv -e(\Sigma) \pmod {16},
\end{equation}
which concludes the proof in the closed case.\footnote{In our statement of the congruence~\eqref{eq:KTcorrect}, we have corrected what we think is a sign error in Kirby-Taylor; our statement has the same signs as in Guillou-Marin and in Gilmer-Livingston. We think the sign error in~\cite[Theorem~9.3]{KirbyTaylor} arises from their computation in~\cite[Theorem~6.3]{KirbyTaylor}. We note that if the Kirby-Taylor sign is in fact correct, and~\eqref{eq:KTcorrect} has the wrong sign, this is a global change and will not affect our only application of Theorem~\ref{thm:GM}, in Corollary~\ref{cor:realiseisometry}, below. Indeed, we use it to show that if two $\Z_2$-surfaces have the same knot as boundary and same normal Euler number then they have the same Brown invariant, and this deduction is valid regardless of which sign is correct overall in this congruence.}

Finally, to extend to surfaces with boundary, the discussion of Gilmer-Livingston in~\cite[\textsection2.2.5]{GilmerLivingston} applies equally well in the topological category to adjust the formula from the closed case to the case of a properly embedded surface with boundary $F\subseteq D^4$. This concludes the proof.
\end{proof}

\begin{corollary}
\label{cor:realiseisometry}
If~$F_1,F_2\subseteq D^4$ have the same nonorientable genus, connected, nonempty boundary $K\subseteq S^3$, and normal Euler number, then there exists a homeomorphism~$f\colon F_1\to F_2$, restricting to the identity map on $K$, and inducing an isometry between the Guillou-Marin forms
\[
f_*\colon (H_1(F_1;\Z_2),Q_{F_1},q_{GM_1})\xrightarrow{\cong} (H_1(F_2;\Z_2),Q_{F_2},q_{GM_2}).
\]
\end{corollary}

\begin{proof}
Since both~$F_1$ and~$F_2$ have the same nonorientable genus, there is an isometry of intersection forms~$(H_1(F_1;\Z_2),Q_{F_1})\cong(H_1(F_2;\Z_2),Q_{F_2})$.
As the surfaces have the same normal Euler number and connected boundary $K$, Theorem~\ref{thm:GM} implies that their Guillou-Marin forms have the same Brown-Kervaire invariant.
As mentioned in~\cite[final remark of Section II]{GuillouMarin}, this implies there is an isometry of~$\Z_4$-quadratic forms~$(H_1(F_1;\Z_2),Q_{F_1},q_{GM_1})\cong (H_1(F_2;\Z_2),Q_{F_2},q_{GM_2})$; this remark concerns abstract non-degenerate quadratic forms and therefore applies since $\partial F_1=\partial F_2$ is connected.
Finally, it is shown in~\cite[Theorem 6.1]{GadgilPancholi}
that every isometry $(H_1(F_1;\Z_2),Q_{F_1})\cong (H_1(F_2;\Z_2),Q_{F_2})$ of intersection forms of compact,  nonorientable surfaces with connected nonempty boundary is induced by some homeomorphism between the surfaces that fixes the boundary pointwise.
\end{proof}

\subsection{The proof of Proposition~\ref{prop:UnionSpin}}
\label{subsec:unionisspin}

We begin with a lemma that will allow us to align choices of nice sections. We then move on to the proof of Proposition~\ref{prop:UnionSpin}.

Given a surface $\Sigma$ with connected, nonempty boundary and a bundle $\mathring{\xi}=(\R^2 \to \mathring{N} \to \Sigma)$, we write $\mathring{N}|_{\partial \Sigma}$ for the total space of $\mathring{\xi}|_{\partial \Sigma}$.
Since $[\partial \Sigma]=0\in H_1(\Sigma;\Z/2)$, it follows that $w_1(\mathring{\xi})([\partial \Sigma]) =0$, and therefore $\mathring{\xi}|_{\partial \Sigma}$ is a trivial bundle.

\begin{lemma}
\label{lem:NewNiceUniqueness}
   For $i=0,1$, let $F_i\subseteq D^4$ be nonorientable surfaces with the same nonorientable genus, and the same connected boundary $K\subseteq S^3$.
For brevity, write $\mathring{\xi_i}=(\R^2\to \mathring{N_i}\to F_i)$ for the respective normal bundles.
Suppose~$s_0,s_1 \colon F_i\to \mathring{N_i}$ are nonvanishing sections.
\begin{enumerate}
\item\label{item:constructiso} Given a homeomorphism $\varphi\colon F_0\to F_1$, restricting to the identity map on $K$, there exists a vector bundle isomorphism~$\Phi \colon \mathring{\xi_0}\to \mathring{\xi_1}$ covering~$\varphi$ and sending~$s_0$ to~$s_1$.
\item\label{item:modifyiso}  Suppose moreover that $s_0|_{K}$ and $s_1|_{K}$ are isotopic as sections of $\nu_{K\subseteq S^3}$. Then, after an isotopy of $s_1$  supported in a collar of $K\subseteq F_0$, the isomorphism $\Phi$ may be assumed to restrict to the identity on $\nu_{K\subseteq S^3}$.
\end{enumerate}
\end{lemma}
\begin{proof}
We prove the first item.
For $i=0,1$, the section~$s_i$ determines a (trivial) line bundle~$\zeta_i=(\R\to E_i\xrightarrow{\pi_i} F_i)$.
Namely~$\zeta_i$ is the subbundle of~$\mathring{\xi}_i$ with fibre over~$x \in  F_i$ given by~$\R\langle s_i(x) \rangle$.
Every~$v \in E_i$ can be written uniquely as $v=\lambda_v^i s_i(\pi_i(v)) \in (E_i)_{\pi_i(v)}$ for some~$\lambda_v^i \in \R$.
It follows that the assignment~$v \mapsto \lambda_v^0 s_1(\varphi(\pi_0(v)))$ defines
a bundle isomorphism $\alpha\colon \zeta_0\cong \zeta_1$, that covers $\varphi$ and
satisfies~$\alpha \circ s_0=s_1 \circ \varphi$.

We now extend $\alpha$ to a bundle isomorphism $\Phi\colon  \mathring{\xi}_0 \cong  \mathring{\xi}_1$, using orthogonal complements.
Choose Riemannian metrics on~$\mathring{\xi}_0$ and $\mathring{\xi}_1$ which agree over $K$, and consider the perpendicular bundles~$\zeta_0^\perp$ and~$\zeta_1^\perp$.
Given an
embedded loop $\gamma \subseteq  F_i$,
the line bundle
$\zeta_i^\perp|_\gamma$ is orientable or nonorientable according to whether~$w_1( \mathring{\xi}_i)$ is trivial on~$\gamma$ or not.  Since $\varphi$ is a homeomorphism, for $\gamma \subseteq F_0$ we have that $w_1(F_0)(\gamma) = w_1(F_1)(\varphi(\gamma))$. Therefore since the total space of $\mathring{\xi}_i$ is orientable, we have that $w_1(\mathring{\xi}_0)(\gamma) = w_1(\mathring{\xi}_1)(\varphi(\gamma))$.
Thus in both cases there is a bundle isomorphism~$\zeta_0^\perp|_\gamma \cong \zeta_1^\perp|_{\varphi(\gamma)}$.
Because line bundles are classified by $w_1$,
this implies there exists a bundle isomorphism~$\beta\colon \zeta_0^\perp\to \zeta_1^\perp$ covering $\varphi$.
Combining~$\alpha$ and $\beta$, we obtain a bundle isomorphism
\[
\Phi\colon  \mathring{\xi}_0 \xrightarrow{\cong} \zeta_0\oplus \zeta_0^\perp\xrightarrow{\alpha\oplus\beta} \zeta_1\oplus \zeta_1^\perp\xrightarrow{\cong} \mathring{\xi}_1,
\]
which covers~$\varphi$
and satisfies~$\Phi \circ s_0=s_1 \circ \varphi$.

We prove the second item.
The normal bundles $\mathring{\xi}_0$ and $\mathring{\xi}_1$ both restrict to $\nu_{K\subseteq S^3}$ over $K$, and recall that our choices of Riemannian metric for $\mathring{\xi}_0$ and $\mathring{\xi}_1$ were constructed to agree over~$K$.
An isotopy of $s_1$ over $K$ that makes it agree with $s_0$ thus determines an isotopy between the line bundles $\zeta_1|_K$ and $\zeta_0|_K$, as subbundles of $\nu_{K\subseteq S^3}$.
Also, because the metrics agree over the boundary, this isotopy induces an isotopy between $\zeta_1^\perp|_K$ and $\zeta_0^\perp|_K$ as subbundles of $\nu_{K\subseteq S^3}$.
This means we obtain an isotopy from $\Phi|_{\nu_{K \subseteq S^3}}$ to a bundle automorphism of $\nu_{K \subseteq S^3}$ that sends $\zeta_0|_K$ to itself and~$\zeta_0^\perp|_K$ to itself.
Since $\nu_{K \subseteq S^3} \cong \zeta_0|_K \oplus \zeta_0^\perp|_K$,  it follows that this automorphism is in fact the identity.
Taper the isotopies of $\Phi|_{\nu_{K \subseteq S^3}}$ and of $s_1$ over a collar neighbourhood of $K$ in $F_0$ and $F_1$, to obtain the desired isotopy of $\Phi$ to a new bundle isomorphism $\Phi' \colon \mathring{\xi}_0 \to \mathring{\xi}_1$ that is the identity over~$K$, and to obtain the claimed isotopy of $s_1$.
\end{proof}

\begin{lemma}\label{lem:NENcontribution}
    For $i=0,1$, let $F_i\subseteq D^4$ be nonorientable surfaces with the same nonorientable genus, the same connected boundary $K\subseteq S^3$ and each with normal Euler number $e$. Let $s_i\colon F_i\to \nu_{F_i\subseteq D^4}$ be nonvanishing sections. Then $s_0|_K$ and $s_1|_K$ are isotopic.
\end{lemma}

\begin{proof}
Recall that, by convention in this section, the tubular neighbourhood of any surface with boundary a given $K\subseteq S^3$ extends a given tubular neighbourhood for $K$. So there are assumed identifications $\nu_{F_i\subseteq D^4}|_K=\nu_{K\subseteq S^3}$ for $i=0,1$. Isotopy classes of nonvanishing sections of $\nu_{K\subseteq S^3}$ are in one-to-one correspondence with framings of $\nu_{K\subseteq S^3}\cong S^1\times\R^2$, and we will identify the two, using the same notation for a section and the corresponding framing.
Write $\tau\colon K\to \nu_{K\subseteq S^3}$ for a nonvanishing section corresponding to the Seifert framing of $K$. The set of framings is in affine correspondence with $\Z$, so for $i=0,1$, the difference between framings determines an integer $d(\tau,s_i|_K)\in\Z$. We claim that for $i=0,1$ there is an equality $d(\tau,s_i|_K)=e$. To see this, consider that if we add~$-k_i$ full turns to~$\tau$ to make the difference 0, and thus to make $\tau$ isotopic to~$s_i|_K$, this introduces~$-k_i$ signed intersections between a generic extension of~$\tau$ to~$F_i$ and the 0-section of~$\nu_{F_i\subseteq D^4}$.  As~$s_i|_K$ extends to a nonvanishing section, i.e.~$e(\nu_{F_i\subseteq D^4},s_i|_K)$=0, this implies that~$d(\tau,s_i|_K)=k_i=e(\nu_{F_i\subseteq D^4},\tau)=:e$ (using the interpretation of~$e(\nu_{F_i\subseteq D^4},\tau)$ as a signed intersection count from Remark~\ref{rem:intersectioncount}).

We thus have that
$d(\tau,s_1|_{\nu_{K}})=d(\tau,s_1|_{\nu_{K}})$,
from which we deduce
\[
d(s_0|_{\nu_{K}},s_1|_{\nu_{K}})=d(\tau,s_0|_{\nu_{K}})-d(\tau,s_1|_{\nu_{K}})=0.
\]
So $s_0|_{\nu_{K}}$ is isotopic to $s_1|_{\nu_{K}}$ as claimed.
\end{proof}

We can now prove the main result of this section.

\begin{proof}[Proof of Proposition~\ref{prop:UnionSpin}]

Since the surfaces~$F_0,F_1\subseteq D^4$ have the same normal Euler number and nonorientable genus, by Corollary~\ref{cor:realiseisometry} we can, and will, choose a homeomorphism~$\varphi\colon F_0\to F_1$ inducing an isometry of their Guillou-Marin forms and restricting to the identity map on~$K$.
The strategy for the remainder of the proof is to extend $\varphi$ to a homeomorphism $\nu_{F_0 \subseteq D^4} \cong \nu_{F_1 \subseteq D^4}$, and then to
restrict to $S(\nu_{F_0 \subseteq D^4}) \cong S(\nu_{F_1 \subseteq D^4})$. Finally, extending by the identity on $X_K$ leads to the sought-for homeomorphism $\partial X_{F_0} \cong \partial X_{F_1}$. Now, the details.

Recall that for~$i=0,1$,  we fixed a tubular neighbourhood
$\alpha_i$ for~$\nu_{F_i \subseteq D^4}$, so that~$\alpha_i(S(\nu_{F_i \subseteq D^4}))=\partial X_{F_i} \setminus \mathring{X}_K$.
By Lemma~\ref{lem:NiceExistence}, we may pick nice sections~$s_i \colon  F_i \to S(\nu_{F_i \subseteq D^4})|_{ F_i}$ with respect to this tubular neighbourhood.
Thus we may apply Lemma~\ref{lem:NewNiceUniqueness}~\eqref{item:constructiso} to obtain a bundle isomorphism
$$\Phi\colon \nu_{F_0 \subseteq D^4}\cong \nu_{F_1 \subseteq D^4}$$
covering~$\varphi$. Because the normal Euler numbers of $F_0$ and $F_1$ agree, Lemma~\ref{lem:NENcontribution} shows  $s_0|_K$ and $s_1|_K$ are isotopic as sections of $\nu_{K\subseteq S^3}$.
As $\varphi$ restricts to the identity map on $K$, this means we may apply Lemma~\ref{lem:NewNiceUniqueness}~\eqref{item:modifyiso}.
This achieves that $\Phi$ restricts to the identity map on $\nu_{K\subseteq S^3}$, at the expense of changing $s_1$ by an isotopy.
We will continue to call this modified section $s_1$, and we note that the new $s_1$ is still a nice section as this property is preserved by isotopy.


Our chosen tubular neighbourhoods $\alpha_0$ and $\alpha_1$ satisfy~$\alpha_i \colon  S(\nu_{F_i \subseteq D^4}) \hookrightarrow \partial X_{F_i}$ for~$i=0,1$, and are such that each embedding extends the given tubular neighbourhood of the knot~$S(\nu_{K \subseteq S^3}) \cong \partial X_K$.
We thus have a homeomorphism
\[
g:=\alpha_1\circ\Phi\circ (\alpha_0)^{-1}\colon \alpha_0(S(\nu_{F_0 \subseteq D^4}))\xrightarrow{\cong} \alpha_1(S(\nu_{F_0 \subseteq D^4})).
\]
As~$\Phi$ restricts to the identity bundle map over~$K$, this homeomorphism~$g$ restricts to the identity map on the tubular neighbourhood of the knot~$K$. Thus we may extend~$g$ by the identity map~$\id\colon X_K\to X_K$, to define a homeomorphism
\begin{equation}
\label{def:Homeof}
f\colon \partial X_{F_0}\xrightarrow{\cong} \partial X_{F_1}.
\end{equation}
We claim~$f$ has the properties listed in the statement of the proposition.

Property~\eqref{item:2spin}, that $f|_{X_K}$ is the identity, is clear from construction. Properties~\eqref{item:3spin} and~\eqref{item:4spin}, that~$f$ restricts to a map that is~$\nu$-extendable rel.~boundary, and that meridians are preserved, follow because~$f$ was constructed from the restriction of a vector bundle isomorphism $\nu_{F_0 \subseteq D^4} \cong \nu_{F_1 \subseteq D^4}$.
Property~\eqref{item:5spin}, which was about $f$ preserving nice sections,  is also automatic from the construction because we defined $f$ using choices of nice sections.

To finish the proof we prove that~$f$ has property~\eqref{item:1spin}.
Pick a spin structure~$\s\in\Spin(\partial X_{F_1})$ that extends to the exterior~$\Spin(X_{F_1})$.
Write~$\mathfrak{t}:=\alpha_1^*\s$ for the induced spin structure on~$S(\nu_{F_1 \subseteq D^4})$.
Choose an~$S^1$-fibre~$\mu_0\subseteq S(\nu_{F_0 \subseteq D^4})$ and write~$\mu_1:=\Phi(\mu_0)\subseteq S(\nu_{F_1 \subseteq D^4})$ for the corresponding~$S^1$-fibre of~$S(\nu_{F_1 \subseteq D^4})$.
Using the fact that~$\Phi$
satisfies~$\Phi \circ s_0=s_1 \circ \varphi$, the second item of Proposition~\ref{prop:characterisation} implies that
\begin{equation}
\label{eq:ApplyNaturalityProposition}
{\widehat{\Phi}}(\mathfrak{t}|_{\mu_1},q_{KT}(\mathfrak{t}))=((\Phi^*\mathfrak{t})|_{\mu_0}, q_{KT}(\Phi^*\mathfrak{t})).
\end{equation}
On the left-hand side, the quadratic refinement is, by definition~$q_{KT}(\mathfrak{t}) \circ \Phi_*$,
whereas on the right-hand side the quadratic refinement~$q_{KT}(\Phi^*\mathfrak{t})$ is assigned by the construction in Proposition~\ref{prop:characterisation}.
Both of these quadratic refinements are defined on~$H_1(s_0(F_0);\Z_2)$.

We claim that we have the following sequence of equalities of quadratic refinements:
\begin{equation}\label{eqn:sequence-equalities-quad-refinements}
  \widehat{q}_{GM_0}=\widehat{q}_{GM_1}\circ \Phi_*  =q_{KT_1}(\mathfrak{t}) \circ \Phi_*=q_{KT_0}(\Phi^*\mathfrak{t}).
\end{equation}
The last equality follows from~\eqref{eq:ApplyNaturalityProposition}, as just explained.
To see the central equality in~\eqref{eqn:sequence-equalities-quad-refinements} we use Lemma~\ref{lem:optimistic}, which implies that~$q_{KT_1}(\mathfrak{t})=\widehat{q}_{GM_1}$, the exterior Guillou-Marin form of~$s_1(F_1)$.

It remains to see that the isometry of $\Z_2$-forms $\Phi_*$ induces an isometry between the exterior Guillou-Marin forms~$\widehat{q}_{GM_0}$ and $\widehat{q}_{GM_1}$.
As~$\Phi$ satisfies~$\Phi \circ s_0=s_1 \circ \varphi$,
the isometry~$\Phi_*$ factors as a sequence of isomorphisms
\[
\Phi_*\colon H_1(s_0(F_0);\Z_2)\xrightarrow{(s_0)_*^{-1}} H_1(F_0;\Z_2)\xrightarrow{\varphi_*}H_1(F_1;\Z_2) \xrightarrow{(s_1)} H_1(s_1(F_1);\Z_2).
\]
By Proposition~\ref{prop:GMarethesame}, the outer two isomorphisms induce isometries between the~$\widehat{q}_{GM}$ forms on~$H_1(s_i( F_i);\Z_2)$ and the~$q_{GM}$ forms on~$H_1(F_i;\Z_2)$. One of the defining properties of the central isomorphism $\varphi_*$ is that it induces an isometry of~$q_{GM}$ forms.
Thus~$\Phi_*$ induces an isometry of exterior Guillou-Marin forms, concluding the proof of the claim.

We therefore have from~\eqref{eqn:sequence-equalities-quad-refinements} that~$\widehat{q}_{GM_0}=q_{KT_0}(\Phi^*\mathfrak{t})$.
From Corollary~\ref{cor:sumitup} it follows that~$f^*\s$, corresponding to the spin structure~$\Phi^*\mathfrak{t}$ under the tubular neighbourhood map $\alpha_0$, extends to~$X_{F_0}$. This completes the proof of property~\eqref{item:1spin} and therefore of Proposition~\ref{prop:UnionSpin}.
\end{proof}

\section{The odd-dimensional~$\ell$-monoid}
\label{sec:Monoid}

We recall the definition and some facts surrounding the
monoid~$\ell_{2q+1}(R)$.
We will introduce quasi-formations,  define~$\ell_{2q+1}(R)$,  study elementary quasi-formations and recall how a quadratic form~$v$ over~$R$ determines a subset~$\ell_{2q+1}(v,v) \subseteq \ell_{2q+1}(R)$.
At the end of the section, we focus on~$\ell_{2q+1}(\Z)$ and its relation to linking forms.

For the remainder of the section we fix an integer~$q$ and set~$\varepsilon:=(-1)^q$.
Since we will ultimately work with~$R=\Z$ and $\Z[\Z_2]$ which are both group rings of groups with vanishing Whitehead group,  we work in less generality  than in~\cite{KreckSurgeryAndDuality,CrowleySixt} in order to avoid technicalities related to bases and Whitehead torsion.
As a consequence,  throughout this  section we assume that~$R=\Z[\pi]$ is a group ring of a group $\pi$ with vanishing Whitehead group,  referring to~\cite[Section~3]{CrowleySixt} for the discussion in full generality.
The involution $x \mapsto \overline{x}$ on $R$ is $\sum n_g g \mapsto \sum n_g g^{-1}$.

\subsection{Quasi-formations}
\label{sub:quasi-formation}
We introduce quasi-formations and the $\ell$-monoid. The original reference for these monoids is~\cite{KreckSurgeryAndDuality}, but we also refer the reader to~\cite{CrowleySixt} as we will work with the equivalent formulation developed there.
It might help to refer back to Section~\ref{sub:QuadraticForms} on hermitian and quadratic forms.


\begin{definition}
\label{def:quasi-formation}
An~$\varepsilon$-quadratic \emph{quasi-formation} is a triple~$((P,\psi);F,V)$, where~$(P,\psi)$ is a nonsingular~$\varepsilon$-quadratic form,~$F\subseteq P$ is a lagrangian,  and~$V \subseteq P$ is a half rank direct summand; $P,F$ and $V$ are assumed to be stably free.
\end{definition}

This definition and terminology are due to Crowley and Sixt~\cite[Section 3]{CrowleySixt} but are inspired by work of Kreck~\cite{KreckSurgeryAndDuality}.

\begin{remark}
In both~\cite{KreckSurgeryAndDuality} and~\cite{CrowleySixt}, the authors work with stably based modules throughout, as this is needed in the most general applications of the theory. However, we have made the assumption that $R=\Z[G]$ where $G$ has vanishing Whitehead torsion, meaning we can avoid using based modules, as we now briefly explain. We assume some familiarity with the definition of the first reduced $K$-group $\widetilde{K}_1(R)$ and refer to e.g.~\cite[Section 3]{CrowleySixt} for more details and definitions.
An $R$-module $P$ is \emph{stably based} if $P$ is stably free and is equipped with an equivalence class of a basis for a free stabilisation $R^n \cong P \oplus R^k $. Two bases are defined to be equivalent if there is a common free stabilisation in which the change of basis matrix defines an element of~$Z:=\lbrace \pm g \mid g \in G \rbrace \subseteq \widetilde{K}_1(R)$. Assuming that $\operatorname{Wh}(G)=0$ implies that~$Z=\widetilde{K}_1(R)$, so that all stable bases are equivalent, and so do need not be kept track of. For this reason, we do not work with stably based modules,  merely stably free modules.
\end{remark}

We recall the relation between quasi-formations and formations.
The latter play an important r\^ole in classical surgery;
see e.g.~\cite[Section 9.2.1]{LueckMacko} for further background.

\begin{definition}
\label{def:Formation}
An $\varepsilon$-quadratic \emph{formation} is an $\varepsilon$-quadratic quasi-formation~$((P,\psi);F,G)$ for which the half rank direct summand~$G \subseteq P$ is a lagrangian.
\end{definition}

In some references,  e.g.\ in~\cite{RanickiExact},  a ``formation'' refers to a quasi-formation~$((P,\psi);F,G)$ with $G \subseteq P$ a sublagrangian, and what  we refer to as a formation is then called a nonsingular formation.
We do not follow this convention; our formations are always nonsingular.

The next definitions introduce further terminology related to quasi-formations and formations.

\begin{definition}~
\begin{itemize}
\item Two $\varepsilon$-quadratic quasi-formations~$((P,\psi);F,V)$ and~$((P',\psi'),F',V')$ are \emph{isomorphic} if there is an isometry~$f \colon (P,\psi) \to (P',\psi')$ such that~$f(F)=F'$ and~$f(V)=V'$.
\item An $\varepsilon$-quadratic quasi-formation is \emph{trivial} if for some $F$ it is isomorphic to a formation of the type~$(H_\varepsilon(F);F,F^*)$.
Here, and afterwards when it causes no confusion, we use the abbreviation $F,F^* \subseteq F \oplus F^*$ instead of $F \oplus 0, 0 \oplus F^* \subseteq F \oplus F^*$.
\item The \emph{boundary $\varepsilon$-quadratic quasi-formation~$\delta(F,\kappa)$ of a sesquilinear form~$(F,\kappa)$} is the quasi-formation $(H_\varepsilon(F),F,\Gamma_\kappa)$, where~$\Gamma_\kappa \subseteq F \oplus F^*$ denotes the graph of~$\kappa \colon F \to F^*$.
Note that the sesquilinear form~$\kappa$ need not be hermitian.
\item A quasi-formation is a \emph{boundary} if it is isomorphic to the boundary~$\delta(F,\theta)$ of some sesquilinear form~$(F,\kappa)$.
\item Two $\varepsilon$-quadratic quasi-formations are \emph{stably isomorphic} if they become isomorphic after some number of trivial $\varepsilon$-quadratic formations is added to each.
We use~$\cong$ to denote isomorphism and~$\cong_s$ to denote stable isomorphism.
\end{itemize}
\end{definition}

\begin{definition}~
\begin{itemize}
\item If a sesquilinear form~$(F,\kappa)$ is even (meaning, by definition,  that~$\kappa$ can be written as  $\theta-\varepsilon\theta^*$ for some $\theta$)
and~$(-\varepsilon)$-hermitian, then the graph $\Gamma_{\kappa} \subseteq F \oplus F^*$ is a lagrangian for $H_\varepsilon(F)$.
If~$(F,[\theta])$ is a~$(-\varepsilon)$-quadratic form, then $(H_\varepsilon(F),F,\Gamma_{\theta-\varepsilon\theta^*})$
is called the \emph{boundary $\varepsilon$-quadratic formation} of~$(F,[\theta])$.
Note that $\theta-\varepsilon\theta^*$ is tautologically even and is also $(-\varepsilon)$-hermitian.
\item An $\varepsilon$-quadratic formation is a \emph{boundary} if it is isomorphic to
 the boundary~$\delta(F,\theta)$ of some~$(-\varepsilon)$-quadratic form~$(F,[\theta])$.
\end{itemize}
\end{definition}
\color{black}

Every $\varepsilon$-quadratic quasi-formation~$((P,\psi);F,V)$ is isomorphic to one of the form~$(H_\varepsilon(F);F,V')$ for some half rank direct summand~$V' \subseteq F\oplus F^*$. 
To see this, consider that the Fundamental Lemma of $L$-theory (see e.g.~\cite[Lemma 8.95]{LueckMacko})
implies that the embedding~$(F,0) \hookrightarrow (P,\psi)$ extends to an isometry~$f \colon H_\varepsilon(F) \xrightarrow{\cong} (P,\psi)$, which therefore induces an isomorphism
$$(H_\varepsilon(F);F,f^{-1}(V)) \cong ((P,\psi);F,V).$$

\begin{definition}
\label{def:LMonoid}
The \emph{$\ell$-monoid}~$\ell_{2q+1}(R)$ is the unital abelian monoid  (under direct sum) of stable isomorphism classes of~$(-1)^q$-quadratic quasi-formations modulo the relation
\begin{equation}
\label{eq:RelationMonoid}
((P,\psi);F,G) \oplus ((P,\psi),G,V) \sim ((P,\psi);F,V),
\end{equation}
where~$F,G \subseteq P$ are both lagrangians.
\end{definition}

\begin{definition}\label{def:Lgroup}
The \emph{$L$-group}~$L_{2q+1}(R)$ is the group (under direct sum) of stable isomorphism classes of~$(-1)^q$-quadratic formations modulo the equivalence relation of cobordism, where $x$ and~$x'$ are said to be \emph{cobordant} if there are boundary formations $b$ and $b'$ such that $x \oplus b$ and~$x' \oplus b'$ are stably isomorphic.
\end{definition}

\begin{remark}
An alternative definition of the $L$-group $L_{2q+1}(R)$ is as the abelian group (under direct sum) of stable isomorphism classes of~$(-1)^q$-quadratic formations modulo the relation~\eqref{eq:RelationMonoid}; see e.g.~\cite[Remark 9.15]{RanickiAnIntroductionToAlgebraic}.  Thus $L_{2q+1}(R) \subseteq \ell_{2q+1}(R)$.
\end{remark}

\subsection{Elementary quasi-formations}
\label{sub:Elementary}

We recall the definition of elementary quasi-formations and the submonoid of $\ell_{2q+1}(R)$ that they define.
We continue with the notation $\varepsilon:=(-1)^q$ and, from now on,  write ``quasi-formation" (resp.~``formation")  instead of ``$\varepsilon$-quadratic quasi-formation" (resp.``~$\varepsilon$-quadratic formation").
References for the material include~\cite{KreckSurgeryAndDuality} and~\cite[Section~3]{CrowleySixt}.

\medbreak
We start by recording some terminology that we will use frequently in the sequel.
\begin{definition}
\label{def:IDislikeDefinitionEnvironmentsTerminologyIsBetter}
A \emph{direct complement} to a submodule $V \subseteq P$ is a submodule $U \subseteq P$ such that the inclusions $U,V \subseteq P$ extend to an isomorphism $U \oplus V \cong P$. We say that $U$ and $V$ are \emph{complementary}.
If $(P,\psi)$ is a quadratic form, then a \emph{lagrangian complement} to $V \subseteq P$ is a direct complement $U$ that is moreover a lagrangian.
\end{definition}

We introduce a geometrically significant notion of triviality for a quasi-formation.

\begin{definition}
\label{def:Elementaryquasi-formation}
 A quasi-formation~$((P,\psi);F,V)$ is \emph{elementary} if $F \subseteq P$ is a direct complement to $V \subseteq P$.
A class in~$\ell_{2q+1}(R)$ is \emph{elementary} if it admits an elementary representative.
\end{definition}

For an elementary quasi-formation $((P,\psi);F,V)$, since $F$ is by definition a lagrangian, it follows that $F$ is a lagrangian complement to $V$.
In particular, $V$ admits a lagrangian complement.
The next proposition proves a partial converse by working on the level of the monoid, modulo formations.
This proposition was proved by Kreck~\cite[Proposition 8 (iii)]{KreckSurgeryAndDuality} but we translate the argument into the language of quasi-formations.

\begin{proposition}
\label{prop:ElementaryModL2q+1}
If~$\Theta \in \ell_{2q+1}(R)$ is represented by a quasi-formation~$x=(H_\varepsilon(F);F,V)$ such that~$V$ admits a lagrangian complement, then there exists a formation~$y$ such that~$\Theta+[y]$ is elementary.
\end{proposition}
\begin{proof}
Let~$L \subseteq F\oplus F^*$ be a lagrangian complement for~$V \subseteq F\oplus F^*$.
Note that~$(H_\varepsilon(F);L,V)$ is elementary.
Consider the formation~$y:=(H_\varepsilon(F);L,F)$. Working in~$\ell_{2q+1}(R)$ and using~\eqref{eq:RelationMonoid},  we have
\begin{equation}
\label{eq:ElementaryModL5}
[(H_\varepsilon(F);L,V)]=[(H_\varepsilon(F);L,F)]+[(H_\varepsilon(F);F,V)]=[y]+[x]=[y]+\Theta,
\end{equation}
and thus $\Theta + [y]$
 is elementary as required.
\end{proof}

Proposition~\ref{prop:ElementaryModL2q+1} motivates considering the action of~$L_{2q+1}(R)$ on~$\ell_{2q+1}(R)$ by direct sum.

\begin{construction}
\label{cons:Action}
The action of~$[y] \in L_{2q+1}(R)$ on~$[x] \in \ell_{2q+1}(R)$ is  by~$[y] \cdot [x]:=[x\oplus y]$.
To verify this is well defined,  it suffices to prove that trivial formations and boundary formations act trivially.
For trivial formations, this is clear: in~$\ell_{2q+1}(R)$ stably isomorphic quasi-formations are equal.
To verify that boundary formations act trivially, use that in~$\ell_{2q+1}(R)$ a boundary formation~$y=(H_\varepsilon(F);F,\Gamma_{\theta-(-1)^q\theta^*})$ decomposes as~$y_1 \oplus y_2:=(H_\varepsilon(F);F,F^*) \oplus (H_\varepsilon(F);F^*,\Gamma_{\theta-(-1)^q\theta^*})$,  which is the sum of two trivial formations. To see that the latter is trivial one shows that $F^*$ and~$\Gamma_{\theta-(-1)^q\theta^*}$ are complementary.
Here is the proof that~$(0 \oplus F^*) \oplus \Gamma_{\theta-(-1)^q\theta^*}=F \oplus F^*$.
The fact that $(0 \oplus F^*) \cap \Gamma_{\theta-(-1)^q\theta^*}=0$ is verified by noting that if $(0,\varphi)=(x,(\theta-(-1)^q\theta^*)(x))$, then $x=0$ and therefore $\varphi=({\theta-(-1)^q\theta^*})(0)=0$.
The inclusion $(0 \oplus F^*) \oplus \Gamma_{\theta-(-1)^q\theta^*} \subseteq F \oplus F^*$ is clear and the reverse inclusion follows by writing $(x,\varphi) \in F \oplus F^*$ as $(0,\varphi-(\theta-(-1)^q\theta^*)(x)) +(x,({\theta-(-1)^q\theta^*})(x))$.
Now apply~\cite[Lemma~9.13~i)]{LueckMacko}, 
which states that a formation $((P,\psi);F,G)$ is trivial if $F \oplus G = P$. This completes the proof that $(H_\varepsilon(F);F^*,\Gamma_{\theta-(-1)^q\theta^*})$ is a trivial formation.
Thus in~$\ell_{2q+1}(R)$ we obtain~$x\oplus y \sim x\oplus y_1\oplus y_2 \sim x$, as required.
\end{construction}

The following proposition  describes an equivalent characterisation of $\Theta \in \ell_{2q+1}(R)$ being elementary when~$L_{2q+1}(R)=0$ cf.~\cite[Proposition 2]{HambletonKreckTeichnerNonOri}.

\begin{proposition}
\label{prop:ElementaryForTrivalL5}
Assume that~$L_{2q+1}(R)=0$.
Given~$\Theta \in \ell_{2q+1}(R)$, the following assertions are equivalent.
\begin{enumerate}
\item $\Theta$ is elementary, i.e.\ $\Theta$ admits a representative $(H_\varepsilon(F);F,V)$ such that $F$ is a lagrangian complement to $V$.
\item $\Theta$ is represented by a quasi-formation~$(H_\varepsilon(F);F,V)$ such that~$V$ admits a lagrangian complement.
\end{enumerate}
\end{proposition}

\begin{proof}
The implication $(1) \Rightarrow (2)$ is immediate.
Assuming that $(2)$ holds, Proposition~\ref{prop:ElementaryModL2q+1} implies that $\Theta$ is elementary modulo~$L_{2q+1}(R)$.
Since $L_{2q+1}(R)=0$, we deduce that $(2) \Rightarrow (1)$.
\end{proof}

If a formation $((P,\psi);F,V)$ is elementary then in particular $V$ admits a lagrangian complement, but it is unclear in general that for \emph{every} $((P',\psi');F',V')$ in the class $[((P,\psi);F,V)]\in\ell_{2q+1}(R)$, the module $V'\subseteq P'$ admits a lagrangian complement.
We will show that this property holds if~$L_{2q+1}(R)=0$ and one is willing to add a formation to $((P',\psi');F',V')$.

For brevity in the next lemma and its proof, given a quasi-formation $x=((P,\psi);F,V)$ we will say \emph{$x$ admits a lagrangian complement} to meant that $V\subseteq P$ admits a lagrangian complement.

\begin{lemma}\label{lem:inductionstep}
Let $R$ be a ring with $L_{2q+1}(R)=0$. Let $a$ and $b$ be quasi-formations over $R$ that are either stably isomorphic, or related by the defining relation~\eqref{eq:RelationMonoid} of $\ell_{2q+1}(R)$. Then there is a formation $f$ such that $a\oplus f$ admits a lagrangian complement if and only if there is there is a formation $f'$ such that $b\oplus f'$ admits a lagrangian complement.
\end{lemma}

\begin{proof}
First assume that $a$ and $b$ are stably isomorphic, so that there are trivial formations $t_1$ and~$t_2$ such that $a \oplus t_1 \cong b \oplus t_2$.
If $a\oplus  f$ admits a lagrangian complement for some formation $f$, then~$b\oplus(f\oplus t_2) \cong a\oplus f\oplus t_1$ admits a lagrangian complement because both $a\oplus f$ and $t_1$ do.
The proof of the converse is identical.

Next assume that $a=((P,\psi);F,V)$ and $b=((P,\psi);F,G)\oplus((P,\psi);G,V)$ are as in~\eqref{eq:RelationMonoid}.
If~$a\oplus f$ admits a lagrangian complement for some formation $f$, then $((P,\psi);G,V)\oplus f$ admits a lagrangian complement.
Since $L_{2q+1}(R)=0$, every formation is stably isomorphic to a boundary formation~\cite[Corollary 9.12]{RanickiAnIntroductionToAlgebraic}.
In particular,  after adding a trivial formation $t$, the formation~$((P,\psi);F,G)$ admits a lagrangian complement,  because every boundary formation does.
It follows that
\[
b\oplus f \oplus t \cong \left(((P,\psi);F,G)\oplus t \right) \oplus(((P,\psi);G,V)\oplus f)
\]
admits a lagrangian complement.

For the converse,  assume that $b\oplus f$ admits a lagrangian complement for some formation $f$.
More explicitly,~$((P,\psi);G,V)\oplus(((P,\psi);F,G)\oplus f)$ admits a lagrangian complement.
Since admitting a lagrangian complement is a statement that only involves the half rank direct summand, it follows that~$((P,\psi);F,V) \oplus ((P,\psi);F,G)\oplus f \ a\oplus((P,\psi);F,G)\oplus f$ also admits a lagrangian complement.
\end{proof}

\begin{proposition}
\label{prop:CleverCorollary}
Let $R$ be a ring with $L_{2q+1}(R)=0$, let~$\Theta \in \ell_{2q+1}(R)$ and let $u$ be any representative quasi-formation for $\Theta$.
If $\Theta$ is elementary, then there is a formation $v$ such that
in the quasi-formation $u\oplus v=:((P,\psi);F,V)$, the summand $V\subseteq P$ admits a lagrangian complement.
\end{proposition}

\begin{proof}
The definition of $\ell_{2q+1}(R)$ is as the transitive closure of the set of quasi-formations under the relation of being either stably isomorphic or related as in~\eqref{eq:RelationMonoid}. The proposition is now immediate from Lemma~\ref{lem:inductionstep}.
\end{proof}

\subsection{Induced forms}
\label{sub:InducedForms}

Following~\cite[Section 5]{CrowleySixt},  we describe some additional facts about the structure of the monoid~$\ell_{2q+1}(R)$.
\medbreak
Given a quasi-formation $x=((P,\psi);F,V)$ (a representative for) the quadratic form~$\psi \colon P \times P \to R$ restricts to a quadratic form~$\theta:=\psi|_{V \times V}$ on~$V$ leading to a split isometric inclusion~$(V,\theta) \hookrightarrow (P,\psi)$.
The same holds for $V^\perp$ and we formalise this as follows.

\begin{definition}
\label{def:InducedForm}
The \emph{induced forms} of a quasi-formation~$((P,\psi);F,V)$ are the quadratic forms~$(V,\theta)$ and~$(V^\perp,-\theta^\perp)$, where~$\theta :=\psi|_{V \times V}$ and~$\theta^\perp := \psi|_{V^\perp \times V^\perp}$.
\end{definition}

\begin{example}
\label{ex:InducedForm}
We describe the induced forms in some simple cases.
\begin{enumerate}
\item The induced forms of a formation~$((P,\psi);F,G)$ are both equal to the zero form~$(G,0)$. This is immediate since~$G$ is a lagrangian of~$(P,\psi)$ so $\psi|_{G \times G}=0$ and $G^\perp = G$.
\item\label{item:inducedform2} Adding a formation to a quasi-formation has the effect of adding a zero form to its induced forms: if~$x=(H_\varepsilon(K),K,V)$ is a quasi-formation and~$y=(H_\varepsilon(F);F,G)$ is a formation, then the induced forms of the quasi-formation~$x \oplus y=(H_\varepsilon(K \oplus F), K \oplus F,V \oplus G)$ are the quadratic forms~$(V \oplus G,\theta \oplus 0)$ and~$(V^\perp \oplus G,-\theta^\perp \oplus 0)$.
\end{enumerate}
\end{example}

Example~\ref{ex:InducedForm} \eqref{item:inducedform2} motivates the following definition from~\cite{CrowleySixt}.

\begin{definition}
\label{def:0Stabil}
Two quadratic forms~$(P,\psi)$ and~$(P',\psi')$ are \emph{$0$-stably equivalent} if there are zero forms $(Q,0)$ and $(Q',0)$ and an isometry
$$(P,\psi) \oplus (Q,0) \cong (P',\psi') \oplus (Q',0).$$
The $0$-stable equivalence class of a form~$(P,\psi)$ is denoted~$[P,\psi]_0$.
Such an equivalence class is called a \emph{0-stabilised quadratic form}.
We write~$\mathcal{F}^{\zs}_{2q}(R)$ for the unital abelian monoid of~$0$-stabilised~$(-1)^q$-quadratic forms over~$R$, where addition is given by direct sum.
\end{definition}

\begin{definition}
\label{ex:0Stabilisatioquasi-formation}
The \emph{induced~$0$-stabilised forms} of a quasi-formation~$((P,\psi);F,V)$ are the~$0$-stable equivalence classes~$[V,\theta]_0$ and~$[V^\perp,-\theta^\perp]_0$ of its induced forms~$(V,\theta)$ and~$(V^\perp,-\theta^\perp)$.
\end{definition}

The notion of induced~$0$-stabilised forms descends to~$\ell$-monoids, as first observed in~\cite{CrowleySixt}.
\begin{proposition}
\label{prop:InducedFormsMonoid}
Given a ring~$R$,  taking induced~$0$-stabilised forms gives rise to a well defined monoid homomorphism
\begin{align*}
b \colon \ell_{2q+1}(R) &\to \mathcal{F}^{\zs}_{2q}(R) \times \mathcal{F}^{\zs}_{2q}(R)  \\
[((P,\psi);F,V)] &\mapsto ([V,\theta]_0,[V^\perp,-\theta^\perp]_0).
\end{align*}
\end{proposition}
\begin{proof}
Adding a trivial formation to a quasi-formation adds a zero form to the induced forms; recall Example~\ref{ex:InducedForm}.
It therefore remains to study the effect of the relation in~\eqref{eq:RelationMonoid}:
$$((P,\psi);F,G) \oplus ((P,\psi);G,V) \sim ((P,\psi);F,V).$$
Since~$((P,\psi);F,G)$ is a formation, its induced forms are zero forms and therefore the induced forms on both sides of this equivalence are~$0$-stably equivalent.
Note that the induced forms of a quasi-formation are independent of the data of the first lagrangian.
\end{proof}

Given quadratic forms~$v,v'$ such that~$[v]_0 \oplus H_\varepsilon(R^k) = [v']_0 \oplus H_\varepsilon(R^k)$,~\cite[Corollary 5.3]{CrowleySixt} ensures that~$([v],[v']) \in \operatorname{\im}(b)$.
This leads to the following definition, which is due to Crowley-Sixt~\cite[Subsection 5.2]{CrowleySixt}.

\begin{definition}
\label{def:l5vv}
Given quadratic forms~$v,v'$ such that~$[v]_0 \oplus H_\varepsilon(R^k) = [v']_0 \oplus H_\varepsilon(R^k)$,  we define
\[
 \ell_{2q+1}(v,v'):=b^{-1}([v]_0,[v']_0)\subseteq \ell_{2q+1}(R).
\]
This is the set of all (equivalence classes of) quasi-formations whose induced forms are $0$-stably equivalent to~$v$ and~$v'$ respectively.
If $[v]_0=[v']_0$, then we write $\ell_{2q+1}(v)$ instead of $\ell_{2q+1}(v,v)$.
Note that for any $0$-form $(V,0)$, we have $\ell_{2q+1}(V,0)=L_{2q+1}(R)$~\cite[Lemma 6.2]{CrowleySixt}.
\end{definition}

Crowley and Sixt also show that for every $0$-stable equivalence class of a quadratic form~$v=(V,\theta)$, there exists a unique elementary class $e(v) \in \ell_{2q+1}(v)$ such that
$b \circ e(v) = ([v]_0,[v]_0)$~\cite[Corollary 5.3]{CrowleySixt}.
Explicitly,~$e(v)$ is the class of the boundary quasi-formation $\delta(V,\rho)$, where $\rho$ is a sesquilinear form representing~$\theta \in Q_\varepsilon(V)$.

\begin{definition}
\label{def:Triviall5v}
Given a quadratic form $v$ we say that $\ell_{2q+1}(v)$ is \emph{trivial} if it contains a single (necessarily elementary) element.
\end{definition}

\subsection{Linking forms and boundary automorphism sets}
\label{sub:LinkingForm}

This section focuses on quasi-formations over $\Z$.
The reason is that in the proof of our main theorem,  we convert a question about being elementary in~$\ell_5(\Z[\Z_2])$ into a question about being elementary in $\ell_5(\Z)$.
Given a quadratic form~$v=(V,\theta)$ over~$\Z$, we therefore recall Crowley-Sixt's description of~$\ell_5(v) $ in terms of linking forms~\cite{CrowleySixt}.
We use this to discuss various examples where~$\ell_5(v)$ is trivial.
In this section, the group $T$ is a finite abelian group and the group $V$ is always understood to be a finitely generated abelian group.
\medbreak

\begin{definition}\label{def:symmetric-linking}
A (symmetric) \emph{linking form} (over~$\Z$) refers to a pair~$(T,b)$ where $T$ is a finite abelian group and~$b \colon T \times T \to \Q/\Z$ is a symmetric bilinear form.
\end{definition}

\begin{definition}[{Ranicki~\cite[\S3.4]{RanickiExact}}]
\label{def:QuadraticLinking1}
A \emph{split quadratic linking form}~$(T,b,\nu)$ (over~$\Z$) is a symmetric linking form~$(T,b)$ together with a map~$\nu \colon T \to \Q/\Z$ such that the following equalities in $\Q/\Z$ hold for all~$x,x_1,x_2 \in T$ and all $r \in \Z$:
\begin{enumerate}
\item $\nu(rx)=r^2\nu(x),$
\item $\nu(x_1+x_2)-\nu(x_1)-\nu(x_2)=b(x_1,x_2)$.
\end{enumerate}
\end{definition}

Note that applying the second condition with $x_1=x_2=x$ yields $\nu(2x)-2\nu(x) = b(x,x)$, and then applying the first condition we learn that $2\nu(x)=b(x,x)$.

\begin{definition}
\label{def:BoundarySymmetric}
A nondegenerate symmetric bilinear form~$\lambda \colon V \times V \to \Z$, whose adjoint we denote $\widehat{\lambda} \colon V \to V^*$,  determines a linking form~$\partial \lambda \colon \coker(\widehat{\lambda}) \times \coker(\widehat{\lambda})\to \Q/\Z$ via the formula $([x],[y])\mapsto \frac{1}{s}y(z)$, where~$sx=\widehat{\lambda}(z)$ for some $z \in V$ and some $s \in \Z \setminus \lbrace 0 \rbrace$.
We call $\partial \lambda$ the \emph{boundary linking form} of the symmetric form $\lambda$.
\end{definition}

We recall the corresponding definition for quadratic forms.
Note that if~$\theta:=[\rho] \in Q_\varepsilon(V)$ is a quadratic form, then for $z \in V$, the integer $\rho(z,z)$ only depends on the the class of $\rho$ in $Q_\varepsilon(V)$.
As a consequence, we write $\theta(z,z)$ instead of $\rho(z,z)$.

\begin{definition}
\label{def:BoundaryQuadratic}
Let~$(V,\theta) \in Q_\varepsilon(V)$ be a nondegenerate quadratic form over $\Z$ with symmetrisation~$(V,\lambda)$.
The \emph{boundary split quadratic linking form} of~$(V,\theta)$ is the split quadratic linking form $\partial (V,\theta):=(\operatorname{coker}(\widehat{\lambda}),\partial \lambda,\partial \theta)$ where~$(\operatorname{coker}(\widehat{\lambda}),\partial \lambda)$ is the boundary linking form of the symmetric form~$(V,\lambda)$ and
\begin{align*}
\partial \theta \colon \operatorname{coker}(\widehat{\lambda}) &\to \Q/\Z \\
([x]) &\mapsto \frac{1}{s^2}\theta(z,z),
\end{align*}
where~$sx=\widehat{\lambda}(z)$ for $z \in V$ and $s \in \Z \setminus \lbrace 0 \rbrace$.
\end{definition}

One verifies that if a nondegenerate symmetric form $(V,\lambda)$ is represented by a size~$n$ matrix~$A$, then~$\partial\lambda([x],[y])=xA^{-1}y$ where~$A^{-1}$ denotes the inverse of~$A$ over~$\Q$ and $[x],[y] \in \Z^n/A\Z^n.$
We record the corresponding calculation in the quadratic case.

\begin{remark}
\label{rem:Matrix}
If a nondegenerate quadratic form~$(V,\theta)$ is represented by a size $n$ matrix~$Q$ so that its symmetrisation is represented by~$A=Q+Q^T$, then \[\partial \theta([x])=x^T(A^{-1})^{T}QA^{-1}x = x^TA^{-1}QA^{-1}x\] for every~$[x] \in \Z^n/A\Z^n$, where once again we invert $A$ over $\Q$.
\end{remark}

\begin{example}
\label{ex:Rank1}
We use Remark~\ref{rem:Matrix} to calculate the boundary split quadratic linking form of a rank one quadratic form.
When~$V = \Z$ and $\theta \in \Z \setminus \lbrace 0\rbrace$, we have~$\lambda=2\theta$ and the boundary forms are therefore defined on~$\Z/2\theta$.
Remark~\ref{rem:Matrix} implies that~$\partial \lambda([x],[y])=\frac{1}{2\theta}xy \in \Q/\Z$ and
\[\partial \theta([x])=\smfrac{\theta}{4\theta^2}x^2=\smfrac{1}{4\theta}x^2 \in \Q/\Z.\]
\end{example}

Still building towards the description of $\ell_5(v)$ in terms of linking forms, we now describe isometries of linking forms.

\begin{definition}
\label{def:IsometryLinkingForm}
An \emph{isometry} of split quadratic linking forms~$(T_0,b_0,\nu_0)$ and~$(T_1,b_1,\nu_1)$ is an isomorphism~$f \colon T_0 \xrightarrow{\cong} T_1$ such that~$\nu_1(f(x))=\nu_0(x)$ for every~$x\in T_0$.
\end{definition}

Observe that an isometry automatically satisfies~$b_1(f(x),f(y))=b_0(x,y)$ because
\[ b_1(f(x),f(y))=\nu_1(f(x)+f(y))-\nu_1(f(x))-\nu_1(f(y))=\nu_0(x+y)-\nu_0(x)-\nu_0(y)=b_0(x,y).\]
Given an isometry~$h \colon (V_0,\theta_0) \to (V_1,\theta_1)$ of quadratic forms, one verifies that~$(h^*)^{-1} \colon V_0^* \to V_1^*$ induces an isomorphism~$\coker(\widehat{\lambda}_0) \to \coker(\widehat{\lambda}_1)$ on the cokernels.
This isomorphism, which we denote by~$\partial h:=(h^*)^{-1}$, is an isometry of the boundary split quadratic linking  forms.

\begin{definition}
\label{def:BoundaryIsometry}
The \emph{boundary} of an isometry~$h \colon (V_0,\theta_0) \to (V_1,\theta_1)$ is the isometry
\[\partial h:=(h^{*})^{-1} \colon \partial(V_0,\theta_0) \to \partial  (V_1,\theta_1).\]
This determines a homomorphism
\[\partial \colon \Aut(V,\theta) \to \Aut(\partial(V,\theta)).\]
\end{definition}

The group~$\Aut(V,\theta)$ of isometries of~$(V,\theta)$ acts on the group~$\Aut(\partial(V,\theta))$ on the left both via~$h \cdot f:=\partial h \circ f$ and~$g \cdot f:=f\circ \partial g^{-1}$ where $f \in \Aut(\partial(V,\theta))$ and $g,h \in \Aut(V,\theta)$.
Combining these actions, which clearly commute, leads to a left action of~$\Aut(V,\theta) \times \Aut(V,\theta)$ on~$\Aut(\partial(V,\theta))$.

\begin{definition}
\label{def:BAut}
The \emph{boundary automorphism set} of a nondegenerate quadratic form~$(V,\theta)$ is the orbit set
$$\operatorname{bAut}(V,\theta)=\Aut(\partial (V,\theta))/\Aut(V,\theta) \times \Aut(V,\theta).$$
\end{definition}

Observe that an isometry~$f \in \Aut(\partial(V,\theta))$ is trivial in the boundary automorphism set~$\operatorname{bAut}(V,\theta)$ if and only if~$f$ extends to an isometry of~$(V,\theta)$.

\begin{example}
\label{ex:Rank1bAut}
Using Example~\ref{ex:Rank1}, one verifies that for a rank one form $(\Z,\theta)$, the automorphisms of the module $\Z/2\theta$ are the units of the ring~$\Z/2\theta$, while the automorphisms of $\partial(\Z,\theta)$ consist of the units~$x$ of~$\Z/2\theta$ such that~$x^2=1$ mod~$4\theta$.  The corresponding boundary automorphism set $\bAut(\Z,\theta)$
 consists of the same $x \in \Z/2\theta$, but with $x$ and $-x$ identified.
\begin{itemize}
\item For~$\theta=4$, the units of~$\Z/8$ are~$\pm 1$ and~$\pm 3$,
but only~$1$ and $-1$ satisfy~$x^2=1$ mod~$16$.
It follows that~$\operatorname{bAut}(\Z,4)$ is trivial.
\item For~$\theta=6$, the units of~$\Z/12$ are~$\pm 1,\pm 5$ all of which satisfy $x^2=1$ mod~$4\theta=24$.
It follows that~$\operatorname{bAut}(\Z,6)=\lbrace 1,5\rbrace$.
\end{itemize}
\end{example}


\begin{theorem}[{Crowley-Sixt~\cite[Section 6.3]{CrowleySixt}}]
\label{thm:CS}
Given a nondegenerate quadratic form~$v=(V,\theta)$ over $\Z$, there is a bijection between~$\ell_5(v)$ and~$\operatorname{bAut}(v)$.
\end{theorem}

\begin{example}
\label{ex:Lawson}
We illustrate Theorem~\ref{thm:CS} with some examples that will be relevant to the study of $\Z_2$-surfaces of nonorientable genus $h=1,2$.
\begin{itemize}
\item In Example~\ref{ex:Rank1bAut} we saw $\bAut(\Z,4)$ is trivial, and thus so is $\ell_5(\Z,4)$, by Theorem~\ref{thm:CS}.
\item We argue that $\bAut(4H_+(\Z))$ is trivial, thus proving $\ell_5(4H_+(\Z))$ is trivial.
First observe that the isometries of the quadratic form $4H_+(\Z)$ are
\[ \Aut(4H_+(\Z))=\Aut(H_+(\Z))=\left\lbrace
\bsm
\varepsilon_1&0\\ 0 & \varepsilon_2
\esm,
\bsm
0&\varepsilon_1\\ \varepsilon_2&0
\esm
\mid \varepsilon_1 \varepsilon_2= 1 \right\rbrace.\]
We verify that all the isometries of $\partial(4H_+(\Z))$ are induced by $\Aut(4H_+(\Z))$.
This was already observed by Kreck~\cite[p.~70]{KreckOnTheHomeomorphism}.
First, note that the symmetric linking form is $b \colon (\Z_4 \oplus \Z_4) \times (\Z_4 \oplus \Z_4) \to \Q/\Z$, represented by $B:= \bsm 0 & \frac{1}{4} \\ \frac{1}{4} & 0 \esm$, and the quadratic enhancement is $\nu \colon \Z_4 \oplus \Z_4 \to \Q/\Z, ke_1 + \ell e_2 \mapsto k\ell/4$.
An automorphism $\alpha$ of this is represented by a matrix $A:= \bsm a & b \\ c & d \esm$ over $\Z_4$ such that~$A^T B A = B$ over $\Q/\Z$ and such that~$\nu \circ \alpha = \nu \colon \Z_4 \oplus \Z_4 \to \Q/\Z$.
The former condition implies that $ad+bc =1 \in \Z_4$, while the conditions that $\nu(e_1) = \nu(ae_1 + ce_2)$ and $\nu(e_2) = \nu(be_1 + d e_2)$ give~$ac \equiv 0 \equiv bd$.  A case analysis yields exactly the same four matrices for $A$ as in the description of $\Aut(4H_+(\Z))$.  It follows that $\Aut(4H_+(\Z)) \to \Aut(\partial(4H_+(\Z)))$ is surjective, as claimed.
\end{itemize}
\end{example}

\subsection{Nikulin's criterion for trivial bAut}\label{sub:Nikulin}
Determining whether $\ell_5(v)$ is trivial is geometrically significant, as we will recall in the next section.
Theorem~\ref{thm:CS} converts this problem into a question about linking forms.
In order to state a criterion due to Nikulin~\cite{Nikulin} that describes conditions for $\bAut(v)$ to be trivial, we introduce some notation.
Namely,  given an abelian group~$G$ and a prime~$p$, we write~$n_p(G)$ for the minimal number of generators for the~$p$-primary part of~$G$, or equivalently $n_p(G) := \dim_{\mathbb{F}_p} (G \otimes \mathbb{F}_p)$.

\begin{theorem}[{Nikulin~\cite[Theorem 1.14.2]{Nikulin}}]
\label{thm:Nikulin}
Let~$v=(V,\theta)$ be an indefinite,  nondegenerate quadratic form over~$\Z$ with symmetrisation~$(V,\lambda)=(V,\theta+\theta^*)$.
Suppose that
\begin{enumerate}
\item the inequality $\operatorname{rk}(V) \geq n_p(\coker(\widehat{\lambda})) + 2$ holds for all odd primes~$p$, and
\item if~$\operatorname{rk}(V) = n_2(\coker(\widehat{\lambda}))$, then the split quadratic linking form~$\partial (\Z^2, [\bsm 0&2 \\ 0&0\esm ])$ is a summand of $\partial(V,\theta)$.
\end{enumerate}
Then $\partial \colon \Aut(V,\theta) \to \Aut(\partial (V,\theta))$ is surjective, and so $\bAut(v)$ consists of one element.
\end{theorem}

The statement in Nikulin's paper is not written this way. The fact that Nikulin's work could be translated into this language and used in this context was observed in Crowley-Sixt~\cite[Proposition 6.18 i)]{CrowleySixt}, where they reported the outcome of that translation. We provide the details of the translation momentarily, but first record the key consequence for our purposes.

\begin{corollary}
\label{cor:CS}
Under the conditions in the hypotheses of Theorem~\ref{thm:Nikulin},  $\ell_5(v)$ is trivial.
\end{corollary}
\begin{proof}
Combine Theorem~\ref{thm:Nikulin} with Theorem~\ref{thm:CS}.
\end{proof}

We now translate the result stated in \cite[Theorem 1.14.2]{Nikulin} to that stated in Theorem~\ref{thm:Nikulin}. First, Nikulin works with the nonsplit version of quadratic linking forms, whose definition we recall next, in which the quadratic enhancements $\mu$ are valued in $\Q/2\Z$.

\begin{definition}[{Ranicki~\cite[\S3.4]{RanickiExact}}]
\label{def:QuadraticLinking2}
A \emph{nonsplit quadratic linking form}~$(T,b,\mu)$ (over~$\Z$) is a symmetric linking form~$(T,b)$ together with a map~$\mu \colon T \to \Q/2\Z$ such that the following equalities in $\Q/2\Z$ hold for all~$x,x_1,x_2 \in T$ and all $r \in \Z$:
\begin{enumerate}
\item $\mu(rx)=r^2\mu(x),$
\item $\mu(x_1+x_2)-\mu(x_1)-\mu(x_2)= 2b(x_1,x_2)$.
\end{enumerate}
\end{definition}

In the more general context where Ranicki defined these notions, there is in general a difference between split and nonsplit quadratic linking forms. For us, because $1/2\in\Q/\Z$, there is effectively no difference.


\begin{proposition}[{\cite[Proposition~3.4.2]{RanickiExact}}]\label{prop:Ran-equiv-categories}
A split quadratic linking form $(T,b,\nu)$ determines a nonsplit quadratic linking form $(T,b,\mu)$ by sending $\nu \colon T \to \Q/\Z$ to $\mu := 2\nu \colon T \to \Q/2\Z$.
  This gives rise to an equivalence of categories
\[
  \left\lbrace
    \begin{array}{c}
         \text{split quadratic linking forms over } \Z   \\
         \text{and morphisms thereof}
    \end{array}
        \right\rbrace
            \xrightarrow{\cong}
         \left\lbrace
    \begin{array}{c}
         \text{nonsplit quadratic linking forms over } \Z  \ \\
         \text{and morphisms thereof}
    \end{array}
    \right\rbrace
\]
that is the identity on the morphisms.
\end{proposition}

\begin{definition}\label{defn:nonsplit-boundary-linking-form}
The \emph{boundary} of a nondegenerate, even symmetric bilinear form $\lambda \colon V \times V \to \Z$ is the nonsplit quadratic linking form $\partial(V,\lambda) := (\operatorname{coker}(\widehat{\lambda}),\partial \lambda,\mu)$, where
$(\operatorname{coker}(\widehat{\lambda}),\partial \lambda)$ is the boundary linking form of the symmetric form~$(V,\lambda)$, and
  \[\mu \colon \operatorname{coker}(\widehat{\lambda}) \to \Q/2\Z\] is defined by $\mu([y])=\tmfrac{1}{s^2}\lambda(z,z)$, where~$s\in \Z, z \in V$ and $y \in V^*$ satisfy~$sy=\widehat{\lambda}(z)$.
\end{definition}

\begin{remark}
\label{rem:Caution}
We caution the reader that $\partial (V,\lambda)$ does not refer to a symmetric linking form but to a nonsplit \emph{quadratic} linking form.
As illustrated in Example~\ref{ex:Rank1bAut}, quadratic linking forms typically have fewer isometries than their symmetrisations.
\end{remark}

\begin{construction}
\label{cons:RemarkCommuteFunctor}
Write $\mathbf{Quad},\mathbf{EvenSym},\mathbf{SplitQuadLink}$, and $\mathbf{NonSplitQuadLink}$ for the categories of nondegenerate quadratic forms over $\Z$, even symmetric forms over $\Z$, split quadratic linking forms over $\Z$ and nonsplit quadratic linking forms over $\Z$, together with morphisms thereof.
Consider the following  functors:
\begin{equation}
\label{eq:CommutativeFunctor}
\begin{tikzcd}[column sep = small]
\mathbf{Quad}
\ar[d,"\cong"] \ar[r,"\partial"]  &
\mathbf{SplitQuadLink}
 \ar[d,"\cong"]\\
\mathbf{EvenSym}
 \ar[r,"\partial"] & \
\mathbf{NonSplitQuadLink}.
\end{tikzcd}
\end{equation}
Here, the left vertical functor maps a quadratic form to its symmetrisation and is the identity on morphisms.  It is an equivalence of categories because every even symmetric form over $\Z$ admits a unique quadratic refinement.
The right hand side vertical functor is the equivalence described in Proposition~\ref{prop:Ran-equiv-categories}, and the horizontal functors take the boundary.
\end{construction}

\begin{proposition}
The square \eqref{eq:CommutativeFunctor} commutes.
\end{proposition}
\begin{proof}
To see that the diagram of categories and functors commutes on the level of objects, note that if $\theta \in Q_+(V)$ is a quadratic form with symmetrisation~$\lambda$, then
\[\partial \theta([y])=\tmfrac{1}{s^2}\theta(z,z)=\tmfrac{1}{2s^2}\lambda(z,z)=\tmfrac{1}{2}\mu([y]).\] 

The commutativity on the level of morphisms holds because a homomorphism $f \colon V \to V$ is sent via both routes to $\partial f \colon \coker (\widehat{\lambda})\to \Q/2\Z$.
\end{proof}

\begin{corollary}
\label{prop:CommuteFunctor}
The diagram in~\eqref{eq:CommutativeFunctor} induces the following commutative diagram of automorphism groups:
\begin{equation}
\label{eq:AutVtheta=AutVlambda}
\begin{tikzcd}[row sep = small]
    \Aut(V,\theta) \ar[r,"\partial"] \ar[d,"="] & \Aut(\partial(V,\theta)) \ar[d,"="]  \\
    \Aut(V,\lambda) \ar[r,"\partial"] & \Aut(\partial(V,\lambda)).
  \end{tikzcd}
  \end{equation}
  The vertical equalities refer to equalities as subsets of the automorphisms $\Aut(V)$ of the group~$V$.
\end{corollary}
 \begin{proof}
Restrict \eqref{eq:CommutativeFunctor} to the isometries.
\end{proof}

Nikulin's theorem is stated in his paper in  terms of the objects on the bottom row of~\eqref{eq:AutVtheta=AutVlambda} and Crowley-Sixt's Theorem~\ref{thm:CS} is stated in terms the objects on the top row.  For our statement in Theorem~\ref{thm:Nikulin}, we translated Nikulin's statement  into the language of quadratic forms and split quadratic linking forms, in order to be able to apply it in the context of Theorem~\ref{thm:CS}.


Summarising, Nikulin's map $\partial \colon \Aut(V,\lambda) \to \Aut(\partial (V,\lambda))$ is identified with the map $\partial \colon \Aut(V,\theta) \to \Aut(\partial(V,\theta))$ from Definition~\ref{def:BoundaryIsometry}.
Thus for $(V,\lambda)=(V,\theta+\theta^*)$, Nikulin's requirement that the nonsplit quadratic form~$\partial (V,\lambda)$ split off a $\partial (\Z^2,\bsm 0&2 \\ 2&0 \esm )$ summand (Nikulin writes $U^{(2)}(2)$ for the symmetric form $\lambda$ and $u_+^{(2)}(2)$ for its nonsplit quadratic boundary $\partial (V,\lambda)$~\cite[Proposition~1.8.1]{Nikulin}) is equivalent to requiring that the split quadratic linking form~$\partial (V,\theta)$ contains a $\partial (\Z^2,[\bsm 0&2 \\ 0&0 \esm])$ summand.

\medbreak

We conclude this section by illustrating Corollary~\ref{cor:CS} with some examples that will be useful when studying $\Z_2$-surfaces.
Given an integer $n \in \Z$,  consider the quadratic form
$$
X_n= \operatorname{sgn}(n)\left[\begin{pmatrix}
2&2&2&\ldots &2 \\
0&1&1&\cdots &1\\
0&0&\ddots&\ddots &\vdots \\
\vdots&\ddots&\ddots&\ddots &1\\
0&\ldots&0&0 &1
\end{pmatrix}\right] \in Q_+(\Z^n)
$$
from Proposition~\ref{prop:Calculate}.
For nonnegative integers $a,b \in \Z_{\geq 0}$,  set $h:=a+b$ and $\sigma=a-b$.
Observe that $h-|\sigma|$ is nonnegative and even, and note that if $\sigma=0$ then $h$ is even.
Consider the quadratic form
\[\theta_{a,b}:=
\begin{cases}
2H_+(\Z) \oplus H_+(\Z)^{\oplus (\frac{h}{2}-1)}  &\quad \text{ if } \sigma=0,   \\
X_\sigma \oplus H_+(\Z)^{\oplus \frac{1}{2}(h-|\sigma|)} &\quad \text{ if } \sigma \neq 0.
\end{cases}
\]
This is a rank $h$, signature $\sigma$,  nondegenerate quadratic form. If $h \neq |\sigma|$ then $h - |\sigma| > 0$ so $\theta_{a,b}$ is an indefinite form on $V := \Z^h$.

\begin{proposition}
\label{prop:ApplyNikulin}
Given nonnegative integers $a,b \in \Z_{\geq 0}$,  set $h:=a+b$ and $\sigma=a-b$.
If $h \leq 3$ or if $|\sigma| \neq h$, then the set $\ell_5(\Z^h,2\theta_{a,b})$ is trivial.
\end{proposition}
\begin{proof}
For $h=1$ we have $2\theta_{a,b}=4$ and we already proved in Example~\ref{ex:Lawson} that $\ell_5(\Z,4)$ is trivial.
For $h=2$ and $\sigma=0$ we have~$2\theta_{a,b}=4H_+(\Z)$ and again
we already proved in Example~\ref{ex:Lawson} that~$\ell_5(4H_+(\Z))$ is trivial.
The cases where $|\sigma|=h=2,3$ involve detailed explicit computations, and are treated in the appendix.

We now assume that $h\geq 3$ and $|\sigma| \neq h$. Write $\theta:=\theta_{a,b}$.
If $h \geq 3$ then $\theta$ has at least one hyperbolic summand, and therefore $2\theta$ has a $2H_+(\Z)$-summand.
We deduce that~$\partial (\Z^h,2\theta)$ splits off a $\partial (\Z^2,[\bsm 0&2\\ 0&0 \esm])$ summand.

Thanks to Corollary~\ref{cor:CS},  the proof therefore reduces to verifying that $\coker(\widehat{\lambda})$ is $2$-primary and satisfies~$n_2(\coker(\widehat{\lambda}))=\operatorname{rk}(V) = h$.  Thus in our case $n_p(\coker(\widehat{\lambda}))=0$ for every odd prime, and so $h \geq 3$ implies that $h = \operatorname{rk}(V) \geq n_p(\coker(\widehat{\lambda})) + 2 = 2$.
We have that
\[
2\theta \cong
\begin{cases}
4H_+(\Z) \oplus 2H_+(\Z)^{\oplus \frac{h}{2}-1}  &\quad \text{ if } \sigma=0,   \\
2X_\sigma \oplus 2H_+(\Z)^{\oplus \frac{1}{2}(h-|\sigma|)} &\quad \text{ if } \sigma\neq 0.
\end{cases}
\]
When $\sigma=0$, we have that $\coker(4H^+(\Z) \oplus 2H^+(\Z)^{\oplus \frac{h}{2}-1}) \cong \Z_4 \oplus \Z_4 \oplus \Z_2^{\oplus (h-2)}$.
This is 2-primary and generated by no fewer than $h$ elements.  Thus $n_2(\operatorname{coker}(\widehat{\lambda}))= h$ in this case.
When $\sigma \neq 0$, we have that $\operatorname{coker}(2H_+(\Z)^{\oplus \frac{1}{2}(h-|\sigma|)}) \cong (\Z_2)^{\oplus h-|\sigma|}$, which is 2-primary and generated by no fewer than $h-|\sigma|$ elements.
To complete the $\sigma \neq 0$ case, it therefore suffices to show that $\operatorname{coker}(2X_\sigma+2X_\sigma^T)$ is $2$-primary and generated by $|\sigma|$ and no fewer elements.

We consider
\[2X_\sigma+2X_\sigma^T=
 \operatorname{sgn}(\sigma)\begin{pmatrix}
8&4&4&\ldots &4 \\
4&4&2&\cdots &2\\
4&2&\ddots&\ddots &\vdots \\
\vdots&\vdots&\ddots&\ddots &2\\
4&2&\ldots&2 &4
\end{pmatrix}.
\]
By performing row and column operations, we compute that
$$\operatorname{coker}(2X_\sigma+2X_\sigma^T)
=\begin{cases}
 \Z_2^{|\sigma|-1} \oplus \Z_8 &\quad \text{ if } |\sigma| \text{ is odd,} \\
 \Z_2^{|\sigma|-2} \oplus \Z_4 \oplus \Z_4 &\quad \text{ if } |\sigma| \text{ is even.}
\end{cases}
$$
Thus~$\operatorname{coker}(2X_\sigma+2X_\sigma^T)$ is indeed~$2$-primary and is generated by no fewer than $|\sigma|$ elements.
This concludes the verification that the assumptions of Corollary~\ref{cor:CS} hold, from which it follows that~$\ell_5(\Z^h,2\theta)$ is trivial.
\end{proof}

\begin{remark}
\label{rem:hleq3}
We currently know of no pair $(a,b)$ where $\ell_5(\Z^h,2\theta_{a,b})$ is nontrivial; here recall that~$h:=a+b$ and $\sigma=a-b$.
In particular, it could be that $\ell_5(\Z^h,2\theta_{a,b})$ is trivial for all~$a,b$.  It would be very interesting to know whether or not this holds.
Proposition~\ref{prop:ApplyNikulin} implies that~$\ell_5(\Z^h,2\theta_{a,b})$ could only potentially be nontrivial when $h=|\sigma| \geq 4$.
In the appendix, we verify by direct calculations of isometries that $\ell_5(\Z^h,2\theta_{a,b})$ is trivial in the cases corresponding to $h=|\sigma| \in \lbrace 2,3 \rbrace$. The computations  grow rapidly in complexity with $h$. For~$h=|\sigma|=2$,  the group~$\Aut(\Z^h,\theta_{a,b}))$ has $8$ elements, whereas for $h=|\sigma|=3$, it has $48$ elements.
This computational complexity explains why, for extremal values of $\sigma$, we only proved Proposition~\ref{prop:ApplyNikulin} for~$h \leq 3$.
\end{remark}

\section{Modified surgery theory}
\label{sec:ModifiedSurgery}
In this section we recall some elements of Kreck's modified surgery theory~\cite{KreckSurgeryAndDuality}.
In Subsection~\ref{sub:NormalSmoothings}, we review the concept of a \emph{normal $k$-smoothing} of a manifold $M$.
Restricting to dimension~$4$, Subsection~\ref{sub:SurgeryObstruction} recalls the definition of the modified surgery obstruction that takes values in $\ell_5(\Z[\pi_1(M)])$.
In Subsection~\ref{sub:WallForm} we narrow down where in the monoid~$\ell_5(\Z[\pi_1(M)])$ the modified surgery obstruction lives.

\subsection{Normal smoothings}
\label{sub:NormalSmoothings}

We recall the notion of a normal smoothing and of the normal $k$-type from~\cite{KreckSurgeryAndDuality}.
We then consider the concrete case where the manifolds are $4$-dimensional and orientable with fundamental group $\Z_2$.

\medbreak

Let $B$ be a topological space with the homotopy type of a CW complex and let $\xi \colon B \to \BSTOP$ be a fibration.
An \emph{$n$-dimensional $(B, \xi)$-manifold} consists of a pair
$(M,\overline{\nu})$, where $M$ is an oriented~$n$-manifold
and $\overline{\nu} \colon M \to B$ is a lift of the oriented stable (topological) normal bundle $\nu \colon M \to \BSTOP$ of~$M$, meaning that $\nu=\xi \circ \overline{\nu}$.

A~\emph{$(B, \xi)$-null-bordism} for a closed~$n$-dimensional~$(B, \xi)$-manifold~$(M_0,\overline{\nu}_0)$ is an $(n{+}1)$-dimensional $(B, \xi)$-manifold $(W,\overline{\nu})$ for which~$\partial (W,\overline{\nu})=(M_0,\overline{\nu}_0)$.
A~\emph{$(B, \xi)$-cobordism} between two $n$-dimensional $(B, \xi)$-manifolds $(M_0,\overline{\nu}_0)$
and~$(M_1,\overline{\nu}_1)$ consists of:
\begin{enumerate}[(i)]
  \item a homeomorphism $f \colon \partial M_0 \to \partial M_1$;
  \item\label{item:condition2} for $i=0,1$, homotopies $\ol{\nu}_i \sim \ol{\nu}^\dag_i$, to some maps $\ol{\nu}^\dag_i \colon M_i \to B$, such that
\[{\ol{\nu}^\dag_0}|_{\partial M_0} = {\ol{\nu}^\dag_1}|_{\partial M_1} \circ f \colon \partial M_0 \to B;\]
  \item an~$(n{+}1)$-dimensional $(B, \xi)$-null-bordism of $(M_0 \cup_f -M_1, \overline{\nu}')$
where $\overline{\nu}'|_{M_0} = \overline{\nu}^\dag_0$ and $\overline{\nu}'|_{M_1} = -\overline{\nu}^\dag_1$.
\end{enumerate}
We note that the homotopies in condition~\eqref{item:condition2} are there to ensure one can produce a well-defined $B$-structure $\ol{\nu}'$ on $M_0\cup_f-M_1$ in the scenario that the $B$-structures $\ol{\nu}_1\circ f$ and $\ol{\nu}_0$ are isomorphic, but not necessarily equal (equality is required for glueing).

We briefly discuss orientation reversal, to give meaning to the notation $ -\overline{\nu}^\dag_1$ above. Given a~$(B,\xi)$-manifold $(M, \ol{\nu})$, the map $\ol{\nu} \circ \pr_1 \colon M\times [0,1]\to B$ defines a $(B,\xi)$-manifold $(M\times[0,1],\pr_1\circ \ol{\nu})$. We identify $(M, \ol{\nu})$ with the restriction to $M\times\{0\}$, and define $-\ol{\nu}\colon M\to B$ as the restriction to the end $M\times\{1\}$, which then produces a $(B,\xi)$-manifold $(-M,-\ol{\nu})$, i.e.~$-\ol{\nu}$ lifts the oppositely oriented stable normal bundle $\nu_{-M}\colon -M\to \BSTOP$.

Recall that a map of spaces $f\colon X\to Y$ is \emph{$m$-connected} if $f_*\colon\pi_i(X)\to \pi_i(Y)$ is an isomorphism for $i<m$ and is surjective for $i=m$. A map of spaces $f\colon X\to Y$ is \emph{$m$-coconnected} if $f_*\colon\pi_k(X)\to \pi_k(Y)$ is an isomorphism for $i>m$ and is injective for~$i=m$.

\begin{definition}
\label{def:NormalSmoothing}
Let $B$ be a space with the homotopy type of a CW complex with finite $(k+2)$-skeleton, let $\xi \colon B \to \BSTOP$ be a fibration, and let $(M,\overline{\nu})$ be a $(B, \xi)$-manifold.
\begin{enumerate}
\item If $\overline{\nu}$ is $(k{+}1)$-connected, then $(M, \overline{\nu})$ is called a \emph{normal $k$-smoothing} into $(B,\xi)$.
\item The pair $(B,\xi)$ is a \emph{normal $k$-type} for $M$ if $\xi$ is $(k+1)$-coconnected and there exists a normal $k$-smoothing $(M, \overline{\nu})$ into $(B,\xi)$.
\end{enumerate}
\end{definition}

The existence of the Moore-Postnikov factorisation~\cite{Baues-obstruction-theory} of the stable normal bundle~$\nu \colon M \to \BSTOP$ ensures that for all $k \geq 0$, there exists a normal $k$-type for $M$. Moreover, this theory implies any two normal $k$-types for $M$ are fibre homotopy equivalent.

We now assume that $M$ is a $4$-manifold with fundamental group $\Z_2$, and equipped with a spin structure. This implies the existence of maps $c \colon M \to B\Z_2$ and $\s \colon M \to \BTOPSpin$, where $c$ classifies the universal cover of $M$, and $\s$ is a lift of the stable normal bundle along the standard principal fibration $\gamma\colon \BTOPSpin\to \BSTOP$. These maps are unique up to homotopy. 
They will be used in the statement of the next lemma.

\begin{lemma}
\label{lem:Normal1Type}
Let $M$ be a spin $4$-manifold with fundamental group $\Z_2$, equipped with a spin structure. Write
\[
(B,\xi):=(\BTOPSpin\times B\Z_2,\gamma\circ\pr_1).
\]
Then the map
\[
 \overline{\nu}:=\s\times c \colon M \to \BTOPSpin \times B\Z_2 
 \]
determines a normal $1$-smoothing into $(B,\xi)$, and $(B,\xi)$ is a normal $1$-type for $M$.
\end{lemma}

The lemma is well-known to experts, e.g.~\cite{KreckOnTheHomeomorphism,KreckSurgeryAndDuality}; we provide the details for the convenience of the reader.

\begin{proof}
The space $B$ has finite $3$-skeleton because both $B \Z_2 \simeq \R P^\infty$ and $\operatorname{\BTOPSpin}$ do (for the latter, use $\pi_i(\BTOPSpin) = 0$ for $i=1,2,3$ together with CW approximation).
  The map $\BTOPSpin \to \BSTOP$ is 2-coconnected because $\TOPSpin \to \STOP$ is the universal covering map.
   It follows that $\xi$ is 2-coconnected, because $\pi_i(B\Z_2)=0$ for $i \geq 2$.
Next, since $M$ is spin,  the stable normal bundle
$\nu\colon M \to \BSTOP$ lifts to $\s \colon M \to \BTOPSpin$
and therefore combining with~$c$ we obtain a smoothing $M \to \BTOPSpin \times B\Z_2$. This is 2-connected because $\pi_i(\BTOPSpin) = 0$ for $i=1,2$.
\end{proof}


\begin{proposition}
\label{prop:StablyHomeo}
If~$F_0$ and $F_1$ are~$\Z_2$-surfaces with the same nonorientable genus and normal Euler number,  then there exist normal $1$-smoothings $(X_{F_i}, \ol{\nu}_i)$ for $i=0,1$, that are bordant over their normal $1$-type,  relative to a homeomorphism $f \colon \partial X_{F_0} \to \partial X_{F_1}$ that restricts to a homeomorphism $S(\nu F_0)\to S(\nu F_1)$ that is $\nu$-extendable rel.~boundary.
\end{proposition}

\begin{proof}
Write~$(B,\xi)$ for the normal~$1$-type of the~$X_{F_i}$ that was described in Lemma~\ref{lem:Normal1Type}.
By Proposition~\ref{prop:UnionSpin},  there is a homeomorphism~$f \colon \partial X_{F_0} \to \partial X_{F_1}$ such that  $M:=X_{F_0} \cup_{f} -X_{F_1}$ is spin and has fundamental group~$\Z_2$.
Fix a spin structure on $M$. Via Lemma~\ref{lem:Normal1Type}, we obtain a normal $1$-smoothing $\ol{\nu}\colon M\to B$. Denote by $\ol{\nu_0}\colon X_{F_0}\to B$ and $-\ol{\nu_1}\colon -X_{F_1}\to B$ the restriction of $\ol{\nu}$ to the respective exteriors.
By Lemma~\ref{lem:Normal1Type}, $(X_{F_i}, \ol{\nu}_i)$ is a normal $1$-smoothing for $i=0,1$.
The proposition now reduces to proving that~$[(M,\overline{\nu})]$ vanishes in the bordism group~$\Omega_4(B,\xi)\cong\Omega_4^{{\TOPSpin}}(\Z_2)$.
A spectral sequence calculation shows that~$\Omega_4^{\TOPSpin}(\Z_2)\cong\Omega_4^{\TOPSpin}\cong\Z$, detected by the signature divided by 8. See e.g.~\cite[Proof~of~Proposition~5.1]{OrsonPowellSpines} for the details.

For $i=0,1$, the Mayer-Vietoris sequence corresponding to the decomposition~$D^4 = X_{F_i} \cup -\ol{\nu} F_i$ implies that the inclusion-induced map $H_2(\partial X_{F_i};\Z) \to H_2(X_{F_i};\Z)$ is surjective.
As a result,  the ordinary intersection pairing of $X_{F_i}$ vanishes identically, and so in particular $\sigma(X_{F_i}) =0$.
Then Novikov additivity of the signature implies that the signature of~$M$ also vanishes.
It follows that~$(M,\overline{\nu})$ is $(B,\xi)$-null-bordant, as desired.
\end{proof}

\subsection{The surgery obstruction}
\label{sub:SurgeryObstruction}

We recall the definition of the modified surgery obstruction in odd-dimensions.
The main reference is~\cite{KreckSurgeryAndDuality}.
For simplicity,
we restrict our discussion to $4$-manifolds whose fundamental group has vanishing Whitehead group.
Since $\operatorname{Wh}(\Z_2)=0$~\cite{WhiteheadSimple}, this will suffice for our purposes.

\medbreak
We say that $(W,\overline{\nu},M_0,\overline{\nu}_0,M_1,\overline{\nu}_1)$ is a \emph{modified surgery problem} if $M_0$ and $M_1$ are $4$-manifolds with normal $1$-type $(B,\xi)$,  the $\overline{\nu}_i \colon M_i \to B$ are normal $1$-smoothings, and $(W,\overline{\nu})$ is a $(B,\xi)$-cobordism between $(M_0,\overline{\nu}_0)$ and  $(M_1,\overline{\nu}_1)$.


Let $(W,\overline{\nu},M_0,\overline{\nu}_0,M_1,\overline{\nu}_1)$ be a modified surgery problem.
By surgery below the middle dimension~\cite[Proposition~4]{KreckSurgeryAndDuality}, we may assume that~$\overline{\nu} \colon W \to B$
is~$2$-connected i.e.~that $(W,\overline{\nu})$ is a normal $1$-smoothing.
It can then be proved that~$K\pi_2(W):=\ker(\overline{\nu}_*)$ is finitely generated; see e.g.~\cite[Lemma 3.4]{CCPS}.
Choose a finite set~$\omega$ of generators for~$K\pi_2(W)$.  Choose disjoint embeddings~$\varphi_i \colon S^2 \times D^3 \hookrightarrow W$ with basing paths representing these generators, that are compatible with the~$(B, \xi)$-structure.
Set $$ U:=\bigsqcup_i \varphi_i(S^2 \times D^{3}).$$
%
%

Since
~$\partial U\cong\bigsqcup_i S^2 \times S^2$,
its quadratic intersection form, denoted~$(H_2(\partial U;\Z[\pi_1(B)]),\psi)$, is hyperbolic.
For concision, we refer to the pair~$(\omega,\varphi):=(\omega,\sqcup_i \varphi_i)$ as a \emph{set of embedded generators for~$K \pi_2(W)$}, we set~$\Lambda:=\Z[\pi_1(B)]$, and consider the~$\Lambda$-modules
\begin{equation}
\label{eq:Kernelquasi-formation}
F:=H_{3}(U,\partial U;\Lambda)\quad \text{ and } \quad V:=H_{3}(W \setminus \mathring{U},M_0 \sqcup \partial U;\Lambda).
\end{equation}
Via the boundary maps in the long exact sequences of the appropriate pairs,~$F$ and~$V$ can be identified with submodules of~$P:=H_2(\partial U;\Lambda)$. Under these identifications,~$F \subseteq P$ is a lagrangian and~$V \subseteq P$ is a stably free half-rank direct summand~\cite[~p.~734]{KreckSurgeryAndDuality}.
The \emph{kernel quasi-formation} associated with these data is defined as
$$\Sigma(W,\overline{\nu},\omega,\varphi):= ((P,\psi);F,V).$$
The class of this quasi-formation in~$\ell_{5}(\Lambda)$ only depends on the~$(B, \xi)$-bordism class rel.\ boundary of $(W,\overline{\nu})$~\cite[Theorem 4]{KreckSurgeryAndDuality}.
This leads to the following definition~\cite[p.~734]{KreckSurgeryAndDuality}; see also~\cite[p.~494]{CrowleySixt}.

\begin{definition} \label{def:SurgeryObstruction}
Let $(W,\overline{\nu},M_0,\overline{\nu}_0,M_1,\overline{\nu}_1)$ be a modified surgery problem.
The \emph{modified surgery obstruction} of~$(W,\overline{\nu})$ is
$$ \Theta(W,\overline{\nu}):=[(\Sigma(W',\overline{\nu}',\omega',\varphi'))] \in \ell_{5}(\Lambda),$$
where~$(W',\overline{\nu}')$ is a normal~$1$-smoothing over~$(B, \xi)$, which is~$(B, \xi)$-bordant
rel.\ boundary to~$(W, \overline{\nu})$
and~$(\omega',\varphi')$ is a set of embedded generators for~$K \pi_2(W')$.
\end{definition}

\begin{remark}
The class~$\Theta(W,\overline{\nu}) \in \ell_{5}(\Lambda)$ only depends on the~$(B, \xi)$-bordism class rel.\ boundary of the~$(B, \xi)$-null-bordism~$(W,\overline{\nu})$: it is independent of the choice of the~$1$-smoothing~$(W',\overline{\nu}')$ and of the subsequent choice of a set of embedded generators for~$K \pi_2(W')$~\cite[Theorem 4]{KreckSurgeryAndDuality}.
\end{remark}

\begin{theorem}[Kreck~{\cite[Theorem 4]{KreckSurgeryAndDuality}}]
\label{thm:Kreck}
Let $(W,\overline{\nu},M_0,\overline{\nu}_0,M_1,\overline{\nu}_1)$ be a modified surgery problem.
The modified surgery obstruction $\Theta(W,\overline{\nu}) \in \ell_5(\Lambda)$ is elementary if and only if $(W,\overline{\nu})$ is $(B,\xi)$-bordant rel.~boundary to an $s$-cobordism.
\end{theorem}

\subsection{Wall forms}
\label{sub:WallForm}

Let $(W^5,\overline{\nu},M_0,\overline{\nu}_0,M_1,\overline{\nu}_1)$ be a modified surgery problem as in the previous section.
We wish to narrow down where in the monoid~$\ell_{5}(\Lambda)$ the modified surgery obstruction~$\Theta(W,\overline{\nu})$ lives. For this we will show how to choose nondegenerate quadratic forms~$v$ and~$v'$, on free $\Lambda$-modules, so that the obstruction lies in the subset~$\ell_{5}(v,v')\subseteq \ell_{5}(\Lambda)$ described in Definition~\ref{def:l5vv}.
\medbreak

Write $(M^4_0,\overline{\nu}_0)$ for a normal $1$-smoothing into the normal $1$-type~$(B,\xi)$ of~$M_0$ and recall that
\[K\pi_2(M_0):=\ker((\overline{\nu}_0)_* \colon \pi_2(M_0) \to \pi_2(B)).\]
Intersections and self-intersections in $M_0$ define a quadratic form~$(K\pi_2(M_0),\psi_{M_0})$ called the \emph{Wall form} of~$(M_0, \overline{\nu}_0)$~\cite[Section 5]{KreckSurgeryAndDuality}.
Write~$\lambda_{M_0}$ for the symmetrisation of $\psi_{M_0}$ (this agrees with the restriction of the equivariant intersection form on $\pi_2(M_0)=H_2(M_0;\Lambda)$) and recall that
the \emph{radical} $\operatorname{rad}(\lambda_{M_0}) \subseteq K\pi_2(M_0)$ of this pairing consists by definition of those~$x \in K\pi_2(M_0)$ such that~$\lambda_{M_0}(x,-) \equiv 0$.
We shorten~$K\pi_2(M_0)/\operatorname{rad}(\lambda_{M_0})$ to~$K\pi_2(M_0)/\operatorname{rad}$ from now on.
The Hermitian form $\lambda_{M_0}$ descends to a nondegenerate Hermitian form $\lambda_{M_0}^{\nd}$ on this quotient and therefore determines a quadratic form $\psi_{M_0}^{\nd}$ on this quotient.
The form $(K\pi_2(M_0)/\operatorname{rad},\psi_{M_0}^{\nd})$ is called the \emph{nondegenerate Wall form}.

We would like to use the nondegenerate Wall forms for $M_0$ and $M_1$ as the forms $v$ and $v'$ mentioned at the start of this subsection.
However, the modules~$K\pi_2(M_i)$ need not be free, or even stably free.
This is solved using the following idea from~\cite{KreckSurgeryAndDuality,CrowleySixt}.

\begin{definition}
\label{def:FreeWall}
Let~$(M_0^{4},\overline{\nu}_0)$ be a normal~$1$-smoothing into the normal $1$-type $(B,\xi)$ of $M_0$.
Given a free~$\Lambda$-module~$S$ and a surjection~$\varpi\colon S \twoheadrightarrow K \pi_2 (M_0)/ \operatorname{rad}$, the pull-back of the nondegenerate Wall form~$(K\pi_2(M_0)/\operatorname{rad},\psi_{M_0}^{\nd})$ by~$\varpi$ is called a {\em free Wall form} of~$(M_0, \overline{\nu}_0)$ and is denoted by
\[  (S, \vartheta) := (S, \varpi^*\psi_{M_0}^{\nd}). \]
When  the free module and the quadratic form are not directly relevant to an argument, we will sometimes simply write $v(\overline{\nu}_0)$ for a free Wall form of $(M_0,\overline{\nu}_0)$. 
\end{definition}

A consequence of~\cite[Proposition 8]{KreckSurgeryAndDuality} is that all free Wall forms of~$(M_0,\overline{\nu}_0)$
are~$0$-stably equivalent (Definition~\ref{def:0Stabil}).
The following proposition is also a consequence of~\cite[Proposition~8 ii)]{KreckSurgeryAndDuality}.
Recall the map of monoids $b \colon \ell_5(\Lambda) \to \mathcal{F}^{\zs}_{4}(\Lambda) \times \mathcal{F}^{\zs}_{4}(\Lambda)$ from Proposition~\ref{prop:InducedFormsMonoid}, which takes a quasi-formation to its pair of induced forms, considered up to 0-stable equivalence.

\begin{proposition}
\label{prop:ConsequenceProp8}
Let $(W,\overline{\nu},M_0,\overline{\nu}_0,M_1,\overline{\nu}_1)$ be a modified surgery problem with modified surgery obstruction~$\Theta(W,\overline{\nu})\in \ell_5(\Lambda)$.
For any choice of free Wall forms~$v(\overline{\nu}_0)$ and~$v(\overline{\nu}_1)$ for~$(M_0,\overline{\nu}_0)$ and~$(M_1,\overline{\nu}_1)$ respectively,
we have a $0$-stable equivalence:
\[b(\Theta(W,\overline{\nu})) = (v(\overline{\nu}_0)_0 , v(\overline{\nu}_1)_0) \in \mathcal{F}^{\zs}_{4}(\Lambda) \times \mathcal{F}^{\zs}_{4}(\Lambda)\]
and hence
\[\Theta(W,\overline{\nu}) \in \ell_{5}(v(\overline{\nu}_0),v(\overline{\nu}_1)).\]
\end{proposition}
\begin{proof}
Since the induced form of a formation is the zero form (recall Example~\ref{ex:InducedForm}), it follows from~\cite[Proposition~8 ii)]{KreckSurgeryAndDuality} that the induced forms of $\Theta(W,\overline{\nu})$ are zero stably equivalent to the free Wall forms~$v(\overline{\nu}_0)$ and~$v(\overline{\nu}_1)$.
The definition of the map $b$ from Proposition~\ref{prop:InducedFormsMonoid} associates to a quasi-formation its induced forms and so the equality~$b(\Theta(W,\overline{\nu})) = (v(\overline{\nu}_0)_0 , v(\overline{\nu}_1)_0)$ immediately follows.
\end{proof}

\section{Forms and quasi-formations over~$\texorpdfstring{\Z[\Z_2]}{Z[Z/2]}$}
\label{sec:PlusMinusForms}

As before we write~$\Z_2=\langle T \mid T^2=1 \rangle$.
In this section we specialise to the ring $\Lambda:=\Z[\Z_2]=\Z[T]/\langle T^2-1\rangle$ and develop some results concerning quasi-formations over this ring.
The eventual goal is to describe a tractable criterion guaranteeing that a class~$\Theta \in \ell_5(\Z[\Z_2])$ is elementary.

In Subsection~\ref{sub:Eigenspaces} we use $\pm 1$ eigenspaces to associate two forms over $\Z$ to a form over $\Z[\Z_2]$, called the \emph{plus and minus forms.}
In Subsection~\ref{sub:PlusMinusquasi-formation},  we use plus and minus forms to associate two quasi-formations over~$\Z$ to a quasi-formation over $\Z[\Z_2]$.
In Subsection~\ref{sub:KeyLemma}, we use this construction to prove the main technical ingredient.
In Subsection~\ref{sub:ElementaryCriterion},  we deduce a criterion ensuring that a class~$\Theta \in \ell_5(\Z[\Z_2])$ is elementary; see Proposition~\ref{prop:ElementaryPlusMinus}.

In this section we assume that $\Z[\Z_2]$-modules are finitely generated and free; cf. Lemma~\ref{lem:StablyFreeImpliesFree}.

\subsection{Forms over~$\texorpdfstring{\Z[\Z_2]}{Z[Z/2]}$ and eigenspaces}
\label{sub:Eigenspaces}

A~$\Z[\Z_2]$-module can also be thought of as an abelian group.
Given a~$\Z[\Z_2]$-module~$P$,  recall that we write~$\mathcal{E}_{\pm}(P) \subseteq P$ for the abelian subgroup given by the~$(\pm 1)$-eigenspaces of~$T$.

In this subsection we recall how a hermitian (resp.~quadratic) form over~$\Z[\Z_2]$ induces a hermitian (resp. quadratic) form over~$\Z$ on $\mathcal{E}_{\pm}(P)$.

\begin{remark}
\label{rem:EigenspaceEquality}
Recall that $P$ is assumed to be free. In this case,  we argue that~$\mathcal{E}_{\pm}(P)$ can equivalently be described as~$(1\pm T)P$.
Certainly, we always have~$(1\pm T)P \subseteq \mathcal{E}_{\pm}(P)$.
For the reverse inclusion,  since~$P$ is free we can work coordinate by coordinate, thereby reducing to the case $P=\Z[\Z_2]$.
Now if~$a +bT \in \Z[\Z_2]$ satisfies~$T(a+bT)=\pm (a+bT)$, then $T(a+bT) = b+ Ta$, so~$a=\pm b$.
\end{remark}

\begin{remark}
In the notation of Subsection~\ref{sub:QuadraticForms}, which concerned $\varepsilon$-quadratic forms over  a ring~$R$, we have $R=\Z[\Z_2]$ and $\varepsilon = 1$. Note that $\ol{a + bT} = a+bT$ in $\Z[\Z_2]$.
It follows that~$Q_+(\Z[\Z_2]) = \Z[\Z_2]$, and so quadratic refinements of hermitian forms take values in $\Z[\Z_2]$.
\end{remark}

Given a quadratic form $(P,\lambda,\mu)$ over $\Z[\Z_2]$,  since $(1 \pm T)\ol{(1\pm T)} = (1\pm T)^2 = 2(1\pm T)$,  we have that $\mu((1\pm T)x)= (1\pm T) \mu(x) \ol{(1 \pm T)} = 2(1\pm T)\mu(x)$ and $\lambda((1\pm T)x,(1\pm T)y)= (1\pm T)\lambda(x,y)\ol{(1 \pm T)} =  2(1\pm T)\lambda(x,y)$ for every $x,y\in P$.
It follows that the restrictions of~$\lambda$ and $\mu$ to~$\mathcal{E}_{\pm}(P)$ take values in $2\mathcal{E}_\pm(\Z[\Z_2]) \subseteq \mathcal{E}_{\pm}(\Z[\Z_2])$.  We identify $\Z \cong \mathcal{E}_{\pm}(\Z[\Z_2])$ using the isomorphism~$n \mapsto (1 \pm T)n$. We refer to the inverse as dividing by $(1 \pm T)$.
Under this identification~$2\mathcal{E}_\pm(\Z[\Z_2])$ is identified with $2\Z$.
Dividing by $2$ gives rise to an isomorphism $2\Z \to \Z$.
We will use the composition
\[\frac{1}{2(1 \pm T)} \colon 2\mathcal{E}_\pm(\Z[\Z_2]) \xrightarrow{/(1 \pm T)} 2\Z \xrightarrow{/2} \Z\]
in the following construction.

\begin{construction}
\label{cons:Eigenspace}
Given a quadratic form $(P,\lambda,\mu)$ over $\Z[\Z_2]$, we set
\begin{align*}
\mathcal{E}_\pm(\lambda)&:=\smfrac{1}{2(1 \pm T)}\circ \lambda|
\colon \mathcal{E}_\pm(P) \times \mathcal{E}_\pm(P) \to \Z,\\
\mathcal{E}_\pm(\mu)&:=\smfrac{1}{2(1 \pm T)}\circ \mu| 
\colon \mathcal{E}_\pm(P) \to \Z.
\end{align*}
%
We write $\mathcal{E}_\pm(P,\lambda,\mu):=(\mathcal{E}_\pm(P),\mathcal{E}_\pm(\lambda),\mathcal{E}_\pm(\mu))$.
Similarly if $\theta=[\rho] \in Q_\varepsilon(P)$
 is a quadratic form, then we write $\mathcal{E}_\pm(\theta) \in  Q_\varepsilon(\E_\pm(P))$
  for the quadratic forms determined by this construction, namely the equivalence class represented by the pairing
\[\smfrac{1}{2(1 \pm T)}\circ \rho|_{\mathcal{E}_\pm(P) \times \mathcal{E}_\pm(P)}
\colon \mathcal{E}_\pm(P) \times \mathcal{E}_\pm(P) \to \Z.
\]
\end{construction}

\begin{remark}
Note that $\mathcal{E}_{\pm}(P)$ is naturally a $\Z[\Z_2]$-module, where $T$ acts by multiplication by $\pm 1$. But $\mathcal{E}_\pm(P,\lambda,\mu)$ is not a quadratic form over $\Z[\Z_2]$ because $\mathcal{E}_{\pm}(P)$ is not stably free over $\Z[\Z_2]$.
\end{remark}

\begin{remark}
\label{rem:PmFormNB}
If $x,y \in \E_\pm(P)$, then $\lambda(x,y)=2(1\pm T)\E_\pm(\lambda)(x,y)$.
\end{remark}

The following proposition can be straightforwardly verified from the definitions and the fact that $(P,\lambda,\mu)$ is a quadratic form.

\begin{proposition}
 The data $\mathcal{E}_\pm(P,\lambda,\mu)$ determines a quadratic form over $\Z$.
If $(P,\lambda,\mu)$ is nonsingular, then so is $\mathcal{E}_\pm(P,\lambda,\mu)$.
\end{proposition}

\begin{definition}
\label{def:PlusMinus}
The \emph{plus form} and the \emph{minus form} of a quadratic form~$(P,\lambda,\mu)$ over $\Z[\Z_2]$ are respectively the quadratic forms $\E_+(P,\lambda,\mu)$ and $\E_-(P,\lambda,\mu)$ over $\Z$.
\end{definition}

\begin{example}\label{example:plugging-in-T-pm1}
Let $\psi \in Q_+(\Z[\Z_2]^n)$ be a quadratic form.
A matrix representing the quadratic form~$\E_\pm(\psi) \in Q_+(\Z^n)$ can be obtained by plugging in $T=\pm 1$ into a matrix representing~$\psi$.

In particular, given a quadratic form $\theta \in Q_+(\Z^n)$,  the plus form of $(1-T)\theta \in Q_+(\Z[\Z_2]^n)$ is zero, while its minus form is~$2\theta \in Q_+(\Z^n)$.
To see this, let $\{e_i\}_{i=1}^n$ be the standard basis for~$\Z[\Z_2]^n$ and pick a sesquilinear form $\rho$ representing $\psi=[\rho] \in Q_+(\Z[\Z_2]^n)$.
Then $\{(1 \pm T)e_i\}_{i=1}^n$ is a basis for $\mathcal{E}_{\pm}(\Z[\Z_2]^n)$.
For~$p \in \Z[\Z_2]$ we write $p(\pm1)$ for the evaluation of $p$ at $\pm 1$.
Similarly to above we have
\[\rho((1\pm T)e_i,(1\pm T)e_j)= (1\pm T)
\rho(e_i,e_j)\ol{(1 \pm T)} =  2(1\pm T)\rho(e_i,e_j) = 2(1\pm T) \big(\rho(e_i,e_j)(\pm 1)\big).\]
Therefore the $(i,j)$ entry of a matrix representing $\mathcal{E}_{\pm}(\rho)$ is
\[\mathcal{E}_{\pm}(\theta)((1\pm T)e_i,(1\pm T)e_j) = \smfrac{1}{2(1 \pm T)} \Big(2(1\pm T) \big(\rho(e_i,e_j)(\pm 1)\big)\Big) = \rho(e_i,e_j)(\pm 1),\]
which arises by evaluating as asserted.
\end{example}

\subsection{Plus and minus quasi-formations}
\label{sub:PlusMinusquasi-formation}

Given a quasi-formation over $\Z[\Z_2]$, we describe how plus and minus forms can be used to produce two quasi-formations over $\Z$.
\medbreak
First of all,  observe that if $i \colon V \hookrightarrow P$ is a split injection
then the induced map
$i_\pm \colon \mathcal{E}_\pm(V) \to \mathcal{E}_\pm(P)$
 is  also a split injection.


\begin{proposition}
\label{cons:PlusMinusquasi-formation}
If~$x:=((P,\psi);F,V)$ is a quasi-formation over~$\Z[\Z_2]$,
then the triple
$$\mathcal{E}_\pm(x)= (\mathcal{E}_\pm(P,\psi);\mathcal{E}_\pm(F),\mathcal{E}_\pm(V))$$
 is a quasi-formation over~$\Z$.
\end{proposition}

\begin{proof}
 First observe (e.g. using Example~\ref{example:plugging-in-T-pm1}) that $\E_\pm(P,\psi)$ is hyperbolic over~$\Z$.
Next we observe that since $F,V \subseteq P$  are summands,  we have inclusions of summands $\E_\pm(F),\E_\pm(V) \subseteq \E_\pm(P)$.
Finally, one verifies that
$\E_\pm(F) \subseteq \E_\pm(P,\psi)$
is a lagrangian and
$\E_\pm(V)$
is a half rank direct summand.
\end{proof}

Using~$(V,\theta)$ and $(V^\perp,-\theta^\perp)$ to denote the induced forms of a quasi-formation $((P,\psi);F,V)$ (Definition~\ref{def:InducedForm}),  note that the induced form of
~$\E_\pm(P,\psi)$ on~$\E_\pm(V)$ and $\E_\pm(V^\perp)$  coincide with~$\E_\pm(\theta)$ and~$\E_\pm(-\theta^\perp)$ respectively.

\begin{proposition}
The following assignments define monoid maps
$$f_\pm \colon \ell_5(\Z[\Z_2]) \to \ell_5(\Z),  [x]\mapsto [\E_\pm(x)].$$
  \end{proposition}
\begin{proof}
Since $f_\pm$ maps trivial formations to trivial formations,  we deduce that it preserves stable isomorphisms.
The reader can then verify that $f_\pm$ maps quasi-formations of the form~$((P,\psi);F,V)$ and~$((P,\psi);F,G)\oplus ((P,\psi);G,V)$ (which are equivalent in $\ell_5(\Z[\Z_2]$) to equivalent quasi-formations in $\ell_5(\Z)$.
Additivity follows as well.
\end{proof}

In what follows, we will often write $\E_\pm(\Theta)=[\E_\pm(x)]$ instead of $f_\pm(\Theta)$ and refer to $\E_+(x)$ and~$\E_-(x)$ as the \emph{plus and minus quasi-formations} of $x$.

\begin{remark}
\label{rem:PlusMinusl5vv}
Given quadratic forms~$v,v'$ over~$\Z[\Z_2]$, the monoid map~$f_\pm$ from Proposition~\ref{cons:PlusMinusquasi-formation} restricts to a map~$f_\pm \colon \ell_5(v,v') \to  \ell_5(\E_\pm(v),\E_\pm(v'))$.
This follows from the observation
that taking induced forms commutes with taking~$\pm$-forms/formations i.e.\ the following diagram commutes:
\[\begin{tikzcd}
    \ell_5(\Z[\Z_2]) \ar[r,"f_{\pm}"] \ar[d,"b"] & \ell_5(\Z) \ar[d,"b"] \\
\mathcal{F}^{\zs}_{4}(\Z[\Z_2]) \times \mathcal{F}^{\zs}_{4}(\Z[\Z_2])) \ar[r,"f_{\pm}"]  & \mathcal{F}^{\zs}_{4}(\Z) \times \mathcal{F}^{\zs}_{4}(\Z)).
  \end{tikzcd}\]
\end{remark}

\subsection{Finding a lagrangian complement}
\label{sub:KeyLemma}

The following lemma is a technical ingredient towards establishing our criterion for showing that a quasi-formation over $\Lambda = \Z[\Z_2]$ is elementary.
The lemma will justify that, for this purpose, it suffices to understand the plus and minus quasi-formations.
A sketch was given on~\cite[p.~70]{KreckOnTheHomeomorphism}.
 We follow the same overall strategy, but the proof given below is considerably more detailed, and includes arguments not found in Kreck's sketch.

\begin{lemma}
\label{lem:InfamousLemma}
Suppose $(P,\lambda)$ is a nonsingular, even, hermitian form on a free $\Lambda$-module $P$, with a free half-rank direct summand $V\subseteq P$.
Suppose that $\E_+(V)$ is a lagrangian for $(\E_+(P),\E_+(\lambda))$ and that $\E_-(V)\subseteq \E_-(P)$ admits a lagrangian complement.
Then $V\subseteq P$ admits a lagrangian complement.
\end{lemma}

\begin{proof}
We will frequently use the following fact.
Let~$U$ be a complementary free direct summand to~$V$ in~$P$. Let~$\{u_i\}$ be a basis for~$U$ and let~$\{v_i\}$ be a basis for~$V$.
Add linear combinations of basis elements of~$V$ to the basis elements for~$U$, to obtain a set~$\{u_i + \sum_j a_{ij} v_j\}$, for some~$a_{ij} \in \Lambda$.  An elementary linear algebra argument shows that this set is a basis for a new complementary direct summand~$U'$ to~$V$ in~$P$,  for any $a_{ij} \in \Lambda$.

%

Now we begin the proof.
Choose a basis~$x_1,\dots,x_n$ for~$V$ and complete it to a basis of~$P$ via~$v_1,\ldots,v_n$.
Since~$P$ is free we may consider the dual basis elements~$x_i^* \in P^*$ for~$i=1,\ldots,n$ that satisfy~$x_i^*(x_j)=\delta_{ij}$ and~$x^*_i(v_j)=0$ for every~$i,j$.
Since the form~$\lambda$ is nonsingular,  we can choose~$u_1,\ldots,u_n \in P$ such that~$\widehat{\lambda}(u_i)=x_i^*$.
In particular, we have~$\lambda(x_i,u_j)=\delta_{ij}$ for every~$i,j$.

We assert that the~$x_j$ and~$u_j$ generate~$P$. We will prove this by showing that each~$v_i$ can be expressed as a linear combination of the~$u_j$ and~$x_j$.
Write~$u_i=\sum_k a_{ik}x_k+\sum_l b_{il}v_l$ for some~$a_{ik},b_{il} \in \Lambda$.
Arrange the~$u_i,x_k,v_l$ as column vectors~$\mathbf{u},\mathbf{x},\mathbf{v}$ and the~$a_{ij},b_{il}$ as matrices~$A,B$ so that,  formally, we can write~$\mathbf{u}=A\mathbf{x}+B\mathbf{v}$.
Since we have~$\lambda(u_i,x_j)=\delta_{ij}$ and since the form~$\E_+(\lambda)$ vanishes on $\E_+(V) \times \E_+(V)$, we obtain
\[ \delta_{ij}
=\lambda(u_i,x_j)
=\E_+(\lambda)((1+T)u_i,(1+T)x_j)
=\sum_{l=1}^n  b_{il} \cdot \E_+(\lambda)((1+T)v_l,(1+T)x_j).\]
The second equality here uses that $\lambda(u_i,x_j) \in \Z$.
Write~$L$ for the matrix with coefficients~$L_{lj}=\E_+(\lambda)((1+T)v_l,(1+T)x_j)$ and~$I_n$ for the size~$n$ identity matrix so that~$I_n=BL$.
A square matrix with coefficients in a commutative ring admits an inverse if and only if its determinant is a unit. That inverse is both a left and a right inverse. The equalities~$1=\det(I_n) = \det(BL) = \det(B) \det(L)$ imply that ~$\det (B)$ is a unit.
We deduce that~$B$ is invertible with inverse~$L$.
Therefore~$\mathbf{u}=A\mathbf{x}+B\mathbf{v}$ can be rewritten as~$\mathbf{v}=L(\mathbf{u}-A\mathbf{x})$.
Thus the~$v_i$ can be expressed as linear combinations of the~$u_j$ and~$x_j$.
This concludes the proof of the assertion the~$x_j$ and~$u_j$ generate~$P$.

Since~$\{x_j\} \cup \{u_j\}$ form a rank~$n$ generating set of~$P \cong \Lambda^n$, they are also linearly independent.
This statement follows from the general fact, which we now prove, that a surjective~$\Lambda$-linear map~$f \colon \Lambda^n \to \Lambda^n$ is necessarily injective.
Let~$A$ be a matrix for~$f$ with respect to the canonical basis~$e_1,\ldots,e_n$ of~$\Lambda^n$.
Surjectivity implies that there are column vectors~$b_1,\ldots,b_n$ such that~$f(b_i)=e_i$.
Arrange the column vectors~$b_i$ into a matrix~$B$ such that~$AB=I_n$.
Taking determinants, as above we deduce that~$f$ is also injective, as claimed.

Since~$\{x_j\} \cup \{u_j\}$ forms a basis of~$P$, it follows that the~$u_i$ span a half rank direct summand $U_0 \subseteq P$ with~$P = V \oplus U_0$ and~$\lambda(x_i,u_j)=\delta_{ij}$.

Let~$U_-\subseteq \E_-(P)$ be the hypothesised complementary lagrangian to~$\E_-(V)\subseteq \E_-(P)$.
First, we will modify~$U_0$ to obtain a direct summand~$U_1\subseteq P$ that is complementary to~$V$ and such that~$\E_-(U_1)=U_-$.
Then we will improve~$U_1$ further, to obtain a lagrangian complement~$U_2$ for~$V$.

Write~$b_i:=(1-T)x_i$ for the basis of~$\E_-(V)$ obtained from~$\{x_i\}$.
Recall that we have a basis~$u_1,\dots,u_n$ for~$U_0$ and that~$\{x_1, \dots,x_n,u_1,\dots,u_n\}$ is a basis for~$P$.
Write~$(1-T)u_i=w_i+z_i\in\E_-(U_0)$ with respect to the decomposition~$\E_-(P)\cong \E_-(V)\oplus U_-$, i.e.\ $w_i \in \E_-(V)$ and~$z_i \in U_-$ for each~$i$.
Write~$w_i=\sum_j n_{ij}b_j$ for some~$n_{ij} \in \Z$.
Define~$U_1\subseteq P$, a new complementary direct summand to~$V$, as the submodule with basis given by the elements
\[y_i:=u_i-\sum_j n_{ij}x_j.\]
Here we used the fact from the first paragraph of the proof to see that $U_1$ is a direct complement to V.

To confirm that~$\E_-(U_1)=U_-$,
first note that~$\{(1-T)y_i\}$ forms a basis for~$\E_-(U_1)$, because~$\{y_i\}$ forms a basis for~$U_1$.
Next, note that~$\{(1-T)x_1,\dots,(1-T)x_n,(1-T)u_1,\dots,(1-T)u_n\} = \{b_1,\dots,b_n,w_1 + z_1,\dots,w_n+z_n\}$ is a basis for~$\E_-(P) = \E_-(V) \oplus \E_-(U_0)$, and therefore so is~$\{b_1,\dots,b_n, z_1,\dots,z_n\}$,  because $w_i \in \E_-(V)$ for all $i$.
 Therefore ~$\{b_1,\dots,b_n, z_1,\dots,z_n\}$ is a basis for~$\E_-(V) \oplus U_- \cong \E_-(P)$, and we deduce that~$\{z_i\}$ forms a basis for~$U_-$.  Now we compute that
\[
(1-T)y_i
=(1-T)u_i-\textstyle{\sum_j} n_{ij} (1-T)x_j
=w_i+z_i-\textstyle{\sum_j} n_{ij} b_j
=z_i.
\]
It follows that~$\E_-(U_1)=U_-$ as asserted.
This completes the~$U_1$ step.  Note that~$U_1$ need not be a lagrangian, so the proof is not yet complete.

At this point, we consider the following subspaces with their respective bases shown underneath each subspace:
\[\begin{array}{ccccccc}
\E_+(V) && \E_-(V) && U_+:=\E_+(U_1) && U_- = \E_-(U_1)\\
\{a_i:=(1+T)x_i\} && \{b_i:=(1-T)x_i\} && \{t_i^+:=(1+T)y_i\} && \{t_i^-:=(1-T)y_i\}.
\end{array}\]
Next we modify the bases~$t^\pm_i$ to something more favourable. 
Consider the~$\Z$-valued forms
$\mathcal{E}_\pm(\lambda)$ on~$\E_{\pm}(P) = \E_\pm(V)\oplus {U}_\pm$; this decomposition holds because~$P = V \oplus U_1$. With respect to this decomposition, these forms are represented by block diagonal matrices over~$\Z$,
say $\mathcal{E}_\pm(\lambda)=\bsm A_\pm & Y_\pm \\ Y_\pm^T &B_\pm \esm$.
Observe for later use that ~$\E_+(\lambda)|_{U_+ \times U_+}$ is even.
To see this note that~$\lambda(y_i,y_i) \in 2\Lambda$ because~$\lambda$ is even and~$p + \ol{p} \in 2 \Lambda$ for every~$p \in \Lambda$.
Therefore
 \[\E_+(\lambda)(t_i^+,t_i^+)= \frac{1}{2(1+T)} \lambda((1+T)y_i,(1+T)y_i) = \varepsilon(\lambda(y_i,y_i)) \in 2\Z,\]
where the last equality uses~$\lambda(y_i,y_i) \in 2\Lambda$ and~$\varepsilon \colon \Lambda \to \Z$ denotes the augmentation map,  defined by~$\varepsilon(a+bT)=a+b$.

In the ``$+$'' case, ~$A_+=0$ because~$\E_+(V)$ is a lagrangian. Thus, since~$\mathcal{E}_+(\lambda)$ is nonsingular,  we deduce that~$Y_+$ is nonsingular. Using column operations,~$Y_+$ can be made into the identity matrix. This corresponds to performing elementary basis change operations on the basis~$\{t_i^+\}$ of~${U}_+$ (using only linear combinations of the~$t_i^+$).
We denote the resulting basis of~$U_+$ by~$c_1,\dots,c_n$.
Note that $\E_+(\lambda)(a_i,c_j) = \delta_{ij}$.
Since~$c_i \in U_+ \subseteq P$ and~$a_i \in \E_+(V) \subseteq P$, we have that~$\lambda(a_i,c_j)=2(1+T)\delta_{ij}$ (recall Remark~\ref{rem:PmFormNB}).
In fact, in the~$+$-case, since~$\E_+(\lambda)|_{U_+ \times U_+}$ is even, by adding linear combinations of the~$a_i$, we can assume that~$\E_+(\lambda)(c_i,c_j)=0$ for all~$i,j$.
Note that~$U_+$ may not satisfy~$\E_+(U_1)=U_+$ anymore, since we changed the basis for~$U_+ \subseteq \E_+(U_1)$ but did not make the corresponding basis change for~$U_1 \subseteq P$,  but this does not affect the remainder of the argument.

In the ``$-$'' case, ~$B_-=0$ because~$U_-$ is a lagrangian. Thus, since~$\mathcal{E}_-(\lambda)$ is nonsingular, we deduce that~$Y_-$ is nonsingular. Arguing as in the ``$+$'' case, we perform elementary basis change operations on~$t_i^-$ to obtain a basis~$d_1,\dots,d_n$ of~${U}_-$ such that~$\lambda(b_i,d_j)=2(1-T)\delta_{ij}$, again because~$d_j \in U_- \subseteq P$ and~$b_i \in \E_-(V) \subseteq P$ (recall Remark~\ref{rem:PmFormNB}).

We summarise the outcome. With respect to the basis~$\{a_1,\ldots,a_n,c_1,\ldots, c_n\}$ of~$\mathcal{E}_+(P)$ and $\{b_1,\ldots,b_n, d_1,\ldots, d_n\}$ of~$\mathcal{E}_-(P)$, matrices for~$\mathcal{E}_+(\lambda)$ and~$\mathcal{E}_-(\lambda)$) are given by:
\[ \mathcal{E}_+(\lambda)=\begin{pmatrix}
0& I \\
 I&0
\end{pmatrix} \ \ \quad \text{and} \quad \ \
\mathcal{E}_-(\lambda)=\begin{pmatrix}
*& I \\
I &0
\end{pmatrix}.
  \]
 We claim that, for all $i,j$, we have:
 \begin{equation}\label{eqn:lambda-facts}
   \lambda(a_i,a_j)=0,\,\, \lambda(a_i,c_j)=2(1+T)\delta_{ij},\,\, \lambda(b_i,d_j)=2(1-T)\delta_{ij}, \text{ and } \lambda(c_i,c_j)=0=\lambda(d_i,d_j).
 \end{equation}
An argument is needed to see that the first and the last two equalities hold.  We will just argue for~$\lambda(c_i,c_j)=0$; the others proceed similarly.
Since~$\E_+(\lambda)(c_i,c_j)=0$, then~$\lambda(c_i,c_j) \in 2\E_+(\Lambda)$ maps to zero under the identifications~$2\E_+(\Lambda) \cong 2\Z \cong \Z$, and therefore~$\lambda(c_i,c_j)=0$.

Now, using the same elementary basis change operations we just used to change~$t_i^-$ into~$d_i$, transform the basis~$\{y_i\}$ of~$U_1$ into a new basis~$\{\rho'_i\}$ of $U_{{1}}$.
For this we take the basis changes over~$\Z$ and consider them as basis changes over~$\Lambda$ via the ring homomorphism~$\Z \to \Lambda$.
By construction we also have~$(1-T)\rho'_i=d_i$.
Write~$(1+T)\rho'_i=\alpha_i+\gamma_i\in\E_+(V)\oplus U_+$ for some~$\alpha_i \in \E_+(V)$ and~$\gamma_i \in U_+$, and note that
\[
\rho'_i=\frac{\alpha_i+\gamma_i+d_i}{2}.\]

Next we claim we can make a further helpful basis change from the $\rho_i'$ to a new basis for $U_1$, given by
\[
\rho_i=\frac{\alpha_i+c_i+d_i}{2}.
\]
We must prove $\alpha_i+c_i+d_i$ is indeed divisible by $2$ and that the resulting $\rho_i$ form a basis, then after that we will modify~$U_1$ to obtain the desired lagrangian complement~$U_2$ to~$V$.

To begin proving the claim that $\alpha_i+c_i+d_i$ is divisible by $2$ and that $\{\rho_i\}$ forms a basis for~$U_1$, first observe that
since the~$\lbrace c_i \rbrace$ form a~$\Z$-basis of~$U_+$, for each~$j$ we can write~$\gamma_j=\sum_k\gamma_{jk}c_k$ with~$\gamma_{jk} \in \Z$.
It is also the case that the~$c_k \in \E_+(P)$ can be written as~$c_k=(1+T)z_k$, for some~$z_k \in P$. We note for later use that
\begin{equation}
\label{eq:MuchNeededDetails}
2\lambda(x_i,c_k)=\lambda(x_i,(1+T)(1+T)z_k)=\lambda((1+T)x_i,(1+T)z_k)=\lambda(a_i,c_k).
\end{equation}
Similarly~$2\lambda(x_i,\alpha_j) = \lambda(a_i,\alpha_j)$ and~$2\lambda(x_i,(1-T)\rho'_j) = \lambda(b_i,(1-T)\rho'_j)$.
Using~\eqref{eq:MuchNeededDetails},  we observe that
\begin{align*}
4\lambda(x_i,\rho'_j)
&=2\lambda(x_i,\alpha_j)+2\lambda(x_i,\gamma_j)+2\lambda(x_i,d_j) \\
&=2\lambda(x_i,\alpha_j)+2\sum_k \lambda(x_i,\gamma_{jk}c_k)+2\lambda(x_i,(1-T)\rho'_j) \\
&=\lambda(a_i,\alpha_j)+\sum_k \lambda(a_i,\gamma_{jk}c_k)+\lambda(b_i,(1-T)\rho'_j) \\
&=0+2(1+T)\gamma_{ji}+\lambda(b_i,d_j) \\
&=2(1+T)\gamma_{ji}+2(1-T)\delta_{ij} \in \Lambda.
\end{align*}
For the last two equalities we applied \eqref{eqn:lambda-facts}.
Since $\Lambda$ is torsion-free as an abelian group,  we deduce that
\begin{equation}
\label{eq:LastMinuteFix}
2\lambda(x_i,\rho'_j)=(1+T)\gamma_{ji}+(1-T)\delta_{ij}.
\end{equation}
Write $\lambda(x_i,\rho'_j)=r_{ij}+s_{ij}(1-T)$ with $r_{ij},s_{ij}\in \Z$; this is possible because $a+bT=a+b-(1-T)b$ for every $a+bT \in \Z[\Z_2].$
Plugging in $T=1$ into~\eqref{eq:LastMinuteFix} gives $r_{ij}=\gamma_{ji}$.
It follows that setting~$T=-1$ into~\eqref{eq:LastMinuteFix}, we get $2\delta_{ij}=2\gamma_{ji}+4s_{ij}$, which we write as $2s_{ij}=\delta_{ij}-\gamma_{ji}$.
Now define
\[
\rho_i:=\rho'_i+\sum_ks_{ki}c_k=\frac{\alpha_i+d_i}{2}+\frac{1}{2}\gamma_i+\frac{1}{2}\sum_k(\delta_{ki}-\gamma_{ik})c_k=\frac{\alpha_i+d_i}{2}+\frac{1}{2}\sum_k\delta_{ki}c_k= \frac{\alpha_i+c_i+d_i}{2}.
\]
Since $c_i \in U_+=\E_+(U_1) \subseteq U_1$ for $i=1,\ldots,n$ and $\lbrace \rho'_1,\ldots,\rho'_n \rbrace$ forms a basis of $U_1$,  it follows that $\lbrace \rho_1,\ldots,\rho_n \rbrace$ again forms a basis for $U_1$.
This concludes the proof of the claim.

Next we  observe some symmetries in the coefficients of~$\alpha_i \in \mathcal{E}_+(V)$ with respect to the~$\Z$-basis~$a_i$ of~$\E_+(V)$.
Write
$$\alpha_i=\sum_j\alpha_{ij}a_j$$
for some~$\alpha_{ij}\in \Z$.
Recall that~$\lambda(\alpha_i,\alpha_j)=0$, $\lambda(c_i,c_j)=0$, $\lambda(d_i,d_j)=0$, and $\lambda(a_i,c_j)=2(1+T)\delta_{ij}$, by~\eqref{eqn:lambda-facts}. We use these and the fact that~$\lambda|_{\E_+(P) \times \E_-(P)}=0$,  which follows from the equality $(1+T)(1-T)=0$, to
compute that
\begin{align*}
\lambda(\rho_i,\rho_j)&=\textstyle{\frac{1}{4}}\lambda(\alpha_i+c_i+d_i,\alpha_j+c_j+d_j)\\
&=\textstyle{\frac{1}{4}}\left(\lambda(\alpha_i,c_j)+\lambda(c_i,\alpha_j)\right)\\
&=\textstyle{\frac{1}{4}}\left( \lambda\left(\sum_k \alpha_{ik}a_k,c_j \right)+\lambda\left(c_i,\sum_s\alpha_{js}a_s\right)\right)\\
&=\textstyle{\frac{1}{4}}\cdot 2(1+T)\cdot\left(\alpha_{ij}+\alpha_{ji} \right)\\
&=\textstyle{\frac{1}{2}}(1+T)\cdot\left(\alpha_{ij}+\alpha_{ji}\right).
\end{align*}
Because~$\lambda(\rho_i,\rho_j)\in\Lambda$ and~$\alpha_{ij} \in \Z$, we deduce that~$\alpha_{ij}+\alpha_{ji} \in 2 \Z.$

We now modify the summand~$U_1\subseteq P$ into a lagrangian complement to~$V$. As~$V_+\subseteq V$, so~$a_j\in V$ for all~$j$. Therefore if we modify~$\rho_i$ by adding multiples of~$a_j$ we will obtain a new direct summand of~$P$ that is still complementary to~$V$. Of course, the objective is to do this so that the resulting direct summand is self-annihilating. Consider that for~$p\in\Lambda$ we have
\[\rho_i+pa_j=\frac{\alpha_i+2pa_j+c_i+d_i}{2}\]
so the proposed operation modifies the coefficient~$\alpha_{ij}$ of $a_j$ in $\alpha_i$ by the addition of~$2p$. Using this operation, and the symmetries in the~$\alpha_{ij}$ (namely $\alpha_{ij}+\alpha_{ji} \in 2 \Z$), we may now modify the coefficients~$\alpha_{ij}$ to~$\alpha_{ij}'$,  which satisfy that~$\alpha'_{ij}+\alpha'_{ji}=0$ for all~$i,j$.
We obtain a new direct summand~$U_2$ of~$P$, which is still complementary to~$V$ by the first paragraph of this proof.
We write~$\rho_i'$ for the modified basis element coming from~$\rho_i$.
The same computation as we made for~$\lambda(\rho_i,\rho_j)$ above now shows that~$\lambda(\rho'_i,\rho'_j)=0$ for all~$i,j$  (here it is important that~$\lambda(a_i,a_j)=0$ for all~$i,j$), and thus~$U_2$ is a complementary {sub}lagrangian to~$V$, i.e.~$U_2\subseteq U_2^\perp$.
Since~$U_2 \subseteq P$ is a half rank sublagrangian and~$\lambda$ is nonsingular, it must in fact be a lagrangian; see e.g.~\cite[Proposition~11.53]{RanickiAlgebraicAndGeometric}.
We have therefore found a complementary lagrangian to~$V \subseteq P$.
\end{proof}

\subsection{A criterion for being elementary in~$\texorpdfstring{\ell_5(\Z[\Z_2])}{l5(Z[Z/2])}$}
\label{sub:ElementaryCriterion}
We now prove the promised criterion, which ensures that a class~$\Theta \in \ell_5(\Z[\Z_2])$ is elementary.
\medbreak

\begin{lemma}
\label{lem:HaveLagrangianComplement}
If $\Theta \in \ell_5(\Z[\Z_2])$ is such that $\E_-(\Theta)$ is elementary, then $\Theta$ is represented by a quasi-formation $x'=((P',\psi');F',V')$ such that $\E_-(V')\subseteq \E_-(P')$ admits a lagrangian complement.
\end{lemma}

\begin{proof}
Let $x=((P,\psi);F,V)$ be a representative of $\Theta$. Since $\E_-(\Theta)$ is elementary, Proposition~\ref{prop:CleverCorollary} implies that there is a formation~$f=(H_{+}(L);L,G)$ over $\Z$ such that $\E_-(V) \oplus G$ admits a lagrangian complement in $\E_-(P) \oplus (L\oplus L^*)$.
Now consider the formation
$$f \otimes_\Z \Z[\Z_2]:=(H_+(L \otimes_\Z \Z[\Z_2]);L\otimes_\Z \Z[\Z_2],G\otimes_\Z \Z[\Z_2]).$$
and set $x':=x\oplus(f \otimes_\Z \Z[\Z_2])$.
Note that $\E_-(f \otimes_\Z \Z[\Z_2]) \cong f$ as can be seen by reasoning in terms of bases for the Lagrangians of $f$.

Since $L_5(\Z[\Z_2])=0$,  observe that $[x']=[x]=\Theta$
 and, by construction,
\begin{align*}
 \E_-(x')
 &=\E_-(x) \oplus \E_-(f \otimes_\Z \Z[\Z_2]) \\
 &= (\E_-(P,\psi),\E_-(F),\E_-(V)) \oplus f \\
 &=(\E_-(P,\psi)\oplus H_+(L),\E_-(F) \oplus L, \E_-(V) \oplus G)
 \end{align*}
 is such that $\E_-(V) \oplus G$ admits a Lagrangian complement in $\E_-(P) \oplus (L \oplus L^*)$.
\end{proof}

\begin{proposition}
\label{prop:ElementaryPlusMinus}
An element $\Theta \in \ell_5(\Z[\Z_2])$ is elementary if it satisfies both:
\begin{enumerate}
\item $\E_+(\Theta)$ is represented by a formation, so $\E_+(\Theta) \in L_5(\Z) = \{0\} \subseteq \ell_5(\Z)$,  and
\item $\E_-(\Theta)$ is elementary.
\end{enumerate}
\end{proposition}

\begin{proof}
Let $\Theta \in \ell_5(\Z[\Z_2])$ and assume that (1) and (2) hold.
Thanks to Lemma~\ref{lem:HaveLagrangianComplement} we can represent $\Theta$ as $\Theta = [x] = [((P,\psi);F,V)]$ where the quasi-formation~$x=((P,\psi);F,V)$ is such that~$\E_-(V) \subseteq \E_-(P)$ admits a lagrangian complement.

We assert that $\E_+(V) \subseteq \E_+(P,\psi)$  is a lagrangian.
Proposition~\ref{prop:InducedFormsMonoid} ensures that  the induced forms of~$\E_+(\Theta)=[\E_+(x)] \in L_5(\Z)$ are independent up to $0$-stabilisation of the quasi-formation representative of~$\E_+(\Theta)$ and therefore of $\Theta$.
Since $[\E_+(x)]$ is represented by a formation, the induced form, which lies in the 0-stabilisation class of~$\E_+(V,\psi)$, is trivial. It follows that $\E_+(\psi)$ vanishes on $\E_+(V)$.
Since~$\E_+(V)$ is a half rank direct summand, $\E_+(V)$ is a lagrangian as asserted.

By the assertion and the hypothesis that $\E_-(V)\subseteq \E_-(P)$ admits a lagrangian complement, the hypotheses of Lemma~\ref{lem:InfamousLemma} are satisfied for $(P,\lambda)$, where $\lambda$ is the symmetrisation of~$\psi$. Here we should also remark that since $P$ is stably free is it free by Lemma~\ref{lem:StablyFreeImpliesFree},  and that $\lambda$ is even and nonsingular since it arises as the symmetrisation of the quadratic form $\psi$ in a quasi-formation.
Lemma~\ref{lem:InfamousLemma} therefore implies that~$V \subseteq P$ admits a lagrangian complement $L$,  with respect to~$\lambda$.
As $L$ is a direct summand of the free group $P$, it is projective. By Corollary~\ref{cor:niceconsequences}, $L$ is thus also a lagrangian with respect to the quadratic form $\psi$.
Thus by Proposition~\ref{prop:ElementaryForTrivalL5}, and since~$L_5(\Z[\Z_2])=0$~\cite[Theorem~13A.1]{WallSurgeryOnCompact},  the class~$\Theta=[x] \in \ell_5(\Z[\Z_2])$ is elementary.
\end{proof}

\begin{proposition}
\label{prop:NewElementaryCriterion}
Fix an integer~$h \geq 0$.
If $(\Z^h,\theta)$ is a nondegenerate quadratic form over $\Z$ such that $\ell_5(\Z^h,2\theta)$ is trivial, then every $\Theta \in \ell_5(\Z[\Z_2]^h,(1-T)\theta)$ is elementary.
\end{proposition}
\begin{proof}
By Example~\ref{example:plugging-in-T-pm1}, the induced positive and negative forms are given by evaluation $T= \pm 1$ respectively.
Thus the induced positive form of $(\Z[\Z_2]^h,(1-T)\theta)$ is trivial, and the induced negative form of~$(\Z[\Z_2]^h,(1-T)\theta)$ is $(\Z^h,2\theta)$.

Write~$\Theta=[((P,\psi);F,V)]$.
The fact that~$\mathcal{E}_+(\Z[\Z_2]^h,(1-T)\theta)$
is the zero form implies, as in the proof of Proposition~\ref{prop:ElementaryPlusMinus}, that~$\E_+(V)$ is a lagrangian.
It follows that $\E_+(\Theta) \in L_5(\Z)$.
In addition, $\E_-(\Theta)  \in \ell_5(\Z^h,2\theta)$, and so is necessarily elementary because $\ell_5(\Z^h,2\theta)$ is trivial by assumption.
The proposition now follows from Proposition~\ref{prop:ElementaryPlusMinus}.
\end{proof}

\section{Analysis of the surgery obstruction}
\label{sec:Obstruction}

We summarise the outcome of the previous two sections on modified surgery in the next theorem.
We then explain how we are going to apply this theorem to surface exteriors.

\begin{theorem}
\label{thm:Recap}
Let~$M_0$ and $M_1$ be two~$4$-manifolds with fundamental group~$\Z_2$ and normal~$1$-type~$(B,\xi)$.
For $i=0,1$ let~$\overline{\nu}_i \colon M_i \to B$ be normal~$1$-smoothings into~$(B,\xi)$, and  let~$f \colon \partial M_0 \to \partial M_1$ be a homeomorphism that is compatible with the~$1$-smoothings, i.e.\ $\ol{\nu}_1|_{\partial M_1} \circ f = \ol{\nu}_0|_{\partial M_0}$.
Suppose that $(M_0 \cup_f -M_1, \ol{\nu}_0 \cup - \ol{\nu}_1)$ is null-bordant over $(B,\xi)$.

Let~$(\Z^h,\theta) \in Q_+(\Z^h)$ be a nondegenerate quadratic form such that $\ell_5(\Z^h,2\theta)$ is trivial, and such that~$(\Z[\Z_2]^h,(1-T)\theta)$ is a free Wall form for $M_0$.
Assume that for both $i=0$ and $i=1$, the pair $(M_i,\ol{\nu}_i)$ admits a free Wall form that is $0$-stably equivalent to~$(\Z[\Z_2]^h,(1-T)\theta)$.
Then~$f$ extends to a homeomorphism~$M_0 \xrightarrow{\cong} M_1$.
\end{theorem}

\begin{proof}
By assumption~$(M_0,\overline{\nu}_0)$ and~$(M_1,\overline{\nu}_1)$ are bordant over their normal~$1$-type,  say via a cobordism~$(W,\overline{\nu})$.
Consider the modified surgery obstruction~$\Theta(W,\overline{\nu}) \in \ell_5(\Z[\Z_2])$.
Proposition~\ref{prop:ConsequenceProp8} states that~$\Theta(W,\overline{\nu})  \in \ell_5(\Z[\Z_2]^h,(1-T)\theta)$.
Since $\ell_5(\Z^h,2\theta)$ is trivial, Proposition~\ref{prop:NewElementaryCriterion} ensures that~$\Theta(W,\overline{\nu})$ is elementary.

Theorem~\ref{thm:Kreck} implies that~$(W,\overline{\nu})$ is bordant rel.\ boundary over~$(B,\xi)$ to an~$s$-cobordism.
The~$5$-dimensional~$s$-cobordism theorem~\cite[Theorem~7.1A]{FreedmanQuinn}, which applies since~$\Z_2$ is a good group (every finite group is good see e.g.~\cite[Theorem 19.2]{DETBook}) with trivial Whitehead torsion, then implies that~$f$ extends to a homeomorphism~$M_0 \xrightarrow{\cong}M_1$.
\end{proof}

\subsection{The free Wall form of $X_F$}
\label{sub:FreeWallSurface}
We will set up some language to describe a specific free Wall form (see Definition~\ref{def:FreeWall}) for the exterior~$X_F$ of a~$\Z_2$-surface~$F$,  which has the convenient property of being directly computable from the integral intersection form~$Q_{\widetilde{X}_F}$ of $\widetilde{X}_F$.
Since~$X_F$ spin, so is~$\widetilde{X}_F$, and therefore the integral intersection form~$Q_{\widetilde{X}_F}$ is even~\cite[Proposition 3.3]{FriedlNagelOrsonPowell}.
By Proposition~\ref{prop:InterectionFormExterior}
we have
$$H_2(\widetilde{X}_F)/\operatorname{rad}(Q_{\widetilde{X}_F}) \cong\Z_-^h.$$
The form $Q_{\widetilde{X}_F}$ descends to a nondegenerate form $Q_{\widetilde{X}_F}^{\nd}$ on this quotient.
Because~$Q_{\widetilde{X}_F}$ is even, so is $Q_{\widetilde{X}_F}^{\nd}$.
Thus~$Q_{\widetilde{X}_F}^{\nd}$ determines a unique nondegenerate quadratic form~$(\Z^h,\theta^{\nd}_{\widetilde{X}_F})$ over $\Z$.
Now, remembering the~$\Z[\Z_2]$-module structure in order to identify~$\Z_-^h=\Z[\Z_2]^h \otimes_{\Z[\Z_2]} \Z_-$, we will think of the integral quadratic form~$\theta^{\nd}_{\widetilde{X}_F}~$ as an equivalence class of a pairing:
\begin{equation}\label{eqn:theta-nd-rho}
\theta^{\nd}_{\widetilde{X}_F} =\big[ \rho\colon (\Z[\Z_2] \otimes_{\Z[\Z_2]} \Z_-)^h \times (\Z[\Z_2] \otimes_{\Z[\Z_2]} \Z_-)^h \to \Z \big].
\end{equation}
Recall this is not a quadratic form over~$\Z[\Z_2]$ because the module $(\Z[\Z_2] \otimes_{\Z[\Z_2]} \Z_-)^h$ is not stably free.
The next proposition establishes a formula for a free Wall form for~$X_F$.

\begin{proposition}
\label{prop:SurfaceData}
Let~$F \subseteq D^4$ be a~$\Z_2$-surface of nonorientable genus~$h$, and let~$\overline{\nu} \colon X_F \to B:=\BTOPSpin \times B\Z_2$ be a normal~$1$-smoothing.
Then a free Wall form for $X_F$ is given by $(\Z[\Z_2]^h,\vartheta)$, where the quadratic form $\vartheta$ is represented by the pairing
\[
 \Z[\Z_2]^h \times \Z[\Z_2]^h\to \Z[\Z_2],
\qquad (x,y)\mapsto
(1-T)\rho(x \otimes 1,y\otimes 1).
\]
We write $\vartheta=(1-T)\theta^{\nd}_{\widetilde{X}_F}.$
\end{proposition}

\begin{proof}
First we describe a specific free Wall form for $X_F$, and then we show that it can indeed be expressed as claimed in the proposition.
Since~${\TOPSpin}$ is simply-connected, $\pi_2(\BTOPSpin)=0$, and therefore we deduce that~$\pi_2(B) \cong \pi_2(\BTOPSpin \times B\Z_2)=0$.
Since~$K\pi_2(X_F)=\ker(\pi_2(X_F) \to \pi_2(B))$ and~$\pi_2(X_F)\cong H_2(X_F;\Z[\Z_2])$,  we deduce that
\[
K\pi_2(X_F)/\operatorname{rad}(\lambda_{X_F})=\pi_2(X_F)/\operatorname{rad}(\lambda_{X_F})=H_2(X_F;\Z[\Z_2])/\operatorname{rad}(\lambda_{X_F})\cong \Z_-^h,
\]
where the last isomorphism follows from Proposition~\ref{prop:RadicalFormUniversalCover} 
We have therefore showed that the $\Z[\Z_2]$-module underlying the nondegenerate Wall form $(K\pi_2(X_F)/\operatorname{rad}(\lambda_{X_F}),\psi_{X_F}^{\nd})$ is defined on $\Z^h$.
The tensor product
\[
\Z[\Z_2]^h \to \Z[\Z_2]^h \otimes_{\Z[\Z_2]} \Z_-,\qquad x \mapsto x \otimes 1
\]
is a surjection of left~$\Z[\Z_2]$-modules.
Identifying~$\Z[\Z_2]^h \otimes_{\Z[\Z_2]} \Z_-=\Z_-^h$, we thus obtain a surjection~$\varpi\colon \Z[\Z_2]^h\to \Z^h_-$. We use~$\varpi$ to pull back the quadratic form~$(K\pi_2(X_F)/\operatorname{rad}(\lambda_{X_F}),\psi_{X_F}^{\nd})$, thus obtaining a representative $(x,y) \mapsto \rho(\varpi(x),\varpi(y))$ for a free Wall form $\vartheta$,  where $\rho$ is as defined in \eqref{eqn:theta-nd-rho}.
Observe that the symmetrisation of this quadratic form is $(x,y) \mapsto \lambda_{X_F}^{\nd}(\varpi(x),\varpi(y))$.

We now establish the expression~$\vartheta=(1-T)\theta^{\nd}_{\widetilde{X}_F}$ for this free Wall form. Recall that if a hermitian form on a free~$\Z[\Z_2]$-module admits a quadratic refinement, then it admits a unique quadratic refinement (Corollary~\ref{cor:niceconsequences}),  and therefore is the symmetrisation of a unique quadratic form.
Observe that the symmetrisation of $\theta^{\nd}_{\wt{X}_F}$ is $Q^{\nd}_{\wt{X}_F}$.
Thus it now suffices to show the symmetrisation~$\lambda$ of~$\vartheta$ is given by the symmetrisation~$(1-T)Q^{\nd}_{\widetilde{X}_F}$ of the claimed expression~$(1-T)\theta^{\nd}_{\widetilde{X}_F}$. To see this, we compute:
\[
\lambda(x,y)=\lambda_{X_F}^{\nd}(\varpi(x),\varpi(y))=\lambda^{\nd}_{X_F}(x \otimes 1,y \otimes 1)=(1-T)Q^{\nd}_{\widetilde{X}_F}(x \otimes 1,y \otimes 1).
\]
Here, the first equality comes from the fact that symmetrisation and pull-back commute, and the final equality follows from the third item of Proposition~\ref{prop:InterectionFormExterior}. This concludes the proof of Proposition~\ref{prop:SurfaceData}.
\end{proof}

\section{Proofs of Theorems~\ref{thm:SurfacesWithBoundaryIntro},~\ref{thm:one-is-enough}, and~\ref{thm:isotopy-of-embeddings}}\label{sec:introproofs}

We now prove Theorem~\ref{thm:SurfacesWithBoundaryIntro}, whose statement we recall for the reader's convenience.  Recall that Theorem~\ref{thm:SurfacesWithBoundaryIntro} implies Theorem~\ref{thm:TopUnknottingIntro}, by taking $K$ to be the unknot.

\begin{customthm}
{\ref{thm:SurfacesWithBoundaryIntro}}
\label{thm:SurfacesWithBoundaryMain}
Let~$F_0,F_1\subseteq D^4$ be $\Z_2$-surfaces of the same nonorientable genus $h$, the same normal Euler number $e$ and the same boundary $K$. Assume that~$|\det(K)|=1$.
If either~$h \leq 3$ or~$e$ is non-extremal,  then~$F_0$ and~$F_1$ are ambiently isotopic rel.\  boundary.
\end{customthm}

\begin{proof}
For $i=0,1$ set~$M_i:=X_{F_i}$ for the~$\Z_2$-surface exteriors. Recall that $B := \BTOPSpin \times B\Z_2 \xrightarrow{\xi} \BSTOP$ is the normal 1-type of $M_i$.
By Proposition~\ref{prop:StablyHomeo} there are normal 1-smoothings $\ol{\nu}_i \colon M_i \to B$ and there is a homeomorphism $f \colon \partial M_0 \to \partial M_1$ such that $(M_0,\ol{\nu}_0)$ and $(M_1,\ol{\nu}_1)$ are $B$-bordant relative to $f$. In other words,  writing
~$M:=M_0 \cup_{f} -M_1$, we have that~$M$ is spin,  $\pi_1(M) \cong \Z_2$, and $(M,\ol{\nu}_0 \cup -\ol{\nu}_1)$ is null-bordant over $B$.
Moreover $f$ restricts to a homeomorphism $S(\nu F_0)\to S(\nu F_1)$ that is $\nu$-extendable rel.~boundary.
Proposition~\ref{prop:SurfaceData} shows that a free Wall form for $M_i$ is given by the~$\Z[\Z_2]$-quadratic form~$(\Z[\Z_2]^h,(1-T)\theta_{\wt{M}_i}^{\nd})$ where~$\theta_{\wt{M}_i}^{\nd}$ denotes the non-degenerate $\Z$-quadratic form determined by the even form $Q_{\widetilde{M}_i}^{\nd}$ on~$H_2(\widetilde{M}_i)/\operatorname{rad}(Q_{\widetilde{M}_i}) \cong \Z_-^h$.
Set $\sigma:=\sign(\Sigma_2(F_i))$, the signature of the $2$-fold cover of $D^4$ branched over~$F_i$.
Recall from Theorem~\ref{thm:MasseyBoundary} that $|\sigma| \neq h$ if and only if $e$ is non-extremal.
By Proposition~\ref{prop:Calculate} we know that
$$\theta^{\nd}_{\wt{M}_i} = \theta_{\widetilde{X}_{F_i}}^{\nd} \cong
\begin{cases}
2H_+(\Z) \oplus H_+(\Z)^{\oplus \frac{h}{2}-1}  &\quad \text{ if } \sigma=0,   \\
X_\sigma \oplus H_+(\Z)^{\oplus \frac{1}{2}(h-|\sigma|)} &\quad \text{ if } \sigma \neq 0.
\end{cases}
$$
Proposition~\ref{prop:ApplyNikulin} states that when $h \leq 3$ or if $|\sigma| \neq h$ (i.e.\  if $e$ is non-extremal), then~$\ell_5(\Z^h,2\theta_{\wt{M}_i}^{\nd})$ is trivial.
Theorem~\ref{thm:Recap} now implies that~$f$ extends to homeomorphism $X_{F_0} \cong X_{F_1}$.
The $\nu$-extendability condition implies that $F_0$ and $F_1$ are equivalent rel.\  boundary.
That is, there is a homeomorphism $G \colon D^4 \to D^4$ with $G|_{S^3} = \Id|_{S^3}$ and $G(F_0)=F_1$. By the Alexander trick, $G$ is isotopic rel.\  boundary to $\Id_{D^4}$ and therefore $F_0$ and $F_1$ are ambiently isotopic rel.\  boundary.
\end{proof}


In the proof of Theorem~\ref{thm:SurfacesWithBoundaryMain}, the assumption $\det(K)=1$ was essentially only used to be able to analyse the isometry type of the intersection form~$Q_{\Sigma_2(F)}$.
In general,  if $F_0$ and $F_1$ have the same nonorientable genus~$h$, the same normal Euler number~$e$, and the same boundary~$K$,  we do not know whether $Q_{\Sigma_2(F_0)} \cong Q_{\Sigma_2(F_1)}$.
Even if this were the case however, there would still be a further  potential obstruction in $\ell_5(\Z^h,2\theta_{\widetilde{X}_{F_i}}^{\nd})$ to surface exteriors being homeomorphic rel.\ boundary.
Nikulin's theorem ensures that if $2\theta_{\widetilde{X}_{F_i}}$ splits off a $2H_+(\Z)$-summand then $\ell_5(\Z^h,2\theta_{\widetilde{X}_{F_i}}^{\nd})$ is trivial.
For general $K$,  we do not know whether this set is trivial, but
this is always the case after we add a single tube to both $F_0$ and $F_1$, as we now prove.
We recall the statement of Theorem~\ref{thm:one-is-enough} for the convenience of the reader.

\begin{customthm}
{\ref{thm:one-is-enough}}
\label{thm:one-is-enough-body}
Let~$F_0,F_1\subseteq D^4$ be~$\Z_2$-surfaces of the same nonorientable genus~$h$, the same normal Euler number~$e$ and the same boundary~$K$.
Then~$F_0 \# T^2$ and~$F_1 \# T^2$ are ambiently isotopic rel.\ boundary.
\end{customthm}

\begin{proof}
Since~$F_0,F_1$ are~$\Z_2$-surfaces with the same nonorientable genus, the same boundary and the same normal Euler number,  a result of Baykur--Sunukjian \cite[Theorem~1]{BaykurSunukjian} (cf.~\cite[Appendix]{ConwayPowell}, where their proof was adapted to the topological category) implies that they become ambiently isotopic rel.\ boundary after adding~$n>0$ unknotted tubes to each.
Adding a tube adds an $S^2 \times S^2$ connected summand to the $2$-fold branched cover, and therefore adds a~$H^+(\Z)=\bsm 0&1 \\ 1&0 \esm$ summand to the intersection form of the~$2$-fold branched cover.
Thus we have
\begin{equation}
\label{eq:OneIsEnoughForms}
Q_{\Sigma_2(F_0 \#^{n} T^2)}=Q_{\Sigma_2(F_0)} \oplus H^+(\Z)^{\oplus n} \cong Q_{\Sigma_2(F_1)} \oplus H^+(\Z)^{\oplus n}= Q_{\Sigma_2(F_1 \#^{n} T^2)}.
\end{equation}
As above we write $M_i :=X_{F_i}$.
Applying Proposition~\ref{prop:BranchedUnbranched} to~\eqref{eq:OneIsEnoughForms} and passing to quadratic forms,  we obtain
\[\theta_{\widetilde{X}_{F_0 \#^{n} T^2}^{\nd}}=\theta_{\wt{M}_0}^{\nd} \oplus H_+(\Z)^{\oplus n} \cong \theta_{\wt{M}_1}^{\nd} \oplus H_+(\Z)^{\oplus n}=\theta_{\widetilde{X}_{F_1 \#^{n} T^2}^{\nd}}.\]
A cancellation result due to Bass~\cite[Corollary IV.3.6]{Bass} (cf.~\cite[Proposition 7.3]{ConwayPowell} for an explanation as to how Bass's general theory applies to this simple setting) gives an isometry~$\theta_{\wt{M}_0}^{\nd} \oplus H_+(\Z) \cong \theta_{\wt{M}_1}^{\nd} \oplus H_+(\Z)$
and thus
\[\theta_{\widetilde{X}_{F_0 \# T^2}^{\nd}}=\theta_{\wt{M}_0}^{\nd} \oplus H_+(\Z)\cong \theta_{\wt{M}_1}^{\nd} \oplus H_+(\Z)=\theta_{\widetilde{X}_{F_1 \# T^2}^{\nd}}.\]
We now follow the same steps as in Theorem~\ref{thm:SurfacesWithBoundaryMain} to verify that $F_0\#T^2$ and $F_1 \# T^2$ are ambiently isotopic rel.\ boundary.
Proposition~\ref{prop:SurfaceData} implies that the free Wall form of $F_i\#T^2$ is~$(1-T)(\theta_{\wt{M}_i}^{\nd} \oplus H_+(\Z))$.
We also know by Proposition~\ref{prop:StablyHomeo} that we can find normal 1-smoothings $\ol{\nu}_{i} \colon X_{F_i} \to B:=\BTOPSpin \times B\Z_2$ and a homeomorphism $f \colon \partial X_{F_0 \# T^2} \cong \partial X_{F_1 \# T^2}$, such that the union along $f$ is spin, has fundamental group $\Z_2$, and is null-bordant over $B$.
Moreover $f$ restricts to a homeomorphism $S(\nu (F_0\#T^2))\to S(\nu (F_1\#T^2))$ that is $\nu$-extendable rel.~boundary.
Thanks to the criterion from Theorem~\ref{thm:Recap}, it now suffices to prove that~$\ell_5((\Z^h,2\theta_{\wt{M}_i}^{\nd}) \oplus 2H_+(\Z))$ is trivial.

We verify the assumptions of Nikulin's Theorem~\ref{thm:Nikulin} for
\[(V,\theta):=(\Z^h,2\theta_{\wt{M}_i}^{\nd}) \oplus 2H_+(\Z). \]
Let~$\lambda$ denote the symmetrisation of $\theta$.
We claim that $\operatorname{rk}(V)=h+2 \geq n_p(\coker(\widehat{\lambda}))+2$ for every odd prime $p$. To see this let \[\lambda' :=  2\theta_{\wt{M}_i}^{\nd} + (2\theta_{\wt{M}_i}^{\nd})^*\] and note that $n_p(\coker(\widehat{\lambda'})) \leq h$. Adding $2H^+(\Z)$ to $\lambda'$ to obtain $\lambda$ adds $\Z_2 \oplus \Z_2$ to the cokernel, but does not affect the odd primary part. Therefore $n_p(\coker(\widehat{\lambda})) = n_p(\coker(\widehat{\lambda'})) \leq h$, so~$n_p(\coker(\widehat{\lambda})) + 2 \leq h+2 = \operatorname{rk}(V)$, as claimed.  The second condition of Theorem~\ref{thm:Nikulin} holds since~$\theta$ splits off a $2H_+(\Z)$ summand.
By Corollary~\ref{cor:CS} we therefore have that $\ell_5((\Z^h,2\theta_{\wt{M}_i}^{\nd}) \oplus 2H_+(\Z))$ is trivial.

Theorem~\ref{thm:Recap} now implies that $f$ extends to a rel.\ boundary homeomorphism~$X_{F_0 \# T^2} \cong X_{F_1 \# T^2}$.
The $\nu$-extendability condition implies that $F_0 \# T^2$ and $F_1 \# T^2$ are equivalent rel.\ boundary, which again using the Alexander trick implies they are ambiently isotopic rel.\ boundary.
\end{proof}


Finally, we give the proof of Theorem~\ref{thm:isotopy-of-embeddings} from the introduction, whose statement we recall now.

\begin{customthm}
{\ref{thm:isotopy-of-embeddings}}
\label{thm:isotopy-of-embeddings-body}
 Suppose that for $i=0,1$, the images $F_i := g_i(F)$
  satisfy the hypotheses of one of Theorems~\ref{thm:TopUnknottingIntro},~\ref{thm:SurfacesWithBoundaryIntro}, or~\ref{thm:one-is-enough} $($in the latter case, we assume that $g_i(F)$ are the stabilised surfaces$)$.
The embeddings~$g_0$  and~$g_1$ are ambiently isotopic rel.~boundary if and only if the homeomorphism
  \[\psi:= g_1 \circ (g_0)^{-1} \colon F_0 \to F_1\]
  induces an isometry of the Guillou-Marin forms.
\end{customthm}

\begin{proof}
Suppose that $g_0$ and $g_1$ are ambiently isotopic rel.~boundary. Then in particular there is a homeomorphism~$\Phi\colon N\to N$ such that $\Phi\circ g_0=g_1\colon F\hookrightarrow N$. The Guillou-Marin form is defined geometrically, using immersed discs in the ambient $N$ (Definition~\ref{def:GuillouMarin}). As $\psi$ is the restriction of the homeomorphism $\Phi$, this immediately implies that $f$ induces an isometry of forms because the geometric definitions match up.

For the converse, assume that $\psi$ induces an isometry of the Guillou-Marin forms.
  In the proofs of Theorems~\ref{thm:TopUnknottingIntro},~\ref{thm:SurfacesWithBoundaryIntro}, or~\ref{thm:one-is-enough}, we constructed ambient isotopies $F_0 \sim F_1$. The construction required that we chose a homeomorphism $F_0 \to F_1$ that preserves the Guillou-Marin form. This choice was made during the proof of Proposition~\ref{prop:UnionSpin} and was made arbitrarily (the existence of such a homeomorphism was shown to follow from the hypothesis that $e(F_0)=e(F_1)$). In the present proof, we may choose this homeomorphism to be $\psi\colon F_0 \to F_1$.
As in Proposition~\ref{prop:UnionSpin}, this homeomorphism becomes part of the construction of a homeomorphism $\ol{\psi} \colon S(\nu_{F_0 \subseteq D^4}) \xrightarrow{\cong} S(\nu_{F_0 \subseteq D^4})$, that is $\nu$-extendable rel.~boundary, and with the property that gluing $X_{F_0}$ and $X_{F_1}$ using a homeomorphism $\partial X_{F_0} \to \partial X_{F_1}$ that restricts to~$\ol{\psi}$  produces a spin union.
The resultant ambient isotopies produced by our proofs of Theorems~\ref{thm:TopUnknottingIntro},~\ref{thm:SurfacesWithBoundaryIntro}, or~\ref{thm:one-is-enough} then induce the chosen~$\psi$. Expanding what this means, we have achieved that $g_0$ and $g_1$ are ambiently isotopic rel.~boundary as desired. Indeed, we have produced an isotopy $G_t \colon N \to N$ with $G_0 = \Id_N$, $G_1|_{F_0} = \psi \colon F_0 \to F_1$. We therefore have that $G_t|_{F_0} \circ g_0 \colon F \hookrightarrow N$ is an isotopy rel.\ boundary from $G_0|_{F_0} \circ g_0 = \Id_N \circ g_0 = g_0$ to $G_1|_{F_0} \circ g_0 = g_1 \circ g_0^{-1} \circ g_0 = g_1$.
Since $G_t$ is an isotopy of $N$, $g_0$ and $g_1$ are ambiently isotopic, as desired.
\end{proof}

\appendix
\section{Quadratic linking forms for $h=2,3$ and extremal normal Euler number}

During our proof of Theorem~\ref{thm:SurfacesWithBoundaryMain} for nonorientable surfaces with nonorientable genus~$h=2,3$ and extremal normal Euler number we used that for the quadratic forms~$(V,\theta)=(\Z^2,\pm \left[ \bsm 2&2 \\ 0&1
\esm \right]) \in Q_+(\Z^2)$ and $(V,\theta)=\left(\Z^3,\left[\pm \bsm
4&4&4 \\ 0&2&2 \\ 0&0&2
\esm \right]\right) \in Q_+(\Z^3)$,  the set~$\bAut(V,2\theta)$ is trivial.
More specifically, this was asserted in the proof of Proposition~\ref{prop:ApplyNikulin}.
This section proves these facts.
As we mentioned above Corollary~\ref{cor:CS} (see specifically the diagram in~\eqref{eq:AutVtheta=AutVlambda}),  for a quadratic form $(V,\theta) \in Q_+(V)$ over $\Z$ with symmetrisation $(V,\lambda)$ we have $\Aut(V,\theta)=\Aut(V,\lambda)$.
We will use Sage to calculate~$\Aut(V,\lambda)$ with respect to a particular basis and then show that~$\Aut(V,\lambda) \to \Aut(\partial (V,\theta))$ is surjective.

\subsection{Extremal normal Euler number calculation for $h=2$. }


We prove that for the quadratic form~$(V,\theta)=(\Z^2,\left[  \bsm 2&2 \\ 0&1
\esm \right]) \in Q_+(\Z^2)=:Q$, the set~$\bAut(V,2\theta)$ is trivial.

\begin{lemma}
\label{lem:h=2Step1}
The quadratic form  $\psi:=[ \bsm
4&4 \\ 0&2
\esm] \in Q_+(\Z^2)$ is represented by
$$ \widetilde{Q}=  \begin{pmatrix}
2&0 \\ 0&2
\end{pmatrix}.$$
The cokernel of the adjoint $\widehat{\lambda}$ of the symmetrisation $\lambda$ of $\psi$ is isomorphic to $(\Z_4)^2$ and, for $[x],[y] \in (\Z_4)^2$ the boundary symmetric form $\partial \lambda \colon (\Z_4)^2 \mapsto \Q/\Z$ is given by
$$
\partial \lambda([x],[y])= x^T\begin{pmatrix}
\frac{1}{4}&0 \\ 0&\frac{1}{4}
\end{pmatrix}y .
$$
The boundary split quadratic form $\partial \psi \colon (\Z_4)^2 \to \Q/\Z$ is given by
\[\partial \psi
\left(
\left[\smallmatrix
  x_1 \\ x_2
\endsmallmatrix\right]
\right)
= \frac{1}{8}(x_1^2+x_2^2) \in \Q/\Z.\]
\end{lemma}

\begin{proof}
We claim that~$ 2\theta \cong \left[  \bsm
2&0 \\ 0&2
\esm \right] \in Q_+(\Z^2).$
Use~$e_1,e_2$ to denote the canonical basis of~$\Z^2$.
Perform the basis change which replaces~$e_1,e_2$ by~$e_1-e_2,e_2$.
With respect to this basis of~$\Z^2$, the quadratic form~$\psi:=2\theta \cong \left[  \bsm  4&4 \\ 0&2 \esm \right]$ now becomes~$\left[  \bsm 1&-1 \\0 &1 \esm   \bsm 4&4 \\0 &2 \esm  \bsm 1&0 \\-1 &1 \esm \right] = \left[  \bsm 2&2 \\-2 &2 \esm \right]$.
The claim now follows because~$\left[ \bsm 2&2 \\-2 &2 \esm  \right] = \left[ \bsm 2&0 \\0 &2 \esm \right] \in Q_+(\Z^2)$.

We now calculate the boundary split quadratic linking form.
The symmetrisation of~$\psi$ is represented by~$A= \bsm 4&0 \\ 0&4  \esm$.
It follows that the boundary linking form~$\partial \lambda$  is defined on~$\operatorname{coker}(\widehat{\lambda})=(\Z_4)^2$, given by~$\partial \lambda ([x],[y])=x^TA^{-1}y$ where~$A^{-1}= \bsm \frac{1}{4}&0 \\ 0&\frac{1}{4}  \esm$.
As in Remark~\ref{rem:Matrix}, the boundary split quadratic linking form is~$\partial \psi([x]) = x^T A^{-1}\widetilde{Q}A^{-1}x \in \Q/\Z$ where~$[x] \in (\Z_4)^2$.
In particular,  evaluating on the canonical basis elements of~$(\Z_4)^2$, we see that~$\partial \psi$ is determined by the pair of elements~$(\frac{1}{8},\frac{1}{8})$ of~$\Q/\Z$.
\end{proof}

In the following proposition, we write automorphisms of $(\Z_4)^2$ as matrices with coefficients in~$\Z_4$.

\begin{lemma}
\label{lem:8}
The automorphism group of the quadratic form  $\psi:=\left[ \bsm
4&4 \\ 0&2
\esm \right]\in Q_+(\Z^2)$ contains $8$ elements.
More precisely,  in the basis of Lemma~\ref{lem:h=2Step1}, we have a bijective correspondence
\begin{align*}
\Aut(\Z^2,\psi)
&
\xrightarrow{1:1}
\bigg\lbrace
 \begin{pmatrix}
\varepsilon_1&0 \\
0&\varepsilon_2
\end{pmatrix},
 \begin{pmatrix}
0&\varepsilon_1  \\
\varepsilon_2&0
\end{pmatrix}
\, \bigg| \, \varepsilon_1,\varepsilon_2 \in \lbrace \pm 1 \rbrace
 \bigg\rbrace \subseteq GL_2(\Z).
 \end{align*}
 With respect to this same basis, the image of $\partial \colon \Aut(\Z^2,\psi) \to \Aut(\partial (\Z^2,\psi))$  is
\begin{align*}
\im(\partial)=\bigg\lbrace
 \begin{pmatrix}
\varepsilon_1&0 \\
0&\varepsilon_2
\end{pmatrix},
 \begin{pmatrix}
0&\varepsilon_1  \\
\varepsilon_2&0
\end{pmatrix}
\, \bigg| \, \varepsilon_1,\varepsilon_2 \in \lbrace \pm 1 \rbrace
 \bigg\rbrace \subseteq \operatorname{Aut}((\Z_4)^2).
\end{align*}
\end{lemma}
\begin{proof}
The calculation of $\Aut(\Z^2,\psi)$ follows from the fact, proved in Lemma~\ref{lem:h=2Step1}, that $\psi$ is represented by $\widetilde{Q}= \bsm 2&0 \\ 0&2 \esm$.
The second statement is a consequence of the definition of $\partial$ as $\partial F:=(F^*)^{-1}$.
\end{proof}

\begin{proposition}
\label{prop:bAutTrivialh=2}
For~$(V,\theta)=(\Z^2,\left[  \pm \bsm 2&2 \\ 0&1
\esm \right]) \in Q_+(\Z^2)=:Q$, the set~$\bAut(V,2\theta)$ is trivial.
\end{proposition}
\begin{proof}
It suffices to prove the statement for $\theta=[ \bsm 2&2 \\ 0&1
\esm ]$ because $\bAut(v)=\bAut(-v)$ for every quadratic form~$v$.
An automorphism of~$(\Z_4)^2$ is determined by its values on~$e_1:=\left[ \bsm 1 \\ 0 \esm \right]$ and~$e_2:=\left[ \bsm 0 \\ 1 \esm \right]$.
We immediately rule out some automorphisms that do not preserve $\partial \psi$, where $\psi:=2\theta$.
Given $f \in \Aut(\partial (\Z^2,\psi))$,  write~$f(e_i)=a_i e_1+b_i e_2$ with~$a_i,b_i \in \Z_4$,  note that since~$\partial \lambda(e_1,e_2)=0$, Lemma~\ref{lem:h=2Step1} implies that
$$\smfrac{1}{8}
=\partial \psi(e_i)
=\partial \psi(f(e_i))
=\partial \psi(a_i e_1+b_i e_2)
= a_i^2 \partial \psi(e_1) +b_i^2 \partial \psi(e_2)+\partial \lambda(e_1,e_2)
=\smfrac{a_i^2+b_i^2}{8} \in \Q/\Z.$$
This implies that~$a_i^2+b_i^2 \equiv 1  \text{ mod } 8  \text{ with } a_i,b_i \in \Z_4$ for~$i=1,2$  and thus~$ \lbrace a_i,b_i \rbrace=\lbrace 0,\pm 1 \rbrace.$
Using Lemma~\ref{lem:8}, this implies that~$\im(\partial)=\Aut(\partial (\Z^2,\psi))$.
We deduce that~$\bAut(\Z^2,\psi)$ is trivial, as claimed.
\end{proof}

\subsection{Extremal normal Euler number calculation for $h=3$. }


We prove that for the quadratic form~$(V,\theta)=\left(\Z^3, \left[ \pm \bsm 2&2&2 \\ 0&1&1 \\ 0&0&1 \esm\right]\right) \in Q_+(\Z^3)$, the set~$\bAut(V,2\theta)$ is trivial.

\begin{lemma}
\label{lem:BoundaryDataForh=3}
Set  $\psi:= \left[\bsm
4&4&4 \\ 0&2&2 \\ 0&0&2
\esm \right] \in Q_+(\Z^3).$
The form $\psi$ is represented by
$$ \widetilde{Q}= \begin{pmatrix}
12&32&-16 \\ 0&22&-22 \\ 0&0&6
\end{pmatrix}.$$
The cokernel of the adjoint $\widehat{\lambda}$ of the symmetrisation $\lambda$ of $\psi$ is isomorphic to $\Z_8 \oplus (\Z_2)^2$ and, for~$[x],[y] \in \Z_8 \oplus (\Z_2)^2$ the boundary symmetric form $\partial \lambda \colon (\Z_8 \oplus (\Z_2)^2) \times (\Z_8 \oplus (\Z_2)^2)  \mapsto \Q/\Z$ is given by
$$
([x],[y])\mapsto  x^T\begin{pmatrix}
\frac{3}{8}&0&0 \\ 0&0&\frac{1}{2} \\ 0&\frac{1}{2}&0
\end{pmatrix}y.
$$
The boundary split quadratic form $\partial \psi \colon \Z_8 \oplus (\Z_2)^2 \to \Q/\Z$ is given on the generators by
$$ \partial \psi
\left(
\left[\smallmatrix
  1 \\ 0  \\ 0
\endsmallmatrix\right]
\right)
=\frac{11}{16},
 \ \ \ \ \
 \partial \psi
\left(
\left[\smallmatrix
  0  \\ 1  \\ 0
\endsmallmatrix\right]
\right)
=\frac{1}{2},   \ \ \ \ \
 \partial \psi
\left(
\left[\smallmatrix
0 \\ 0  \\ 1
\endsmallmatrix\right]
\right)
=\frac{1}{2} .
$$
\end{lemma}
\begin{proof}
Write $Q:= \bsm
4&4&4 \\ 0&2&2 \\ 0&0&2
\esm$ and $A:=Q+Q^T$.
Performing row and column operations on $A$, one gets $\coker(A)=\Z_8 \oplus (\Z_2)^2$ with generators given by the classes of $(1,0,0),(0,1,1)$ and $(2,0,1)$.
With respect to this generating set, the boundary symmetric linking form $\partial \lambda([x],[y])=xA^{-1}y$ is  isometric to the pairing
$$([x],[y]) \mapsto x^T
\begin{pmatrix}
\frac{3}{8}&-\frac{1}{2}&\frac{1}{2}  \\
-\frac{1}{2}&0&-\frac{1}{2}  \\
\frac{1}{2}&-\frac{1}{2}&0
 \end{pmatrix}y  \in \Q/\Z.$$
 We now simplify this form by performing isometries that lift to~$\Z^3$.
 Namely, we set~$B:=\bsm 1&0&2 \\ 0&1&0 \\ 0&1&1 \esm$
and $C:=\bsm 1&0&0 \\ -1&1&0 \\ -1&1&1 \esm$.
Write $D:= BC$.
 The linking form $\partial \lambda$  is isometric to the pairing
$$([x],[y]) \mapsto x^TD^T A^{-1}Dy
=x^T\begin{pmatrix}
\frac{3}{8}&0&0 \\ 0&0&\frac{1}{2} \\ 0&\frac{1}{2}&0
\end{pmatrix}y.$$
The matrix~$\widetilde{Q}$ is obtained by using the defining relation  of $Q_+(\Z^3)$:
$$M:=D^{-1}Q D^{-T}
=\begin{pmatrix}
12&20&-12\\
12&22&-14 \\
-4&-8&6
\end{pmatrix}
\sim
\begin{pmatrix}
12&32&-16 \\ 0&22&-22 \\ 0&0&6
\end{pmatrix}
=\widetilde{Q}.
 $$
The statement about the boundary symmetric linking form now follows by construction, so we conclude by studying the boundary quadratic linking form.
The symmetrisation of $\widetilde{Q}$ is~$G:=D^{-1}AD^{-T}$.
As noted in Remark~\ref{rem:Matrix}, the values for the boundary quadratic linking form on the generators follow by looking at the diagonal entries of the matrix $G^{-1}MG^{-1}=D^T(A^{-1}QA^{-1})D$.
These diagonal entries are precisely $\frac{11}{16},\frac{1}{2}$ and $\frac{1}{2}$.
\end{proof}

As explained in~\cite[Theorem 3.6]{HillarRhea}, there is a canonical isomorphism
\begin{equation}
\label{eq:AutZ8Z2Z2}
\Aut(\Z_8 \oplus \Z_2 \oplus \Z_2) \cong
\frac{
\bigg\lbrace
A=
\bsm a&4b&4c \\ d&e&f \\ g&h&i \esm  \in M_3(\Z)  \, \bigg| \,   A \text{ mod } 2 \in GL_3(\Z_2) \bigg\rbrace
}{
 \bigg\lbrace  \bsm 8a & 8b & 8c \\  2d&2e&2f \\  2g&2h&2i  \esm \, \bigg| \,  a,b,c,d,e,f,g,h,i \in \Z  \bigg\rbrace
}.
\end{equation}
From now on we use this isomorphism implicitly.

\begin{lemma}
\label{lem:192}
Set  $\psi:= \left[\bsm
4&4&4 \\ 0&2&2 \\ 0&0&2
\esm \right] \in Q_+(\Z^3).$
The group $\Aut(\psi)$ has $48$ elements.
More precisely, using the basis from Lemma~\ref{lem:BoundaryDataForh=3},  the image of $\partial \colon \Aut(\Z^3,\psi) \to \Aut(\partial (\Z^3,\psi))$ decomposes as $\im(\partial) =X_0 \sqcup Y_0$ where
$$
X_0= \left\lbrace
\bsm
\pm 1&0&0\\
0&1&0 \\
0&0&1
\esm,
\bsm
\pm 1&0&0\\
0&1&1 \\
0&0&1
\esm,
\bsm
\pm 1&0&0\\
0&0&1 \\
0&1&0
\esm,
\bsm
\pm 1&0&0\\
0&0&1 \\
0&1&1
\esm,
\bsm
\pm 1&0&0\\
0&1&1 \\
0&1&0
\esm,
\bsm
\pm 1&0&0\\
0&1&0 \\
0&1&1
\esm
 \right\rbrace \\
$$
has $12$ elements and $Y_0$ has $36=12 \cdot 3$ elements:
\begin{align*}
Y_0=  \Biggl\lbrace
&\bsm
\pm 3&4&0\\
0&1&0 \\
1&0&1
\esm,
\bsm
\pm 3&4&4\\
0&1&1 \\
1&0&1
\esm,
\bsm
\pm 3 &0&4\\
0&0&1 \\
1&1&0
\esm,
\bsm
\pm 3&0&4\\
0&0&1 \\
1&1&1
\esm,
\bsm
\pm 3 &4&4\\
0&1&1 \\
1&1&0
\esm,
\bsm
\pm 3&4&0\\
0&1&0 \\
1&1&1
\esm, \\
&
\bsm
\pm 3&0&4\\
1&1&0 \\
0&0&1
\esm,
\bsm
\pm 3&0&4\\
1&1&1 \\
0&0&1
\esm,
\bsm
\pm 3 &4&0\\
1&0&1 \\
0&1&0
\esm,
\bsm
\pm 3&4&4\\
1&0&1 \\
0&1&1
\esm,
\bsm
\pm 3 &4&0\\
1&1&1 \\
0&1&0
\esm,
\bsm
\pm 3&4&4\\
1&1&0 \\
0&1&1
\esm,
 \\
&
\bsm
\pm 3&4&4\\
1&1&0 \\
1&0&1
\esm,
\bsm
\pm 3&4&0\\
1&1&1 \\
1&0&1
\esm,
\bsm
\pm 3 &4&4\\
1&0&1 \\
1&1&0
\esm,
\bsm
\pm 3&4&0\\
1&0&1 \\
1&1&1
\esm,
\bsm
\pm 3 &0&4\\
1&1&1 \\
1&1&0
\esm,
\bsm
\pm 3&0&4\\
1&1&0 \\
1&1&1
\esm
 \Biggr\rbrace.
\end{align*}
\end{lemma}
\begin{proof}
Write $Q:=\bsm
4&4&4 \\ 0&2&2 \\ 0&0&2
\esm$ so that the symmetrisation $\lambda$ of $\psi$ is represented by $A=Q+Q^T$.
We used Sage\footnote{Specifically we entered $Q = \text{QuadraticForm}(\text{ZZ}, 3,[12,32,-16,22,-22,6])$ and then $Q.\text{automorphisms}()$ to list the automorphisms.
The relevant documentation can be found here \url{https://doc.sagemath.org/html/en/reference/quadratic_forms/sage/quadratic_forms/quadratic_form.html}.
}
 to list all the elements of $\lbrace M \mid M^TAM=A \rbrace$, which is in bijective correspondence with
$\Aut(\Z^3,\psi)=\Aut(V,\lambda)$.
Taking the inverse transpose of these matrices and reducing the coefficients according to~\eqref{eq:AutZ8Z2Z2} leads to the claimed determination of~$\im(\partial)$.
\end{proof}

\begin{proposition}
\label{prop:bAutTrivialh=3}
For~$(V,\theta)=\left(\Z^3, \left[ \pm \bsm 4&4&4 \\ 0&2&2 \\ 0&0&2 \esm\right]\right) \in Q_+(\Z^3)$, the set~$\bAut(V,2\theta)$ is trivial.
\end{proposition}
\begin{proof}
It suffices to prove the statement for $\theta=\left[ \bsm 4&4&4 \\ 0&2&2 \\ 0&0&2\esm \right]$ because $\bAut(v)=\bAut(-v)$ for every quadratic form~$v$.
Write $\psi:=2\theta$ and $\lambda$ for the symmetrisation of $\psi$.
Recall that $\partial \psi$ (resp. $\partial \lambda$) denotes the boundary quadratic linking form (resp.\ boundary symmetric linking form) of $\psi$ and~$\lambda$.
Recall also that $\partial \psi(x+y)=\partial \psi(x)+\partial \psi(y)+\partial \lambda(x,y)$ and~$\partial \psi(rx)=r^2\partial \psi(x)$.
Using these facts as well as the calculation of $\partial \lambda$ from Lemma~\ref{lem:BoundaryDataForh=3}, we see that for $(a,b,c) \in \Z_8 \oplus \Z_2 \oplus \Z_2$ we have
\begin{equation}
\label{eq:nuabc}
\partial \psi(a,b,c)
= \smfrac{11}{16}a^2+\frac{1}{2}(b^2+c^2)+\smfrac{1}{2}bc
=\smfrac{11a^2+8(b^2+c^2+bc)}{16}.
\end{equation}
Now we  begin our study of the automorphisms of $\partial \psi$.
Such an automorphism $f$ is determined by its images on the canonical generators of $\Z_8 \oplus \Z_2 \oplus \Z_2$.
\begin{itemize}
\item For $(a,b,c):=f(1,0,0)$ with $a \in \Z_8,b,c \in \Z_2$,  the calculation in~\eqref{eq:nuabc} gives
\[ \smfrac{11}{16}
=\partial \psi(1,0,0)
=\partial \psi(f(1,0,0))
=\partial \psi(a,b,c)
=\smfrac{11a^2+8(b^2+c^2+bc)}{16} \in \Q/\Z.\]
In other words we have
$$ 11 \equiv 11a^2+8(b^2+c^2+bc) \text{ mod } 16.$$
We also know that $f(1,0,0)$ must have order $8$.
It follows that $a=\pm 1$ or $\pm 3$.
If $a=\pm 1$ then we have the equation $8(b^2+c^2+bc)\equiv 0$ mod $16$ which implies $b=c=0$.
If $a=\pm 3$ then we have the equation $8(b^2+c^2+bc)\equiv 8$ mod $16$ which implies $(b,c) \neq (0,0)$.
Putting this all together, we deduce that
$$f\bsm
1 \\ 0 \\ 0
\esm
 \in
 \bigg\lbrace
\bsm
\pm 1 \\ 0 \\ 0
\esm
  \bigg\rbrace
  \sqcup
 \bigg\lbrace
\bsm
\pm 3 \\ x \\ y
\esm
\, \Big| \,
\bsm
x \\ y
\esm
 \in
\Big\lbrace
\bsm
1 \\ 0
\esm
,
\bsm
0 \\ 1
\esm
,
\bsm
1 \\ 1
\esm
 \Big\rbrace
  \bigg\rbrace.
$$
\item For $(a,b,c):=f(0,1,0)$ with $a \in \Z_8,b,c \in \Z_2$,  the same reasoning as above (but this time using $\partial \psi(0,1,0)=\frac{1}{2}$) leads to the equation
$$ 8 \equiv  11a^2+8(b^2+c^2+bc) \text{ mod } 16.$$
Since $f(0,1,0)$ must have order $2$ we obtain that $a \in \lbrace 0,4\rbrace$.
Since $a^2 \equiv  0$ mod $16$ for both~$a=0$ and $a=4$,  we obtain the equation
$$ 8 \equiv  8(b^2+c^2+bc) \text{ mod } 16.$$
This implies that $(b,c) \neq (0,0)$.
Putting this all together, we deduce that
$$
f\bsm
0 \\ 1 \\ 0
\esm \in
\bigg\lbrace
\bsm
x \\ 1 \\ 0
\esm
,
\bsm
x \\ 0 \\ 1
\esm
,
\bsm
x \\ 1 \\ 1
\esm
\, \bigg| \,
x \in \lbrace 0,4 \rbrace
\bigg\rbrace.
$$
\item The exact same reasoning shows that
$$
f\bsm
0 \\ 0 \\ 1
\esm
 \in
\bigg\lbrace
\bsm
x \\ 1 \\ 0
\esm
,
\bsm
x \\ 0 \\ 1
\esm
,
\bsm
x \\ 1 \\ 1
\esm
\, \bigg| \,
x \in \lbrace 0,4 \rbrace
\bigg\rbrace.
$$
\end{itemize}
Combining these calculations,  we have proved that
$$\im(\partial) \subseteq  \Aut(\partial (\Z^3,\psi)) \subseteq X \sqcup Y \subseteq \Aut(\Z_8 \oplus \Z_2 \oplus \Z_2) $$
where $X$ is the following set with $(2 \cdot 2 \cdot 2) \cdot 6=48$ elements
$$
X= \bigg\lbrace
\bsm
\pm 1&b&c\\
0&1&0 \\
0&0&1
\esm,
\bsm
\pm 1&b&c\\
0&1&1 \\
0&0&1
\esm,
\bsm
\pm 1&b&c\\
0&0&1 \\
0&1&0
\esm,
\bsm
\pm 1&b&c\\
0&0&1 \\
0&1&1
\esm,
\bsm
\pm 1&b&c\\
0&1&1 \\
0&1&0
\esm,
\bsm
\pm 1&b&c\\
0&1&0 \\
0&1&1
\esm
\, \bigg| \,
b,c \in \lbrace 0,4 \rbrace
 \bigg\rbrace, \\
$$
and $Y$ is the following set with $(3 \cdot 2 \cdot 2 \cdot 2)\cdot 6=144$ elements
\begin{multline}
Y=
 \Biggl\lbrace
\bsm
\pm 3&b&c\\
d&1&0 \\
g&0&1
\esm,
\bsm
\pm 3&b&c\\
d&1&1 \\
g&0&1
\esm,
\bsm
\pm 3 &b&c\\
d&0&1 \\
g&1&0
\esm, \\
\bsm
\pm 3&b&c\\
d&0&1 \\
g&1&1
\esm,
\bsm
\pm 3&b&c\\
d&1&1 \\
g&1&0
\esm,
\bsm
\pm 3&b&c\\
d&1&0 \\
g&1&1
\esm
\, \bigg| \,
b,c \in \lbrace 0,4 \rbrace,
\bsm
d \\ g
\esm
\in
\Big\lbrace
\bsm
0 \\ 1
\esm,
\bsm
1 \\ 0
\esm,
\bsm
1 \\ 1
\esm
\Big\rbrace
\Biggr\rbrace.
\end{multline}
We are justified in writing these automorphisms of $\Z_8 \oplus \Z_2 \oplus \Z_2$ as matrices because the entries~$b$ and~$c$ are divisible by $4$ and because the reductions mod $2$ have nonzero determinant; see~\eqref{eq:AutZ8Z2Z2}.

Using the basis from Lemma~\ref{lem:BoundaryDataForh=3}, elements of $\Aut(\partial (\Z^3,\psi))$ must additionally preserve the linking form
$$([x],[y]) \mapsto x^T \bsm
\frac{3}{8}&0&0 \\ 0&0&\frac{1}{2} \\ 0&\frac{1}{2}&0
\esm y.$$
Tedious but explicit matrix calculations (available upon request) show that of the $192$ elements of~$X \sqcup Y$ only those in~$\im(\partial)$ (as listed in Lemma~\ref{lem:192}) preserve the linking form.
For example
$$ \bsm 3&4&0 \\ 1&1&0 \\ 1& 1&1  \esm^T
\bsm
\frac{3}{8}&0&0 \\ 0&0&\frac{1}{2} \\ 0&\frac{1}{2}&0
\esm
 \bsm 3&4&0 \\ 1&1&0 \\ 1& 1&1   \esm
\equiv
\bsm
\frac{3}{8}&\frac{1}{2}&\frac{1}{2} \\
\frac{1}{2}&0&\frac{1}{2} \\
\frac{1}{2}&\frac{1}{2}&0
\esm
\neq
\bsm
\frac{3}{8}&0&0 \\ 0&0&\frac{1}{2} \\ 0&\frac{1}{2}&0
\esm
 \in M_3(\Q/\Z).
 $$
We conclude that $\im(\partial)=\Aut(\partial (\Z^3,\psi))$ and this concludes the proof of the proposition.
\end{proof}

\def\MR#1{}
\bibliography{biblioPi1Z2}
\end{document}